\documentclass[12pt]{amsart}

\usepackage[left=2.2cm,tmargin=2.5cm,bmargin=2.5cm,right=2.2cm]{geometry}

\usepackage{amsmath, amssymb, amsthm}

\usepackage{hyperref}
\usepackage{tikz-cd}
\usepackage[capitalise]{cleveref}

\usepackage{comment}

\newtheorem{proposition}{Proposition}[section]
\newtheorem{theorem}[proposition]{Theorem}
\newtheorem{corollary}[proposition]{Corollary}

\newtheorem{lemma}[proposition]{Lemma}

\newtheorem{conjecture}[proposition]{Conjecture}

\theoremstyle{definition}
\newtheorem{definition}[proposition]{Definition}

\renewcommand{\det}{\operatorname{det}}
\newcommand{\tr}{\operatorname{tr}}
\newcommand{\std}{\operatorname{std}}
\newcommand{\sgn}{\operatorname{sgn}}
\newcommand{\ex}{\operatorname{exp}}
\newcommand{\weight}{\operatorname{weight}}

\newcommand{\eigen}{\lambda}

\newcommand{\rp}{r}
\newcommand{\ry}{\tilde{r}}
\newcommand{\Cqs}{\mathbb C[[q^{-s}]]^+}
\usepackage{physics}
\begin{document} 

\begin{abstract} We propose a refinement of the random matrix model for a certain family of $L$-functions over $\mathbb F_q[u]$, using techniques that we hope will eventually apply to an arbitrary family of $L$-functions. This consists of a probability distribution on power series in $q^{-s}$ which combines properties of the characteristic polynomials of Haar-random unitary matrices and random Euler products over $\mathbb F_q[u]$. The support of our distribution is contained in the intersection of the supports of the two original distributions. The expectations of low-degree polynomials in the coefficients of our series approximate the expectations of the same polynomials in the coefficients of random Euler products, while the expectations of high-degree polynomials approximate the expectations of the same polynomials in the coefficients of the characteristic polynomials of random matrices. Furthermore, the expectations of absolute powers of our series approximate the \cite{CFKRS}-\cite{AK} prediction for the moments of our family of $L$-functions. \end{abstract}

\title{A refined random matrix model for function field $L$-functions}

\author{Will Sawin}

\maketitle 

\section{Introduction}

We begin by defining two probability distributions: one describing uniformly random multiplicative functions and their associated Euler products, and the other uniform random matrices and their characteristic polynomials. We next construct a probability distribution as a hybrid of both, which describes non-uniform random matrices. We then state our results about this hybrid distribution, and explain how it can be used to model the behavior of a certain family of Dirichlet $L$-functions.

Let $\mathbb F_q[u]^+$ be the set of monic polynomials in one variable over a finite field $\mathbb F_q$. We say a polynomial in $\mathbb F_q[u]^+$ is prime if it is irreducible. Take for each prime $\pi\in \mathbb F_q[u]^+$ an independent random variable $\xi(\pi)$ uniformly distributed on the unit circle in $\mathbb C$ and form the random Euler product
\[L_\xi(s)= \prod_{ \substack{\mathfrak p \in \mathbb F_q[u]^+ \\ \textrm{prime}} }\frac{1}{1- \xi(\mathfrak p) q^{-s \deg \mathfrak p}} .\]
We can extend $\xi$ uniquely to a function $\xi \colon \mathbb F_q[u]^+ \to \mathbb C$ that is completely multiplicative in the sense that $\xi(1)=1$, $\xi(fg)=\xi(f)\xi(g)$ for all $f,g\in \mathbb F_q[u]^+$. In other words, $\xi$ is a Steinhaus random multiplicative function. Then we can equally well express $L_\xi(s)$ as a sum $\sum_{f \in \mathbb F_q[u]^+}\xi(f) q^{- s \deg f} $. We have
\[ \log  L_\xi(s)=\sum_{ \substack{\mathfrak p \in \mathbb F_q[u]^+ \\ \textrm{prime}} } \log  \frac{1}{1- \xi(\mathfrak p) q^{-s \deg \mathfrak p}} = \sum_{ \substack{\mathfrak p \in \mathbb F_q[u]^+ \\ \textrm{prime}} }  \sum_{m=1}^{\infty}  \frac{ \xi(\mathfrak p)^m q^{ - s m \deg \mathfrak p}}{m}   .\]
Let $X_{n ,\xi}$ be the coefficient of $q^{- ns}$ in $ \log L_\xi(s) $, i.e.  \begin{equation}\label{Xn-formula} X_{n,\xi} = \sum_{d \mid n} \sum_{ \substack{\mathfrak p \in \mathbb F_q[u]^+ \\ \deg \mathfrak p =d  \\ \textrm{prime}} }  \frac{d}{n} \xi(\mathfrak p)^{ \frac{n}{d}} . \end{equation}

Fix $k$ a natural number. Assume $q>2$.  Let $F(x_1,\dots, x_k)$ be the probability density function of the tuple of random variables $X_{1,\xi},\dots, X_{k,\xi}$. (We assume $q>2$ since it is not hard to check that this probability density function does not exist for $q=2$ as long as $k\geq 2$.)

Let $\Cqs$ be the set of power series in $q^{-s}$ with constant coefficient $1$. The power series $L_\xi(s)$ lies in $\Cqs$. We endow $\Cqs$ with a topology by viewing it as a product of copies of $\mathbb C$ and taking the product topology, and consider the Borel $\Sigma$-algebra. Let $\mu_{\textrm{ep}}$ be the measure on $\Cqs$ given by the distribution of the random variable $L_\xi$.

For $N$ a natural number and $M$ in the unitary group $U(N)$, let \begin{equation}\label{LM-def} L_M(s) = \det ( I - q^{ \frac{1}{2} - s } M). \end{equation} We have $L_M\in  \Cqs.$ (In fact, $L_M$ is a polynomial in $q^{-s}$ and not just a power series.) Let $\mu_{\textrm{rm}}$ be the measure on $\Cqs$ given by the distribution of the random variable $L_M$ for $M$ Haar-random in $U(N)$. In other words, $\mu_{\textrm{rm}}$ is the pushforward of the Haar measure $\mu_{\textrm{Haar}}$ from $U(N)$ to $\Cqs$.

The probability distributions $\mu_{\textrm{ep}}$ and $\mu_{\textrm{rm}}$ can both be used as models for properties of random Dirichlet $L$-functions over $\mathbb F_q[u]$. We now describe a distribution that combines properties of both and thus, we hope, serves as a better model than either.

Fix $\beta \in \left( \frac{1}{4},  \frac{1}{2}\right)$ and $N$ a natural number and let $k= \lfloor N^\beta \rfloor$.

Consider the non-uniform measure on $U(N)$
\begin{equation}\label{weighted-def} \mu_{\textrm{weighted} } = \gamma \frac{ F( - q^{1/2} \tr(M),\dots, -q^{k/2} \tr (M^k)/k)}{  \prod_{j=1}^k  \left( e^{  -  \frac{  \abs{ \tr (M^j)}^2 }{ j } }  \frac{ j }{ q^j \pi} \right) } \mu_{\textrm{Haar} } \end{equation} where $\gamma $ is the unique constant that makes $\mu_{\textrm{weighted}}$ a probability measure.  We let $\mu_{\textrm{ch}}$ be the distribution of $ L_M $ for $M$ a matrix in $U(N)$ distributed according to $\mu_{\textrm{weighted} } $, i.e. the pushforward of $\mu_{\textrm{weighted} } $ from $U(N)$ to $\Cqs$.

We think of $\mu_{\textrm{ch}}$ as a chimera in the sense of a strange hybrid of the more familiar creatures $\mu_{\textrm{ep}}$ and $\mu_{\textrm{rm}}$. The measure $\mu_{\textrm{ch}}$ exists to serve as a model of a function field analogue of the Riemann zeta function, and hence combines the Euler product and random matrix perspectives on the zeta function. (We will meet the exact function field analogue which $\mu_{\textrm{ch}}$ models shortly in \S\ref{L-function-example}.) We prove three fundamental results about $\mu_{\textrm{ch}}$ describing the support of the measure and its integrals of the measure against a general class of test functions. These results explain which properties $\mu_{\textrm{ch}}$ shares with each of the simpler measures $\mu_{\textrm{ep}}$ and $\mu_{\textrm{rm}}$. Combining these, we will in Theorem \ref{intro-cfkrs} evaluate the integral against a specific test function that models the moments of the zeta function, and the resulting formula will look similar to predictions from \cite{CFKRS} for the moments of the zeta function.

Note that all our measures depend implicitly on the parameters $q,N$ (except that $\mu_{\textrm{ep}}$ is independent of $N$), so it should not be surprising when error terms in our estimates for them depend on $q$ or $N$. All implicit constants in big $O$ notation will be independent of $q,N$ except where noted.

\begin{proposition}\label{intro-support} For $N $ larger than some absolute constant, the support of $\mu_{\textrm{ch}}$ is contained in the intersection of the supports of $\mu_{\textrm{ep}}$ and $\mu_{\textrm{rm}}$. \end{proposition}

Natural test functions to use on $\Cqs$ are polynomials in the coefficients of the power series and their complex conjugates, i.e. we consider elements of the polynomial ring $\mathbb C[ c_1,c_2,\dots, \overline{c_1}, \overline{c_2},\dots]$ as functions on an element $1+\sum_{d=1}^{\infty} c_d q^{-ds}$ of $\Cqs$. We define the (weighted) degree of a polynomial in  $\mathbb C[ c_1,c_2,\dots, \overline{c_1}, \overline{c_2},\dots]$ by letting $c_d$ and $\overline{c_d}$ have degree $d$ for all $d$. We define the $L^2$ norm of such a polynomial using the random matrix measure, as 
\[ \norm{\phi}_2^2=\int_{\Cqs }  \abs{ \phi}^2   \mu_{\textrm{rm} }   =  \int_{U(N)}  \abs{ \phi ( L_M) }^2   \mu_{\textrm{Haar} } .\]

\begin{theorem}\label{intro-lf} Assume that $q>5$. Let $\phi \in \mathbb C[ c_1,c_2,\dots, \overline{c_1}, \overline{c_2},\dots] $ have degree $\leq k$. For $N$ sufficiently large in terms of $\beta$, we have
\begin{equation}\label{eq-intro-lf} \int_{\Cqs }  \phi    \mu_{\textrm{ch}}  =  \int_{\Cqs }    \phi  \mu_{\textrm{ep}} + O (e^{ -(\frac{1}{2} - o_N(1)) N^{1-\beta} \log (N^{1-\beta})} \norm{\phi}_2 ) .\end{equation}
\end{theorem}

\begin{theorem}\label{intro-hf} Assume that $q>11$. Let $\phi \in \mathbb C[ c_1,c_2,\dots, \overline{c_1}, \overline{c_2},\dots] $. Assume that for all polynomials $\psi\in \mathbb C[ c_1,c_2,\dots, \overline{c_1}, \overline{c_2},\dots]  $ of degree $\leq k$ we have
\begin{equation}\label{hf-hypothesis} \int_{U(N)} \phi(L_M) \overline{\psi(L_M)} \mu_{\textrm{Haar}} =0 .\end{equation}
Then for $N$ sufficiently large in terms of $\beta$, we have
\begin{equation}\label{eq-intro-hf} \int_{\Cqs}   \phi    \mu_{\textrm{ch}}  =O_q ( N^{-  \beta \frac{q-2}{4}} \norm{\phi}_2   ) .\end{equation}
\end{theorem}

In interpreting \cref{intro-lf,intro-hf}, it is helpful to consider the heuristic that the typical size of $\phi$ on the support of $\mu_{\textrm{rm}}$ is approximately $\norm{\phi}_2$, and therefore, we should expect a trivial bound for the integral $\int_{\Cqs}   \phi    \mu_{\textrm{ch}}$ to be of size roughly $\norm{\phi}_2$. From this point of view we can view the factors of $e^{- (\frac{1}{2} - o_N(1)) N^{1-\beta} \log (N^{1-\beta})}$ in \cref{intro-lf} and $ N^{- \beta \frac{q-2}{4}} $ in \cref{intro-hf} as the amount of savings over the trivial bound, although obtaining an error term of size $O( \norm{\phi}_2)$ is not completely trivial.

\subsection{A family of $L$-functions and their moments}\label{L-function-example}

We next explain how $\mu_{\textrm{ch}}$ can be used as a model for a certain family of $L$-functions. We first consider a family of characters (discussed in more detail in \cite{s-rep}). 

\begin{definition} We say a Dirichlet character $\nu \colon \left( \mathbb F_q[x]/x^{N+2} \right)^\times \to \mathbb C^{\times}$ is ``primitive" if $\nu$ is nontrivial on elements congruent to $1$ mod $x^{N+1}$, and ``even" if $\nu$ is trivial on $\mathbb F_q^\times$. For a Dirichlet character $\nu$, define a function $\chi$ on monic polynomials in $\mathbb F_q[u]$ by, for $f$ monic of degree $d$, \[ \chi(f) = \nu ( f(x^{-1}) x^{d} ) .\] It is easy to see that $\chi$ depends only on the $N+2$ leading terms of $f$. Let $S_{N+1,q}$ be the set of characters $\chi$ arising from primitive even Dirichlet characters $\nu$ in this way. Because there are $q^{N+1} $ even Dirichlet characters of which $q^{N}$ are imprimitive, $S_{N+1,q}$	 has cardinality $q^{N+1}- q^{N}$. 

For $\chi\in S_{N+1,q}$, form  the associated $L$-function \[L(s,\chi) = \sum_{ f \in \mathbb F_q[u]^+ } \chi(f)  \abs{f}^{-s}.\] 
\end{definition}

The family of $L$-functions we consider consists of, for $\chi \in S_{N+1,q}$ and $t\in [ 0,\frac{2\pi }{\log q} ] $, the function $L(s+it,\chi)$. A random $L$-function of this family is obtained by choosing $\chi$ and $t$ independently uniformly at random. Let us see why these $L$-functions form a reasonable model for the statistics of the Riemann zeta function. 

$L(s+it,\chi)$ is the Dirichlet series with coefficients $f \mapsto \chi(f) \abs{f}^{-it}$. Viewed as characters of the idele class group of $\mathbb F_q(u)$, these comprise all the unitary characters ramified only at $\infty$ with conductor exponent $N+2$ at $\infty$. They are thus comparable to the characters $n \to n^{it}$ of $\mathbb N$, which are the unitary characters of the idele class group of $\mathbb Q$ ramified only at $\infty$. Said in a more elementary fashion, $n^{it}$ for $\abs{t}\leq T$ may be accurately approximated given the leading $\approx \log T$ digits of $n$ as well as the total number of digits, and all multiplicative functions with this approximation property have the form $n^{it}$, while $\chi(f) \abs{f}^{-it}$ may be computed exactly given the leading $N+2$ coefficients of $f$ as well as the degree of $f$, and all multiplicative functions with this property (or even those that may be approximated given this information) are of the form $\chi(f) \abs{f}^{-it}$ for some $\chi\in S_{N' +1,q}$ for some $N' \leq N$. Thus the statistics of $L(s+it, \chi)$ for random $\chi \in S_{N+1,q}, t\in [0, \frac{2\pi}{\log q}] $ are comparable to the statistics of $\zeta(s+it) $ for random $t\in [T,2T]$, i.e. the local statistics of the Riemann zeta function on the critical line.

The distribution of the coefficients  $f \mapsto \chi(f) \abs{f}^{-it}$ converges in the large $N$ limit to the distribution of a random multiplicative function $\xi$. (Without the average over $t$, they would converge to random multiplicative function subject to the restriction $\xi(u)=1$.)  On the other hand, by work of Katz~\cite{katz-wvqkr}, in the large $q$ limit the distribution of the $L$-functions $L(s+it,\chi)$ converges to the distribution $\mu_{\textrm{rm}}$, as long as $N \geq 3$. (Technically, we must express our power series in the variable $q^{\frac{1}{2}-s}$ instead of $q^{-s}$ for this convergence to make sense, as otherwise $\mu_{\textrm{rm}}$ depends on $q$.)  More precisely, \cite[Theorem 1.2]{katz-wvqkr} proves equidistribution of conjugacy classes in $PU(N)$ whose characteristic polynomials correspond to $L(s,\chi)$ against the Haar measure of $PU(N)$, and the additional averaging over $t$ is equivalent to averaging over the fibers of $U(N)\to PU(N)$.

Because of this $N\to \infty$ and $q\to\infty$ limiting behavior, the distribution of the family of $L$-functions $L(s+it,\chi)$ for finite $q,N$ is expected to have some similarity with $\mu_{\textrm{ep}}$ and some with $\mu_{\textrm{rm}}$. Thus $\mu_{\textrm{ch}}$, which interpolates between $\mu_{\textrm{ep}}$ and $\mu_{\textrm{rm}}$, is a plausible model for the distribution of $L(s+it,\chi)$. To test this model we must compare to facts known or expected to hold for the family of $L$-functions. We begin that investigation in this paper by comparing to the Conrey-Farmer-Keating-Rubinstein-Snaith predictions~\cite{CFKRS}, adapted to function fields by Andrade and Keating~\cite{AK}, for moments of $L$-functions. The moment of $L$-functions we consider is 
\[ \frac{1}{ q^N (q-1) }  \frac{\log q}{2\pi}  \sum_{\chi \in S_{N+1,q}}  \int_0^{ \frac{2\pi}{\log q}}   \prod_{j=1}^{\rp} L(s_j+i t ,\chi) \prod_{j=\rp+1}^{\rp + \ry} \overline{L ( s_j+ it  ,\chi) }\]
i.e. the average over this family of the product of $\rp$ special values of $L(s+it,\chi)$ with $\ry$ special values of $\overline{L(s+it,\chi)}$.  The recipe of \cite{CFKRS} predicts a main term for this moment of
\[ \operatorname{MT}_{N}^{\rp,\ry} (s_1,\dots, s_{\rp+\ry} )= \prod_{j=\rp+1}^{\rp+\ry} q^{ - (1/2-s_j) N}  \sum_{\substack{  S \subseteq \{1,\dots,\rp+\ry\} \\ \abs{S}=\ry}}  \prod_{j \in S} q^{ (1/2-s_j) N} \sum_{ \substack{ f_1,\dots, f_{\rp+\ry} \in \mathbb F_q[u]^+ \\ \prod_{j \in S} f_j = \prod_{j\notin S} f_j}}  \prod_{j\in S}\abs{f_j} ^{ -1+ s_j } \prod_{j\notin S} \abs{f_j} ^{ -s_j} .\]
In other words, in its most optimistic form the prediction is
\[ \frac{1}{ q^N (q-1) }  \frac{\log q}{2\pi}  \sum_{\chi \in S_{N+1,q}}  \int_0^{ \frac{2\pi}{\log q}}   \prod_{j=1}^{\rp} L(s_j+i t ,\chi) \prod_{j=\rp+1}^{\rp + \ry} \overline{L ( s_j+ it  ,\chi) } = \operatorname{MT}_{N}^{\rp,\ry} (s_1,\dots, s_{\rp+\ry} ) +  O_{q, \rp,\ry,\epsilon} (  (q^N)^{ - ( 1/2-\epsilon)}) \]
and less optimistically one makes the same prediction with a larger error term.

\begin{theorem}\label{intro-cfkrs}Assume that $q>11$. Let $\rp$ and $\ry$ be nonnegative integers and $s_1,\dots, s_{\rp + \ry}$ be complex numbers with real part $\frac{1}{2}$. Then
\begin{equation}\label{eq-cfkrs}  \int_{\Cqs } \Bigl( \prod_{j=1}^{\rp} L(s_j) \prod_{j=\rp+1}^{\rp + \ry} \overline{L ( s_j) } \Bigr)\mu_{ \textrm{ch}}   = \operatorname{MT}_{N}^{\rp,\ry} (s_1,\dots, s_{\rp+\ry} ) 
+ O_{q,\rp,\ry} \left(  N^{  \frac{ (\rp+ \ry)^2}{2} -  \beta \frac{q-2}{4}} \right).\end{equation}
\end{theorem}
For the integral on the left hand side of \eqref{eq-cfkrs}, $L$ should be understood as the variable of integration, i.e. $L(s_j)$ is the function that takes a power series to its value at $s_j$, defined on the subset of $\Cqs$ of power series in $q^{-s}$ with radius of convergence $> q^{-1/2}$. The integral is well-defined since $\mu_{\textrm{ch}}$ is supported on the even smaller subset consisting of polynomials in $q^{-s}$ of degree $N$. 

If we believe the CFKRS prediction for the moments, then \cref{intro-cfkrs} implies that $\mu_{\textrm{ch}} $ has the same moments as the family of Dirichlet characters, up to a certain error, and thus gives evidence that $\mu_{\textrm{ch}} $ is a good model for the $L$-functions of the Dirichlet characters $S_{N+1,q}$ (with additional averaging in the imaginary axis). Alternately, \cref{intro-cfkrs} could be seen as giving a probabilistic explanation of the CFKRS prediction.

In interpreting \cref{intro-cfkrs}, it is helpful to note that the main term $\operatorname{MT}_{N}^{\rp,\ry} (s_1,\dots, s_{\rp+\ry} )$ is, in the special case $s _1=\dots = s_{\rp + \ry}$, a polynomial in $N$ of degree $\rp \ry$. So the error term in \eqref{eq-cfkrs} is smaller than the main term by a factor of $N^{ \beta \frac{q-2}{4} - \frac{ \rp^2 + \ry^2}{2}}$. In particular, it is actually smaller if $q> 2 + \frac{ 2 \rp^2 + 2 \ry^2}{\beta}$ and the number of coefficients of the polynomial that are visible in this estimate (in the sense that their contribution to the main term is greater than the size of the error term) is \[ \min \left( \left\lceil \beta \frac{q-2}{4} - \frac{ \rp^2 + \ry^2}{2} \right\rceil , \rp \ry +1 \right) .\] Thus for $q$ sufficiently large depending on $\rp, \ry$, all coefficients of the polynomial are visible in this sense.

We conjecture that an even stronger statement holds:

\begin{conjecture}\label{cfkrs-conj}Let $\rp$ and $\ry$ be nonnegative integers,$q>2$ a prime power, $s,\dots, s_{\rp + \ry}$ complex numbers with real part $\frac{1}{2}$, and $A$ a real number.
\begin{equation}  \int_{\Cqs } \Bigl( \prod_{j=1}^{\rp} L(s_j ) \prod_{j=\rp+1}^{\rp + \ry} \overline{L ( s_j ) } \Bigr)\mu_{ \textrm{ch}}  \label{cfkrs-rhs} =  \operatorname{MT}_{N}^{\rp,\ry} (s_1,\dots, s_{\rp+\ry} ) + O_{q,\rp,\ry,A} \left(  N^{ -A }  \right).\end{equation} \end{conjecture}

If \cref{cfkrs-conj} is true then the measure $\mu_{\textrm{ch}}$ correctly predicts every coefficient of the polynomial CFKRS main term. 

While this paper considers a particular family of $L$-functions, we hope that similar methods can be applied to essentially any family of $L$-functions, at least in the function field context. One just needs to consider a random Euler product whose local factors match the distribution of the local factors of the family of $L$-function (e.g. for the family of quadratic Dirichlet characters with prime modulus, take the Dirichlet series of random $\pm 1$-valued completely multiplicative functions) and, if necessary, depending on the symmetry type, replace $U(N)$ with $O(N), SO(N)$, or $Sp(N)$. Passing from random Euler products with continuous distributions to discrete distributions introduces some difficulties, but most likely not insurmountable ones.

\subsection{Prior work}

The oldest probabilistic model for the Riemann zeta function is the random Euler product $\prod_{p} \frac{1}{1- \xi(p) p^{-\sigma }}$  where $\xi(p)$ are independent and identically distributed on the unit circle. The distribution of this random Euler product was proven by Bohr and Jessen~\cite{BohrJessen} to give the limiting distribution of $\zeta(\sigma+it)$ for fixed $\sigma >1/2$, and Bagchi~\cite{Bagchi} proved a generalization giving the distribution of $\zeta(s+it)$ as a holomorphic function on a fixed domain to the right of the critical line. 

On the critical line, Selberg's central limit theorem shows that $\frac{ \log \abs{ \zeta(1/2+it)}}{\log \log T}$ has a Gaussian limiting distribution for $t \in [T,2T]$ as $T\to\infty$. The division by $\log \log T$ means that this result is not sensitive to the exact size of zeta, and it similarly gives no information about the zeroes. Probabilistic models for zeta and $L$-functions that give precise descriptions of the behavior on the critical line must, for now, be conjectural.

A crucial starting point is the work of Montgomery~\cite{Montgomery}, who conjectured that the statistics of $k$-tuples of zeroes of the zeta function, in the limit over larger and larger intervals in the critical line, match the statistics of $k$-tuples of eigenvalues of a Haar-random matrix, in the limit of larger and larger random matrices, for each $k$, and provided evidence fo this.

Katz and Sarnak~\cite{KS1,KS2} observed that zeta and $L$-functions in the function field context arise from characteristic polynomials of unitary matrices of fixed size (depending on the conductor of the $L$-function) and that in the large $q$ limit these unitary matrices become Haar-random for several natural families of $L$-functions, so in fact all statistics match statistics of random matrices in the large $q$ limit. Using this, they made conjectures about the distribution of the low-lying zeroes of $L$-functions.

Keating and Snaith~\cite{KS4,KS3} used a random matrix model to study values of zeta and $L$-functions on the critical line, and not just their zeroes. In particular, they calculated the moments of the characteristic polynomial of a Haar-random unitary matrix at a point on the unit circle, in terms of the size of the matrix. To obtain a conjectural expression for the moments of the zeta function at a random point on the critical line, one has to substitute $\log T$ for the size of the matrix in this formula and then multiply by an arithmetic factor that expresses the contribution of small primes. Thus, if one models the values of the zeta function on a random strip of the critical line by the characteristic polynomial of a random unitary matrix on a strip of the unit circle, one obtains predictions for the moments that are conjecturally correct to within a multiplicative factor.

The situation was improved by Gonek, Hughes, and Keating~\cite{GHK}, using an Euler-Hadamard product, which expresses the zeta function locally as a product of one factor which roughly consists of the Euler factors at small primes and another factor which roughly consists of the contributions of nearby zeroes to the Hadamard product. (This depends on an auxiliary parameter -- the more primes one includes, the fewer zeroes are needed, and vice versa). This thereby suggests a model where the first factor is modeled by the Euler product over small primes of a uniformly random multiplicative function and the second factor is modeled by the contributions of zeroes near a given point to the characteristic polynomial of a random unitary matrix, with the two factors treated as independent. (They were later proved to be asymptotically independent, conditional on the Riemann hypothesis, by Heap~\cite{Heap2023}.) For the moments, the \cite{GHK} model recovers the same prediction as \cite{KS4,KS3} of the product of a random matrix factor and an arithmetic factor.

A related but distinct approach to the moments of zeta and $L$-functions is the work of Conrey, Farmer, Keating, Rubinstein, and Snaith~\cite{CFKRS}.  This work did not directly predict the moments using a characteristic polynomial. Instead the authors found a particular formula for the (shifted) moments of the characteristic polynomial of a random unitary matrix and conjectured a formally similar formula for the (shifted) moments of the Riemann zeta function or another $L$-function, roughly speaking by inserting suitable arithmetic factors at an intermediate step in the calculation instead of at the end. However, the intermediate stages of their recipe lack a clear number-theoretic or probabilistic interpretation. In particular, it is not even obvious that their predictions for expectations of powers of absolute values of the zeta function are positive -- this has to be checked separately. It is not clear that there exists any random holomorphic function whose moments are given by the \cite{CFKRS} predictions for moments of zeta, though conjecturally a random shift of the Riemann zeta function would be an example.

Another approach to predicting moments of $L$-functions is by multiple Dirichlet series, initiated in the work of Diaconu, Goldfeld, and Hoffstein~\cite{DGH}. The highest-order terms in these predictions, made around the same time, agree with~\cite{CFKRS}, but the multiple Dirichlet series can be used to predict additional lower-order terms for certain families of $L$-functions, as in the work of Diaconu and Twiss~\cite{DiaconuTwiss}. Again these predictions are not probabilistic in nature, instead based on assuming that meromorphic functions defined by certain complicated multivariable sums have the greatest amount of analytic continuation allowed by their symmetry properties.

The predictions of \cite{CFKRS} for moments of zeta on the critical line are polynomials in $\log T$. The leading term of these moments agrees with the leading term originally predicted in \cite{KS4,KS3} and probabilistically modeled by \cite{GHK}. Thus, the prediction of \cite{GHK} agrees with what is now believed to be correct to within a factor of $1 + O(1/\log T)$.  Gonek, Hughes, and Keating~\cite{GHK} raised the question of whether their model could be extended to predict all the terms of the \cite{CFKRS} polynomial.

The model of~\cite{GHK} has been extended to the function field setting by Bui and Florea~\cite{BuiFlorea}, and then applied to further families of $L$-functions by Andrade and  Shamesaldeen~\cite{AS} and Yiasemides~\cite{Yiasemides}. Again, in this setting the probabilistic model correctly predicts the leading term of the asymptotic that is conjectured by other methods and known in several cases, but fails to predict the lower-order terms. In this case, the predictions are polynomials in the degree of the conductor (i.e. in $N$) instead of $\log T$. Theorem \ref{intro-cfkrs} shows that, for $q$ sufficiently large, a probabilistic model based on matrices that are random but not Haar-random improves on this by a power of $N$.

The idea of integrating against Haar measure times a weight function to calculate the average of a polynomial function on the coefficients function field $L$-functions appeared earlier in work of Meisner~\cite{Meisner2021}, but this work was not probabilistic in nature: the weight function, unlike a probability density function, is not positive (and not even real) and the average must be normalized by a factor depending on the highest weight of the irreducible representations used to express the polynomial (see \S\ref{ss-unify}).

The idea of restricting the support of Haar measure to obtain more accurate predictions was used in the case of elliptic curve $L$-functions by Due\~{n}ez, Huynh, Keating, Miller, and Snaith~\cite{DHKMS}. Their modification of Haar measure was designed to account for the influence of formulas for the critical value of the $L$-function that force that value, suitably normalized to be an integer and in particular prevent it from being very close to zero but nonzero. They accordingly considered a measure on random matrices where the critical special value of the characteristic polynomial is prevented from taking small nonzero values. Our adjustment of the probability measure, on the other hand, is designed to account for the influences of small primes, it also involves changing the density and not just the support.

It would be interesting to check the compatibility of our paper with some of these prior works in more detail. First, it should be possible to define an ``Euler-Hadamard product" for $L$-functions in the support of $\mu_{\textrm{ch}}$. One could then ask how close the distribution of the Euler factors and Hadamard factors is to a product of independent multiplicative function and random matrix distributions at a point (perhaps using an optimal transport distance for probability distributions). If these distributions are close, then not only would $\mu_{\textrm{ch}}$ and \cite{GHK} give similar predictions of the moments, but they would give these predictions for similar reasons.

It would also be enlightening to prove an analogue of \cref{intro-cfkrs} for ratios of $L$-functions rather than products, using the work of Conrey, Farmer, and Zirnbauer~\cite{CFZ} to obtain a classical prediction to compare with.

More ambitiously, if an analogue of our construction was made for the family of quadratic Dirichlet characters, and an analogue of \cref{intro-cfkrs} was proven with an error term of size $O ( q^{ - \delta N } )$ for $\delta > \frac{1}{4}$, then one could look for a probabilistic explanation of the secondary terms in moments of quadratic Dirichlet $L$-functions predicted by multiple Dirichlet series. (If the error term were larger than this, it would dominate the predicted secondary terms, so including them would be meaningless.) It seems unlikely that they could appear for a direct analogue of $\mu_{\textrm{ch}}$, since these secondary terms ultimately arise from the ability to apply a Poisson summation formula in the modulus of the Dirichlet character and recover a similar sum, and the model of Dirichlet characters based on random multiplicative functions used to construct $\mu_{\textrm{ch}}$ wouldn't reflect this Poisson symmetry, but one could very optimistically hope for a natural modification of $\mu_{\textrm{ch}}$ that predicts these terms. 

\subsection{Motivation, variants, and the number field case}

The operation of multiplying the measure of one probability distribution by the density of another, or, equivalently, multiplying the  density of two probability distributions may at first seem strange. However, it has a natural interpretation. Given two different probability distributions $\mu_1,\mu_2$ on $\mathbb R^n$ with continuous probability density functions, we can consider the distribution of a pair of random variables $X_1,X_2$ independently distributed according to $\mu_1$ and $\mu_2$, and then condition on the event that the distance between $\mu_1$ and $\mu_2$ is at most $\delta$. In the limit as $\delta\to 0$, this conditional distribution will converge to the distribution of two identical random variables, each distributed according to a measure with probability density proportional to the product of the densities of $\mu_1$ and $\mu_2$. 

Thus, multiplying the probability densities arising from random matrices and random Euler products can be seen as, first, generating pairs of random matrices and random Euler products and, second, throwing away those pairs where the characteristic polynomial of the random matrix is not close to the Euler product. This is a plausible way to generate random functions that arise both as characteristic polynomials of matrices and Euler products (as the Dirichlet $L$-functions $L(s,\chi)$ do). However, it cannot be quite right as a model for Dirichlet $L$-functions, giving the wrong answers in the $q\to\infty$ and $N\to\infty$ limits. This can be seen most clearly if we let both $q$ and $N$ head to $\infty$, so the distributions of $L_\xi$ and $L_M$ both converge to the exponentials of random power series with independent complex Gaussian coefficients. Multiplying the probability densities corresponds to squaring the Gaussian probability density function, producing a Gaussian with half the variance. However, to obtain a distribution that interprets between $\mu_{\textrm{ep}}$ and $\mu_{\textrm{rm}}$, we would like a distribution that converges to the original Gaussian in the $q,N \to \infty $ limit. We fix this by dividing by the same Gaussian.

From this heuristic, the right choice of $k$ is not clear. It seems likely that the measure $\mu_{\textrm{ch}}$ does not depend much on the parameter $k$. The specific value of $k$ chosen makes the analysis as easy as possible, but similar results should be true in a broader range of $k$.

In fact, if for each value of $k$ we let  $F_k(x_1,\dots, x_k)$ be the probability density function fo $X_{1,\xi},\dots, X_{k,\xi}$ and define $\mu_{\textrm{weighted}}$ to be proportional to
\[ \lim_{k\to\infty} \Biggl( \frac{ F_k( - q^{1/2} \tr(M),\dots, -q^{k/2} \tr (M^k)/k)}{  \prod_{j=1}^k  \left( e^{  -  \frac{  \abs{ \tr (M^j)}^2 }{ j } }  \frac{ j }{ q^j \pi} \right) } \Biggr) \mu_{\textrm{Haar} } \] then the same results should be true. This definition is more canonical as it lacks the parameter $k$, and fits naturally with an infinite-dimensional version of the heuristic for multiplying two probability density functions. However, proving the same results for this measure introduces additional analytical difficulties, starting with proving that the limit as $k$ goes to $\infty$ exists, that we do not pursue here.

An alternate approach to constructing a measure satisfying \cref{intro-lf} and \cref{intro-hf} is to first check that the pairing $\langle \phi, \psi \rangle = \int_{\Cqs }  \phi  \overline{\psi} \mu_{\textrm{rm}}$ is nondegenerate on polynomials in $\mathbb C[c_0,c_1,\dots, \overline{c_0},\overline{c_1},\dots] $ of degree $\leq k$, and using this, verify that there exists a unique $\psi \in \mathbb C[c_0,c_1,\dots, \overline{c_0},\overline{c_1},\dots] $ of degree $\leq k$ such that \[  \int_{U(N) }  \phi   \overline{\psi}  \mu_{\textrm{rm}}  =  \int_{\Cqs }    \phi  \mu_{\textrm{ep}}\] for all $\phi \in \mathbb C[c_0,c_1,\dots, \overline{c_0},\overline{c_1},\dots] $ of degree $\leq k$, note that $\psi$ is real-valued, and then consider the signed measure $\psi \mu_{\textrm{rm}}$, for which \cref{intro-lf} and \cref{intro-hf} hold with vanishing error term. The main difficulty with this approach is that the signed measure may not actually be a measure, as the function $\psi$ may be negative on the support of $\mu_{\textrm{rm}}$. It is easy to check that $\psi$ is nonnegative as long as $q$ is sufficiently large with respect to $N$, but for $q$ fixed,  $\psi$ is negative even for small values of $N$. If we view $\psi \mu_{\textrm{rm}}$ as an approximation of the true distribution of $L(s,\chi)$, the problem is clear: since $L(s,\chi)$ is supported on power series with first coefficient $c_1$ satisfying $\abs{c_1} \leq q$, while $\mu_{\textrm{rm}}$ is supported on power series with $\abs{c_1} \leq q^{1/2} N$, as long as $q< N^2$, the true measure vanishes on a large region where $\mu_{\textrm{rm}}$ is supported, and so $\psi$ is approximating a function which is zero on that region. Since polynomials cannot be zero on a region without being identically zero, polynomial approximations of functions zero on a region will tend to oscillate between positive and negative values on that region. Thus the signed measure $\psi \mu_\textrm{rm}$ is rarely a measure. However, this argument suggests that, given that \cref{intro-support} shows that $\mu_{\textrm{ch}}$ has more reasonable support, it may be possible to multiply $\mu_{\textrm{ch}}$ by a low-degree polynomial to improve the error term in \cref{intro-lf} without compromising positivity.

Whether the strategy of this paper can be applied in the number field context is not yet clear. The fundamental difficulty seems to be that a number field $L$-function contains much more information than a function field $L$-function, so it is harder for the supports of distributions arising from random Euler products and random matrix models to intersect.

Let us make this more precise. Consider the problem of defining a random holomorphic function in a variable $s$ whose properties approximate the behavior of $\zeta(s+it)$ for $t$ random near a given value $T$. The basic steps are to define an analogue of the random matrix model, define a random Euler product model, and combine them. Since the function $n^{it}$ behaves like a random multiplicative function with absolute value $1$ and values at the primes independently uniformly distributed on the unit circle, we can again use the Dirichlet series of random completely multiplicative functions as our Euler products. Whatever our random matrix model looks like, the holomorphic functions it produces will probably have functional equations (since the characteristic polynomials of unitary and Hermitian matrices each satisfy a functional equation, and we are trying to approximate the zeta function, which satisfies a functional equation). One natural functional equation to choose is $f(1-s) =\epsilon   (T/2\pi)^{1/2-s}  f(s) $ since this is consistent with a holomorphic function having zeroes on the critical line distributed with frequency $(1/2\pi) \log (T/2\pi) $ -- in other words, the frequency of zeroes of zeta near $T$. However, it is easy to see that there are no multiplicative functions whose Dirichlet series satisfy that functional equation, as it forces the coefficients of $n^{-s}$ to vanish for $n> T/2\pi$. So the intersection of the support of the distribution of any random matrix characteristic polynomials having one of these functional equations with the support of the distribution of Dirichlet series of random multiplicative functions will simply vanish, and any attempt to multiply the densities of these distributions will produce a zero distribution.

A similar conclusion can be drawn if we keep the usual functional equation of the Riemann zeta function. In that case, it follows from Hamburger's theorem~\cite{Hamburger} that the intersection of the supports will contain only the actual shifts of the original Riemann zeta function, and so searching for a distribution supported on the intersection, or multiplying the densities, will simply produce the distribution of random shifts of the zeta function. Of course, it is pointless to model these shifts using the shifts themselves.

Observing this problem immediately suggests the rough form of the solution. Rather than looking for a distribution supported on the intersection of the supports of the distributions, we should look for a distribution supported on points which are close to the support of both distributions. In other words, in the heuristic for the product of two probability densities as a $\delta\to 0$ limit, we should avoid taking the limit and instead fix a value of $\delta$. Of course, the nature of this depends on exactly how we define the distance between two holomorphic functions. A natural choice is to integrate the square of the absolute value of the difference between their logarithms against some measure on a subset of the complex plane where they are both defined, but we have a great deal of choice on the measures.

In fact, rather than conditioning the joint distribution of the characteristic polynomial of a random matrix and random Euler product on the event that the two holomorphic functions are close, it seems better to weight the joint probability distribution by the exponential of minus the square of the distance, or another quadratic form in the two functions, before normalizing by a constant to have the total mass one. This weighting sends Gaussians to Gaussians, and should be normalized so that inputting Gaussian approximations to the two distributions outputs a joint distribution whose marginals are the original Gaussians, coupled so that with high probability the two holomorphic functions take similar values at points near $0$ to the right of the critical line. (We cannot compare them on the critical line itself since the random Euler products admit a natural boundary there). But it is not clear if there is a single natural coupling to work with. 

In physics, one can consider the eigenvalues of random matrices as being a statistical mechanics model of particles, either on the line or the unit circle, that repel each other and thus have a lesser probability of being close together than independent random points. Specifically, the probability density function should be the exponential of a negative constant times the energy of the system, so the terms in the Weyl integration formula involving the difference of two eigenvalues correspond to a contribution to the energy depending on the distance between two points. We can view this type of exponentially-weighted joint distribution as a statistical mechanics model of points on the critical line together with values $\xi(p)$ on the unit circle for each prime, where, in addition to interacting with each other, the points interact with $\xi(p)$. For the zeta function itself, the interaction is infinitely strong, to the extent that the primes determine the zeroes and the zeroes determine the primes. By choosing an interaction whose strength is not too large and not too small, we may be able to construct a model of the Riemann zeta function whose properties are amenable to computation.

Regardless, this approach produces a joint distribution of two holomorphic functions, one the characteristic polynomial of a matrix and the other an Euler product, but to model the Riemann zeta function we only want one. The simplest approach is to throw out the Euler product, since its natural boundary on the critical line makes it inappropriate for modeling the zeta function on the line, but it may be possible to combine them in a subtler way. 

The exact random matrix model to use is of course a question. A good choice might be to take the characteristic polynomial of a random unitary matrix and plug in $e^{ \frac{ \log T/2\pi}{\log N} \left(\frac{1}{2}-s \right)}$. This produces a holomorphic function on the whole complex plane with zeroes on the critical line with the correct zero density. Since it is periodic in the imaginary axis, it can't be a good model for the large-scale behavior of the Riemann zeta function, but as long as $ N$ is somewhat larger than $\log T$ it may be a good model for the local behavior. (The models of \cite{KS4,KS3,GHK} require setting $N$ very close to $\log T/2\pi$, but coupling with the Euler product will damp oscillations with frequencies less than that of the leading term $2^{-s}$, allowing us to take larger values of $N$ without getting obviously wrong predictions, and taking larger values of $N$ seems necessary to accurately approximate the contribution of the $2^{-s}$ term.) However, we could also consider the eigenvalues of random Hermitian matrices of fixed size, or point processes on the whole critical line. (The determinantal point process associated to the sine kernel, which is the large $N$ limit of random matrices, is not useful for this, as its ``characteristic polynomial" is not a well-defined holomorphic function, basically because the distribution of the characteristic polynomial of a random matrix, normalized to keep the frequency of zeroes constant, doesn't converge in the $N \to \infty$ limit, but another point process might work.)

One can optimistically hope that there is some reasonably natural way of making the sequence of choices discussed above for which an analogue of \cref{intro-cfkrs} or, ideally, \cref{cfkrs-conj} can be proven. Proving this should be analytically more difficult than \cref{cfkrs-conj}, since the two distributions we are trying to combine are further from each other and further from the Gaussian model and thus showing that the combination has the desired properties of each one should be more difficult, so proving the strongest possible form of \cref{cfkrs-conj} might be a stepping stone to handling the number field case.

\subsection{Geometric and probabilistic approaches to $L$-functions}\label{ss-unify} 

The probabilistic model $\mu_{\textrm{ch}}$ is compatible with the geometric and representation-theoretic approach to the moments of $L$-functions suggested by the same author in \cite{s-rep}. Specifically, from the geometric perspective the most natural test functions to integrate against are the characters of irreducible representations of $U(N)$, which may be expressed as polynomials in the coefficients of the characteristic polynomial $L_M$ using the fundamental theorem of symmetric polynomials, or, more explicitly, the second Jacobi-Trudi identity for Schur polynomials~\cite[Formula A6]{FH}.

Conversely, any polynomial in the coefficients of $L_M$ can be expressed as a linear combination of characters of irreducible representations.  We will check in \S\ref{ss-reps} that the polynomials of degree $\leq k$ are exactly the linear combinations of characters whose highest weights, expressed as an $N$-tuple of integers, have absolute value summing to a number $\leq k$. Thus, by orthogonality of characters, irreducible representations whose highest weights have absolute value sums $>k$ are orthogonal to all polynomials of degree $\leq k$.

Hence \cref{intro-lf} applies to the characters of irreducible representations with small highest weight, showing that the averages of these functions over $\mu_{\textrm{ch}}$ match their averages over $\mu_{\textrm{ep}}$, while \cref{intro-hf} applies to the characters of irreducible representations with large highest weight, showing that the averages of these functions over $\mu_{\textrm{ch}}$ cancel. 

By Weil's Riemann hypothesis, every $L$-function $L(s+it,\chi)$ can be expressed as $L_M$ for $M\in U(N)$ unique up to conjugacy, so we can interpret characters of irreducible representations of $U(M)$ as functions of $L(s+it,\chi)$. For irreducible representations of small highest weight, it is not hard to prove that the averages of their characters over $L(s+it,\chi)$ match the averages of the same characters over $\mu_{\textrm{ep}}$. \cite{s-rep} showed that the CFKRS predictions for moments of $L$-functions could be explained by cancellation in the averages of characters of irreducible representations with large highest weight over the family of $L$-functions, which could in turn be explained by (hypothetical) vanishing of certain cohomology groups whose traces of Frobenius compute this average. \cref{intro-hf} shows that the cancellation of averages of characters of irreducible representations with large highest weight could also be explained by the probabilistic model $\mu_{\textrm{ch}}$. So this cancellation could have both probabilistic and geometric explanations. (However, note that the amount of cancellation that one can prove in the probabilistic model is different from the amount one can prove under geometric hypotheses -- at least currently, it is larger for some representations and smaller for others. Thus it is not possible to say the geometric hypothesis implies the probabilistic model, or vice versa.)

Note that the definition of ``small highest weight" used in the two contexts is not identical (the definition here is stricter). This is because the average of the character of an irreducible representation over $\mu_{\textrm{ep}}$ decreases with the highest weight of the representation, at least for representations relevant to calculating moments of fixed degree. Thus, as long as the highest weight is not too small, it is possible for the average against another measure both to cancel and to approximate the average against $\mu_{\textrm{ep}}$, simply because the average of $\mu_{\textrm{ep}}$ is itself small. So whether we state that these averages cancel or approximate $\mu_{\textrm{ep}}$ is a matter of convenience, and how we sort representations into those two buckets can vary with the context.  The only restriction is that, as the error term in our desired estimates shrinks, fewer representations are flexible in this way.

For the case of $L$-functions of quadratic Dirichlet characters, analysis analogous to \cite{s-rep} was conducted by Bergstr\"om, Diaconu, Petersen, and Westerland~\cite{bdpw}. In the quadratic Dirichlet character setting, the $L$-function is naturally a characteristic polynomial of a conjugacy class in $USp(2N)$, so one considers characters of irreducible representations of $USp(N)$. They derive the CFKRS predictions, or, equivalently in this setting, the highest-order term of the multiple Dirichlet series predictions, from the assumption of cancellation in averages of characters of irreducible representations of $USp(N)$ of large highest weight. They prove a homological stability result which is a topological enhancement of the fact that averages of irreducible representations of small highest weight over $L(s+it,\chi)$ (where $\chi$ is now a quadratic Dirichlet character) match the averages of the same characters against a suitable analogue of $\mu_{\textrm{ep}}$. Since the stable cohomology vanishes in low degrees for representations of large highest weight, the vanishing of cohomology groups whose traces of Frobenius compute the average and hence cancellation in the average follows (for $q$ sufficiently large) from a certain uniform homological stability statement, later proven by Miller, Patzt, Petersen, and Randal-Williams~\cite{mpprw}.

\subsection{Proof sketch}
We now sketch the proofs of the main theorems. Recall that in the definition of $\mu_{\textrm{weighted}}$ we take the Haar measure and multiply by the probability density function of $X_{1,\xi},\dots, X_{k,\xi}$ divided by a Gaussian probability density function. A key observation is that, if we instead took a suitable Gaussian measure and multiplied by the probability density function of $X_{1,\xi},\dots, X_{k,\xi}$ divided by a Gaussian probability density function, the density of the Gaussian would cancel and we would obtain the distribution of $X_{1,\xi},\dots, X_{k,\xi}$. For this modified measure, the expectation of a low-degree polynomial matches its expectation against $\mu_{\textrm{ep}}$ simply because $X_{1,\xi},\dots, X_{k,\xi}$ are the coefficients of the random power series $\log L_\xi$ distributed according to $\mu_{\textrm{ep}}$. (The low degree assumption is necessary here because high-degree polynomials may involve coefficients of the power series beyond the first $k$ and thus can't be expressed as functions of $X_{1,\xi},\dots, X_{k,\xi}$.)

So proving \cref{intro-lf} is a matter of proving that the expectation of low-degree polynomials is not changed much by the fact that we used the Haar measure instead of the Gaussian measure to construct $\mu_{\textrm{weighted}}$ and $\mu_{\textrm{ch}}$. It thus crucially requires a bound for the difference, in some sense, between the Haar measure and the Gaussian measure. We rely on the work of Johansson and Lambert \cite{JohanssonLambert}, who proved a bound for the total variation distance between these distributions. Multiplying a measure by a continuous function can increase the total variation distance proportionally to the sup-norm of the function, so applying this result in our setting requires bounding the sup-norm of the multiplier \begin{equation}\label{intro-multiplier}\frac{ F( - q^{1/2} \tr(M),\dots, -q^{k/2} \tr (M^k)/k)}{  \prod_{j=1}^k  \left( e^{  -  \frac{  \abs{ \tr (M^j)}^2 }{ j } }  \frac{ j }{ q^j \pi} \right) } . \end{equation}  This requires pointwise bounds for the probability density function $F(x_1,\dots,x_k)$ which decrease rapidly as $x_1,\dots, x_k$ grows.

To obtain pointwise bounds for $F(x_1,\dots,x_k)$, we first bound the integrals of $F(x_1,\dots,x_k)$ against a a complex exponential function of $x_1,\dots,x_k$. Taking the Fourier transform, i.e. integrating against an imaginary exponential function of $x_1,\dots,x_k$ would be sufficient if we only wanted a bound for $F(x_1,\dots,x_k)$ which is uniform in $x_1,\dots, x_k$, while integrating against a real exponential function would be sufficient if we wanted a bound for the the integral of $F(x_1,\dots, x_k)$ over a large region which decreases rapidly as the region becomes more distant from $0$, but since we are interested in bounds that are both pointwise and rapidly decreasing we require exponentials of complex-valued functions. The advantage of studying these exponential integrals is that the definition of $F$ as the probability density function of a sum of independent random variables immediately gives a factorization of the exponential integral as a product of simpler integrals, in this case over the unit circle. Thus, a large part of our proof involves proving elementary bounds for these exponential integrals over the unit circle, and then multiplying them together to obtain bounds for integrals of $F$.

For \cref{intro-hf}, on the other hand, the statement becomes trivial if we replace the multiplier \eqref{intro-multiplier} in the definition of $\mu_{\textrm{weighted}}$ and $\mu_{\textrm{ch}}$ by any polynomial of degree $\leq k$ in the coefficients of a power series. Thus proving \cref{intro-hf} is a matter of finding a suitable approximation of \begin{equation}\label{simplified-multiplier} \frac{ F( x_1,\dots, x_k )}{  \prod_{j=1}^k  \Bigl( e^{  -  \frac{  |x_j|^2 }{ j q^j } }  \frac{ j }{ q^j \pi} \Bigr) } \end{equation} by a low-degree polynomial in $x_1,\dots, x_k, \overline{x_1},\dots, \overline{x_k}$ and bounding the error of this approximation. We choose an approximation in the $L^2$ sense, with the $L^2$ norms calculated against the Gaussian measure. (We again use the results of \cite{JohanssonLambert} to compare the Gaussian measure to the Haar measure). The optimal $L^2$ approximation against the Gaussian measure can be obtained using the orthogonal polynomials for the Gaussian measure, the Hermite polynomials: Since they form an orthogonal basis, any $L^2$ function can be written as a linear combination of them, and then one truncates the linear combination by taking only the low-degree polynomial terms, leaving the coefficients of the high-degree polynomials as an error. Bounding the error of this approximation is equivalent to bounding the coefficients of Hermite polynomials of high degree in the Hermite polynomial expansion of \eqref{simplified-multiplier}. These coefficients are naturally expressed as contour integrals of exponential integrals of $F(x_1,\dots,x_k)$ and we can again bound them by bounding the exponential integrals.

Finally, \cref{intro-cfkrs} is obtained by expressing $\prod_{i=1}^{\rp} L(s_i ) \prod_{i=\rp+1}^{\rp + \ry} \overline{L ( s_i ) }$ as a polynomial in the coefficients of $L$ and their complex conjugates, breaking that polynomial into low-degree terms and high-degree terms (using irreducible representations, as in \S\ref{ss-unify}) and applying \cref{intro-lf} to the low-degree terms and \cref{intro-hf} to the high-degree terms.

\subsection{Acknowledgments}

The author was supported by NSF grant DMS-2101491 and a Sloan Research Fellowship while working on this project.

I would like to thank Amol Aggorwal, David Farmer, Jon Keating, and Andrei Okounkov for helpful conversations.

\tableofcontents

\section{Random Euler products}
The variables, $X_{1,\xi},\dots, X_{n,\xi}$ are valued in $\mathbb C$, but it will be convenient for us to treat them as valued in $\mathbb R^2$ by viewing complex numbers as real vectors in the usual way, taking their real and imaginary parts as coordinates. This is because we will mainly be interested in their dot products with other vectors in $\mathbb R^2$, which can be expressed in terms of complex numbers but less directly.

To that end, we give a formula for $X_{n,\xi}$ as a vector in $\mathbb R^2$.  For each prime polynomial $\mathfrak p$, let $\theta_{\mathfrak p}$ be the argument of $\xi(\mathfrak p)$, so that \eqref{Xn-formula} gives
\[  X_{n,\xi} = \sum_{d \mid n} \sum_{ \substack{\mathfrak p \in \mathbb F_q[u]^+ \\ \deg \mathfrak p =d  \\ \textrm{prime}} }  \frac{d}{n}e^{ i \frac{n}{d} \theta_{\mathfrak p} } .\]

Now for $\theta\in \mathbb R$ let  \[ \ex(\theta) = \begin{pmatrix} \cos \theta \\ \sin \theta \end{pmatrix} \] be $e^{ i\theta} \in \mathbb C$ viewed as a vector in $\mathbb R^2$, so that we have
\begin{equation}\label{Xn-R2-formula}  X_{n,\xi} = \sum_{d \mid n} \sum_{ \substack{\mathfrak p \in \mathbb F_q[u]^+ \\ \deg \mathfrak p =d  \\ \textrm{prime}} }  \frac{d}{n} \ex( \frac{n}{d} \theta_{\mathfrak p} ) .\end{equation}

 Our first goal will be to upper bound $F(x_1,\dots, x_k)$. Two basic tools to do this are the Fourier transform of $F(x_1,\dots,x_k)$, i.e. the characteristic function of $X_{1,\xi},\dots, X_{k,\xi}$, represented by the expectation
\[  \mathbb E [ e^{  i \sum_{n=1}^k X_{n,\xi} \cdot w_n } ] \]
for vectors $w_1,\dots, w_n \in \mathbb R^2$, and the Laplace transform of $F(x_1,\dots,x_k)$, i.e. the moment generating function of $X_{1,\xi},\dots, X_{k,\xi}$, represented by the expectation
\[  \mathbb E [ e^{   \sum_{n=1}^k X_{n,\xi} \cdot v_n } ]\]
for vectors $v_1,\dots, v_n \in \mathbb R^2$. We will in fact need a hybrid of these, also referred to as the Laplace transform, expressed as 
\[ \mathbb E [ e^{ \sum_{n=1}^k X_{n,\xi} \cdot v_n  +  i \sum_{n=1}^k X_{n,\xi} \cdot w_n } ].\]
We could equivalently express  $X_{n,\xi} \cdot v_n  +  i ( X_{n,\xi} \cdot w_n)$ as  $ z_{n,1}   X_{n,\xi} + z_{n,2}  \overline{X_{n,\xi} }$ for a certain pair of complex numbers $z_{n,1}, z_{n,2}$ depending real-linearly on $v_n$ and $w_n$, but this would be unwieldy for the calculations we want to do, which focus on the size of these expectations, as we want to separate out the parameters $v_n$ which affect the size of the exponential from the parameters $w_n$ which affect only its argument. Thus it is better for our purposes to work with dot products in $\mathbb R^2$.

Let $E_d$ be the number of prime polynomials of degree $d$ in $\mathbb F_q[u]^+$.

\begin{lemma}\label{laplace-formula} Let $v_1,\dots, v_k$ and $w_1,\dots, w_k$ be vectors in $\mathbb R^2$. Then
\[ \mathbb E [ e^{ \sum_{n=1}^k X_{n,\xi} \cdot v_n  +  i \sum_{n=1}^k X_{n,\xi} \cdot w_n } ] = \prod_{d=1}^k \Bigl(  \int_0^{2\pi}   e^{ \sum_{m=1}^{\lfloor \frac{k}{d} \rfloor } (\ex(m \theta)  \cdot v_{md} + i ( \ex(m \theta)\cdot w_{md} ))/m  }     \frac{d\theta}{2\pi}  \Bigr)^{E_d} .\] \end{lemma}

\begin{proof} From \eqref{Xn-R2-formula} we have
\[ \sum_{n=1}^k X_{n,\xi} \cdot v_n= \sum_{n=1}^k \sum_{d \mid n} \sum_{ \substack{\mathfrak p \in \mathbb F_q[u]^+ \\ \deg \mathfrak p =d  \\ \textrm{prime}} }  \frac{d}{n} \ex( \frac{n}{d} \theta_{\mathfrak p} ) \cdot v_n \]
which writing $m= n/d$ and switching the order of summation is 
\[ \sum_{d=1}^k   \sum_{ \substack{\mathfrak p \in \mathbb F_q[u]^+ \\ \deg \mathfrak p =d  \\ \textrm{prime}} }  \sum_{m=1}^{ \lfloor \frac{k}{d} \rfloor}  \ex( m \theta_{\mathfrak p} ) \cdot v_{md}/m  \]
so
\[ e^{ \sum_{n=1}^k X_{n,\xi} \cdot v_n   } = \prod_{d=1}^k \prod_{ \substack{\mathfrak p \in \mathbb F_q[u]^+ \\ \deg \mathfrak p =d  \\ \textrm{prime}} } e^{ \sum_{m=1}^{ \lfloor \frac{k}{d} \rfloor}  \ex(m \theta_{\mathfrak p})  \cdot v_{md} /m}.\]
An analogous identity holds with $w_n$. Since $\theta_{\mathfrak p}$ are independent for different $\mathfrak p$ and uniformly distributed in $[0,2\pi]$, we have
\[ \mathbb E [ e^{ \sum_{n=1}^k X_{n,\xi} \cdot v_n  +  i \sum_{n=1}^k X_{n,\xi} \cdot w_n } ] =\mathbb E \Bigl[ \prod_{d=1}^k \prod_{ \substack{\mathfrak p \in \mathbb F_q[u]^+ \\ \deg \mathfrak p =d  \\ \textrm{prime}} } e^{ \sum_{m=1}^{ \lfloor \frac{k}{d} \rfloor} (\ex(m \theta_{\mathfrak p})  \cdot v_{md} + i ( \ex(m \theta_{\mathfrak p})\cdot w_{md} ))/m  } \Bigr] \]
\[ = \prod_{d=1}^k \prod_{ \substack{\mathfrak p \in \mathbb F_q[u]^+ \\ \deg \mathfrak p =d  \\ \textrm{prime}} } \mathbb E\Bigl[ e^{ \sum_{m=1}^{ \lfloor \frac{k}{d} \rfloor}  (\ex(m \theta_{\mathfrak p})  \cdot v_{md} + i ( \ex(m \theta_{\mathfrak p}) \cdot w_{md} ))/m  } \Bigr]\]
\[=   \prod_{d=1}^k  \prod_{ \substack{\mathfrak p \in \mathbb F_q[u]^+ \\ \deg \mathfrak p =d  \\ \textrm{prime}} }  \Bigl(  \int_0^{2\pi}   e^{ \sum_{m=1}^{\lfloor \frac{k}{d} \rfloor }  (\ex(m \theta)  \cdot v_{md} + i ( \ex(m \theta)\cdot w_{md} ))/m} \frac{ d\theta}{2\mathfrak p }   \Bigr) . \qedhere \]  \end{proof}

In view of \cref{laplace-formula}, we will begin by estimating $\int_0^{2\pi}   e^{ \sum_{m=1}^{\lfloor \frac{k}{d} \rfloor } (\ex(m \theta)  \cdot v_{md} + i ( \ex(m \theta)\cdot w_{md} ))/m}  \frac{d\theta}{2\pi} $, starting with the case where $w_n=0$ for all $n$ before handling the general case. This will require different techniques to provide useful estimates with $v_1,w_1$ in different ranges.

\subsection{Real exponential integrals}

This subsection is devoted to estimating $\int_0^{2\pi}   e^{ \sum_{m=1}^{\infty}  \ex(m \theta) \cdot v_m/m }     \frac{d\theta}{2\pi} $. We have expanded the finite sum to an infinite sum because our estimates need to be uniform in the length of the sum and bounds uniform in the length of the sum are equivalent to bounds in the infinite sum case but the infinite sum statements are slightly more elegant and general.  A simple guess for the average of this sum, based on a second-order Taylor expansion, is $ e^{ \sum_{m=1}^{\infty}\frac{  \abs{v_m}^2}{4m^2}}$. Our goal will be to prove a bound of roughly this shape, though our final bound will be worse in some ranges and better in others. 

Our rough strategy to estimate $\int_0^{2\pi}   e^{ \sum_{m=1}^{\infty}  \ex(m \theta) \cdot v_m/m }     \frac{d\theta}{2\pi} $ is to use an argument controlling the error of a Taylor series when the variables are small and the trivial bound  $e^{ \sum_{m=1}^{\infty}  \ex(m \theta) \cdot v_m/m }   \leq  e^{ \sum_{m=1}^\infty \abs{v_m}/m}$ when the variables are large. The argument needs to be more complex since we have infinitely many variables, some of which may be small and sum of which may be large. We first handle the case where $  \abs{v_m}=0$ for all $m>1$  in Lemma \ref{laplace-ov}, where we obtain a savings over the simple guess $e^{ \frac{ \abs{v_1}^2}{4} }$ using a finite Taylor series in the small range, the trivial bound in the large range, and a different power series argument in an intermediate range. This savings will be very convenient throughout the argument as by shrinking it slightly we can absorb unwanted terms from other estimates. In \cref{laplace-Taylor} we make a more complicated, multivariable Taylor series estimate. This is expressed in terms of a ratio of integrals to allow us to preserve the savings. Finally in \cref{laplace-unified} we combine these estimates and use a version of the trivial bound that allows us to ignore an individual $v_m$ if it is too large. 

\begin{lemma}\label{laplace-ov} There exists $\delta_1>0$ such that for all $v_1\in \mathbb R^2$ we have
\[ \log \int_0^{\pi} e^{ \ex(\theta) \cdot v_1} \frac{d\theta}{2\pi} \leq \frac{ \abs{v_1}^2}{4} - \delta_1 \min ( \abs{v_1}^4, \abs{v_1}^2) .\]
\end{lemma}

\begin{proof} It is equivalent to show that the function \begin{equation}\label{laplace-ov-transformed} \frac{ \frac{\abs{v_1}^2}{4} - \log \int_0^{2\pi } e^{ \ex(\theta) \cdot v_1} \frac{d \theta}{2\pi}}{ \min ( \abs{v_1}^4, \abs{v_1}^2) }\end{equation} has a lower bound $\delta_1>0$ for all $v_1\in \mathbb R^2\setminus \{0\}$.

We first check that \eqref{laplace-ov-transformed} is positive on all of $\mathbb R^2 \setminus \{0\}$. To do this, we use the power series
\begin{equation}\label{one-variable-power-series} \log \int_0^{2\pi } e^{ \ex(\theta) \cdot v_1} \frac{d \theta}{2\pi}  = \log \sum_{d=0}^{\infty} \frac{ \abs{v_1}^{2d}}{ (d!)^2 2^{2d} } .\end{equation}
and note that \eqref{one-variable-power-series} is strictly less than  \[  \frac{\abs{v_1}^2}{4}  = \log e^{ \frac{\abs{v_1}^2}{4} } = \log   \sum_{d=0}^{\infty} \frac{ \abs{v_1}^{2d}}{ (d!) 2^{2d} }.\] It follows immediately that \eqref{laplace-ov-transformed} is positive.

Next, using the first couple terms of the Taylor series for logarithm, we compute the Taylor series of  \eqref{one-variable-power-series} as $ \frac{ \abs{v_1}^2}{4} - \frac{ \abs{v_1}^4}{64} + \dots$ and conclude that \eqref{one-variable-power-series} is equal to $\frac{ \abs{v_1}^2}{4}  - \frac{ \abs{v_1}^4}{64} + O ( \frac{\abs{v_1}^6}{6}) $ for $\abs{v_1}$ small. Plugging this into \eqref{laplace-ov-transformed}, we see that \eqref{laplace-ov-transformed} converges to $\frac{1}{64}$ as $\abs{v_1}$ goes to $\infty$.

Finally, $e^{ \ex(\theta) \cdot v_1}  \leq e^{\abs{v_1}}$ so that $ \log \int_0^{2\pi } e^{ \ex(\theta) \cdot v_1} \frac{d \theta}{2\pi} \leq \abs{v_1}$ and thus for $\abs{v_1} \geq 1$, \eqref{laplace-ov-transformed} is at least \[ \frac{  \frac{\abs{v_1}^2}{4} -   \abs{v_1}}{\abs{v_1}^2} = \frac{1}{4} - \frac{1}{\abs{v_1}}\] and thus converges to $\frac{1}{4}$ as $\abs{v_1}$ goes to $\infty$.

Thus \eqref{laplace-ov-transformed}, a continuous function on $\mathbb R^2 \setminus \{0\}$, is positive everywhere and bounded away from $0$ in both a neighborhood of $0$ and a neighborhood of $\infty$. By compactness of $\mathbb R^2 \cup \{\infty\}$, it follows that \eqref{laplace-ov-transformed} has a lower bound $\delta_1>0$. \end{proof}

To apply Taylor's theorem in the case of infinitely many variables, we will need some trick to relate the power series of a function in many variables to the power series of a function in fewer variables. This may be accomplished using the following lemma.
\begin{lemma}\label{weird-ps-comparison} Let $g_1$ and $g_2$ be power series in one or more variables, with constant coefficients $1$, such that $g_1$ has nonnegative coefficients and the coefficient of each monomial in $g_1$ is greater than or equal to the coefficient of the corresponding monomial in $g_2$.  Then $-\log(2-g_1)$ has nonnegative coefficients and the coefficient of each monomial in $-\log(2-g_1)$ is greater than or equal to the coefficient of the corresponding monomial in $\log g_2$. \end{lemma}

\begin{proof} We have \[ \log (g_2) = \log (1 + (g_2-1)) =  (g_2-1) - \frac{ (g_2-1)^2}{2} + \frac{ (g_2-1)^3}{3} -  \frac{ (g_2-1)^4}{4} + \dots. \] 
Writing each term in $g_2-1$ as a sum of coefficients times monomials, and bounding each coefficient by the corresponding coefficient in $g_1-1$, we see that the coefficient of any monomial in this expression is at most the coefficient of the same monomial in
\[ (g_1-1) + \frac{ (g_1-1)^2}{2} + \frac{ (g_1-1)^3}{3} + \frac{ (g_1-1)^4}{4} + \dots = - \log ( 1- (g_1-1)) =-\log(2-g_1).\] \end{proof}

 Let $\times$ denote multiplication of complex numbers viewed as a multiplication operation for vectors in $\mathbb R^2$.

\begin{lemma}\label{laplace-Taylor} Let $(v_m)_{m=1}^{\infty}$ be a sequence of vectors in $\mathbb R^2$. For any $n$ let $u_n = \sum_{m=n}^{\infty} \frac{ \abs{v_m}}{m} $.  Then as long as $u_1< \frac{1}{2}$ we have
\begin{equation}\label{eq-laplace-Taylor} \log  \int_0^{2\pi}   e^{ \sum_{m=1}^{\infty}  \ex(m \theta) \cdot v_m/m }     \frac{d\theta}{2\pi}  - \log \int_0^{\pi} e^{ \ex(\theta) \cdot v_1} \frac{d\theta}{2\pi}  = \sum_{m=2}^{\infty} \frac{ \abs{v_m}^2}{4m^2} +  \frac{ (v_1 \times v_1) \cdot v_2}{16} + O( u_1 u_2 u_3 + u_1^3 u_2) . \end{equation}
\end{lemma}

\begin{proof} We use $[x_1^{n_1} x_2^{n_2}]$ to denote extracting the coefficient of $x_1^{n_1} x_2^{n_2}$ in a power series in $x_1,x_2$. We begin with the observation
\[ \abs{ e^{ x_1 \ex(\theta) \cdot v_1   + x_2 \sum_{m=2}^{\infty}  \ex(m \theta) \cdot v_m/m }[x_1^{n_1} x_2^{n_2}] } \leq  \abs{v_1}^{n_1} \abs{u_2}^{n_2} / (n_1! n_2!) = e^{ \abs{v_1} x_1+ u_2 x_2} [x_1^{n_1} x_2^{n_2}] \]
which implies by linearity
\begin{equation}\label{laplace-Taylor-first} \abs{ \int_0^{2\pi}   e^{ x_1 \ex(\theta) \cdot v_1   + x_2 \sum_{m=2}^{\infty}  \ex(m \theta) \cdot v_m/m }  \frac{d\theta}{2\pi} [x_1^{n_1} x_2^{n_2}] } \leq  e^{ \abs{v_1} x_1+ u_2 x_2} [x_1^{n_1} x_2^{n_2}] .\end{equation}

From Lemma \ref{weird-ps-comparison} we obtain
\begin{equation}\label{laplace-Taylor-second} \begin{split} 
& \abs{ \log \int_0^{2\pi}   e^{ x_1 \ex(\theta) \cdot v_1   + x_2 \sum_{m=2}^{\infty}  \ex(m \theta) \cdot v_m/m }  \frac{d\theta}{2\pi} [x_1^{n_1} x_2^{n_2}] }\\  \leq&  - \log (2 - e^{ \abs{v_1} x_1+ u_2 x_2} ) [x_1^{n_1} x_2^{n_2}].\end{split} \end{equation}

We have (evaluating a power series at $x_1=1,x_2=1$)
\[ \log  \int_0^{2\pi}   e^{ \sum_{m=1}^{\infty}  \ex(m \theta) \cdot v_m/m }     \frac{d\theta}{2\pi}   = \sum_{n_1=0}^{\infty} \sum_{n_2=0}^{\infty}  \log \int_0^{2\pi}   e^{ x_1 \ex(\theta) \cdot v_1   + x_2 \sum_{m=2}^{\infty}  \ex(m \theta) \cdot v_m/m }  \frac{d\theta}{2\pi} [x_1^{n_1} x_2^{n_2}] \]
and (evaluating a power series at $x_1=1,x_2=0$)
\[ \log \int_0^{\pi} e^{ \ex(\theta) \cdot v_1} \frac{d\theta}{2\pi} =  \sum_{n_1=0}^{\infty}  \log \int_0^{2\pi}   e^{ x_1 \ex(\theta) \cdot v_1   + x_2 \sum_{m=2}^{\infty}  \ex(m \theta) \cdot v_m/m }  \frac{d\theta}{2\pi} [x_1^{n_1} x_2^{0}] \]
so that
\begin{equation}\label{laplace-Taylor-third} \begin{split} &\log  \int_0^{2\pi}   e^{ \sum_{m=1}^{\infty}  \ex(m \theta) \cdot v_m/m }     \frac{d\theta}{2\pi}  - \log \int_0^{\pi} e^{ \ex(\theta) \cdot v_1} \frac{d\theta}{2\pi}\\  = \sum_{n_1=0}^{\infty} \sum_{n_2=1}^{\infty} & \log \int_0^{2\pi}   e^{ x_1 \ex(\theta) \cdot v_1   + x_2 \sum_{m=2}^{\infty}  \ex(m \theta) \cdot v_m/m }  \frac{d\theta}{2\pi} [x_1^{n_1} x_2^{n_2}] .\end{split} \end{equation}

We split the sum in \eqref{laplace-Taylor-third} up into terms with $n_1 + n_2 \leq 3$, which we evaluate, and the terms with $n_1 + n_2 \geq 4$, which we bound.

For any $b_1 \in [0, \abs{v_1}]$, we observe that $(b_1,u_2)$ lies in the compact set $\{ (a_1,a_2)\in \mathbb R^2 \mid a_1\geq 0, a_2 \geq 0, a_1+a_2 \leq \frac{1}{2} \}$ where the function $\frac{\partial}{\partial a_2} \log ( 2- e^{ a_1 +a_2})$ is smooth. Thus by Taylor's theorem we have
\begin{equation}\label{laplace-Taylor-fourth}  \sum_{\substack{n_1, n_2 \geq 0 \\ n_1+n_2 \geq 3}}  \frac{\partial}{\partial a_2} \log (2 - e^{a_1 + a_2 } ) [a_1^{n_1} a_2^{n_2}]  \abs{v_1}^{n_1} b_2^{n_2} = O( ( \abs{v_1} +b_2 )^3) = O( (\abs{v_1} + u_2)^3) = O(u_1^3) \end{equation} since the left-hand side of \eqref{laplace-Taylor-fourth} is the error in the second-order Taylor approximation to the function $  \frac{\partial}{\partial a_2} \log (2 - e^{a_1 + a_2 } ) $ at the point $ \abs{v_1}, b_2$ and the constant in Taylor's theorem is uniform by compactness.

By \eqref{laplace-Taylor-second} and \eqref{laplace-Taylor-fourth} we have 
 \[ \abs{ \sum_{\substack{n_1\geq 0, n_2 > 0 \\ n_1+n_2 \geq 4}}  \log \int_0^{2\pi}   e^{ x_1 \ex(\theta) \cdot v_1   + x_2 \sum_{m=2}^{\infty}  \ex(m \theta) \cdot v_m/m }  \frac{d\theta}{2\pi} [x_1^{n_1} x_2^{n_2}] } \]
 \[ \leq \sum_{\substack{n_1\geq 0, n_2 > 0\\ n_1+n_2 \geq 4}} \abs{ \log \int_0^{2\pi}   e^{ x_1 \ex(\theta) \cdot v_1   + x_2 \sum_{m=2}^{\infty}  \ex(m \theta) \cdot v_m/m }  \frac{d\theta}{2\pi}} [x_1^{n_1} x_2^{n_2}]\]
 \[ \leq -  \sum_{\substack{n_1\geq 0, n_2 > 0 \\ n_1+n_2 \geq 4}} \log (2 - e^{ \abs{v_1} x_1+ u_2 x_2} ) [x_1^{n_1} x_2^{n_2}]=  -  \sum_{\substack{n_1\geq 0, n_2 > 0 \\ n_1+n_2 \geq 4}} \log (2 - e^{a_1 + a_2 } ) [a_1^{n_1} a_2^{n_2}]  \abs{v_1}^{n_1} u_2^{n_2}  \]
\[  = - \sum_{\substack{n_1\geq 0, n_2 > 0\\ n_1+n_2 \geq 4}} \log (2 - e^{a_1 + a_2 } ) [a_1^{n_1} a_2^{n_2}]  \int_0^{u_2 } n_2 \abs{v_1}^{n_1} b_2^{n_2-1} db_2 \] \[   = - \int_0^{u_2} n_2  \sum_{\substack{n_1\geq 0, n_2 > 0 \\ n_1+n_2 \geq 4}} \log (2 - e^{a_1 + a_2 } ) [a_1^{n_1} a_2^{n_2}]  \abs{v_1}^{n_1} b_2^{n_2-1} db_2  \]
\[  = - \int_0^{u_2}   \sum_{\substack{n_1, n_2 \geq 0 \\ n_1+n_2 \geq 3}}  \frac{\partial}{\partial a_2} \log (2 - e^{a_1 + a_2 } ) [a_1^{n_1} a_2^{n_2}] \abs{v_1}^{n_1} b_2^{n_2} db_2  = \int_0^{u_2} O( u_1^3)  db_1 = O( u_1^3 u_2) .\]

We evaluate the terms with $n_1 + n_2 \leq 3$ by Taylor expanding each term and integrating to obtain
\[  \int_0^{2\pi}   e^{ x_1 \ex(\theta) \cdot v_1   + x_2 \sum_{m=2}^{\infty}  \ex(m \theta) \cdot v_m/m }  \frac{d\theta}{2\pi}\] \[\hspace{-.3in}= 1 +   \frac{\abs{v_1}^2}{4} x_1^2 + \sum_{m=2}^{\infty} \frac{ \abs{v_m}^2}{4m^2} x_2^2 +  \frac{ (v_1 \times v_1) \cdot v_2}{16} x_1^2 x_2 +  \sum_{m=2}^{\infty} \frac{ (v_1 \times v_m) \cdot v_{m+1}}{ 8 m (m+1)} x_1 x_2^2+ \sum_{m_1,m_2=2}^{\infty} \frac{ (v_{m_1} \times v_{m_2}) \cdot v_{m_1+m_2}}{ 16 m_1 m_2(m_1+m_2)} x_2^3 + \dots \]
Taking logarithms of both sides, we obtain
\[\log  \int_0^{2\pi}   e^{ x_1 \ex(\theta) \cdot v_1   + x_2 \sum_{m=2}^{\infty}  \ex(m \theta) \cdot v_m/m }  \frac{d\theta}{2\pi}\] \[=   \frac{\abs{v_1}^2}{4} x_1^2 + \sum_{m=2}^{\infty} \frac{ \abs{v_m}^2}{4m^2} x_2^2 +  \frac{ (v_1 \times v_1) \cdot v_2}{16} x_1^2 x_2 +  \sum_{m=2}^{\infty} \frac{ (v_1 \times v_m) \cdot v_{m+1}}{ 8 m (m+1)}x_1 x_2^2 + \sum_{m_1,m_2=2}^{\infty} \frac{ (v_{m_1} \times v_{m_2}) \cdot v_{m_1+m_2}}{ 16 m_1 m_2(m_1+m_2)} x_2^3 + \dots \]
and ignoring the terms with exponent of $x_2$ zero and then substituting $x_1 ,x_2=1$, we obtain 
\[  \sum_{\substack{n_1\geq 0, n_2 > 0 \\ n_1+n_2 \leq 3}}  \log \int_0^{2\pi}   e^{ x_1 \ex(\theta) \cdot v_1   + x_2 \sum_{m=2}^{\infty}  \ex(m \theta) \cdot v_m/m }  \frac{d\theta}{2\pi} [x_1^{n_1} x_2^{n_2}] \]\[ = \sum_{m=2}^{\infty} \frac{ \abs{v_m}^2}{4m^2} +  \frac{ (v_1 \times v_1) \cdot v_2}{16} +  \sum_{m=2}^{\infty} \frac{ (v_1 \times v_m) \cdot v_{m+1}}{ 4 m (m+1)} + \sum_{m_1,m_2=2}^{\infty} \frac{ (v_{m_1} \times v_{m_2}) \cdot v_{m_1+m_2}}{ 8 m_1 m_2(m_1+m_2)}.  \]
This gives \eqref{eq-laplace-Taylor} once we check that
\[\sum_{m=2}^{\infty} \frac{ (v_1 \times v_m) \cdot v_{m+1}}{ 4 m (m+1)} + \sum_{m_1,m_2=2}^{\infty} \frac{ (v_{m_1} \times v_{m_2}) \cdot v_{m_1+m_2}}{ 8 m_1 m_2(m_1+m_2)}  = O( u_1 u_2 u_3) \]
which is clear since

\[\abs{\sum_{m=2}^{\infty} \frac{ (v_1 \times v_m) \cdot v_{m+1}}{ 4 m (m+1)} + \sum_{m_1,m_2=2}^{\infty} \frac{ (v_{m_1} \times v_{m_2}) \cdot v_{m_1+m_2}}{ 8 m_1 m_2(m_1+m_2)}  } \]
\[ \leq \sum_{m=2}^{\infty} \frac{ \abs{v_1} \abs{v_m} \abs{v_{m+1}}}{ 4 m (m+1)} + \sum_{m_1,m_2=2}^{\infty} \frac{\abs{v_{m_1}} \abs{v_{m_2}} \abs{v_{m_1+m_2}} }{ 8 m_1 m_2(m_1+m_2)} \]
\[ \leq \sum_{m_1=1}^{\infty} \sum_{m_2=2}^{\infty} \frac{\abs{v_{m_1}} \abs{v_{m_2}} \abs{v_{m_1+m_2}} }{ 4 m_1 m_2(m_1+m_2)}\]
\[ \leq \sum_{m_1=1}^{\infty} \sum_{m_2=2}^{\infty} \sum_{m_3=3}^{\infty} \frac{ \abs{v_{m_1}} \abs{v_{m_2}} \abs{v_{m_3}}}{4m_1m_2m_3} = \frac{ u_1 u_2 u_3}{4} .\]\end{proof}

\begin{lemma}\label{laplace-unified}  There exists $\delta_2>0$ and $C_1\geq 1$ such that, for $(v_m)_{m=1}^{\infty}$ a sequence of vectors in $\mathbb R^2$ with $\sum_{m=1}^{\infty} \abs{v_m}^2<\infty$, we have
\begin{equation}\label{laplace-unified-bound} \log  \int_0^{2\pi}   e^{ \sum_{m=1}^{\infty}  \ex(m \theta) \cdot v_m/m }     \frac{d\theta}{2\pi}  \leq  \frac{ \abs{v_1}^2}{4} - \delta_2 \min ( \abs{v_1}^4, \abs{v_1}^2) + \sum_{m=2}^{\infty} \min ( C_1 \abs{v_m}^2, \abs{v_m}/m ). \end{equation} \end{lemma}

\begin{proof} A key fact we will use multiple times is that replacing $v_m$ by $0$ decreases the left hand side of \eqref{laplace-unified-bound} by at most $\abs{v_m}/m$, since it shrinks the integrand $ e^{ \sum_{m=1}^{\infty}  \ex(m \theta) \cdot v_m/m }   $ at each point by a factor of at most $e^{ \abs{v_m}/m}$ and thus, because the integrand is positive, shrinks the integral by a factor of at most $e^{ \abs{v_m}/m}$.

For $m\geq 2$, if $\abs{v_m}  \geq \frac{1}{mC_1}$ then $\min ( C_1 \abs{v_m}^2, \abs{v_m}/m ) = \abs{v_m}/m$. If we replace $v_m$ by $0$, the left side of \eqref{laplace-unified-bound} decreases by at most $\abs{v_m}/m$ while the right side decreases by exactly $\abs{v_m}/m$ so the bound after making the change implies the bound before. Repeating this for all $m$, we may assume \begin{equation}\label{vm-assumption} \abs{v_m}< \frac{1}{m C_1}\end{equation} for all $m\geq 2$.

Take $\delta_1$ as in \cref{laplace-ov} and then set  $\delta_2 =\delta_1/2 $ so that by \cref{laplace-ov} we have \[ \log \int_0^{\pi} e^{ \ex(\theta) \cdot v_1} \frac{d\theta}{2\pi} \leq \frac{ \abs{v_1}^2}{4} - 2\delta_2 \min ( \abs{v_1}^4, \abs{v_1}^2) \] so if we let 
\[ \operatorname{Disc} (v_1,v_2,\dots) =  \log  \int_0^{2\pi}   e^{ \sum_{m=1}^{\infty}  \ex(m \theta) \cdot v_m/m }     \frac{d\theta}{2\pi} - \log \int_0^{\pi} e^{ \ex(\theta) \cdot v_1} \frac{d\theta}{2\pi} \]
then it suffices to check for $C_1$ sufficiently large that
\begin{equation}\label{l-u-s} \operatorname{Disc} (v_1,v_2,\dots)    \leq  \delta_2 \min ( \abs{v_1}^4, \abs{v_1}^2) +   \sum_{m=2}^{\infty} \min ( C_1 \abs{v_m}^2, \abs{v_m}/m ). \end{equation}
We  have \[ \operatorname{Disc} (v_1,v_2,\dots)  \leq \sum_{m=2}^{\infty} \abs{v_m}/m \leq \frac{1}{C_1} (\zeta(2)-1)\] by the key fact and \eqref{vm-assumption}. Thus we may assume that
\begin{equation}\label{v1-assumption} \delta_2 \min ( \abs{v_1}^4, \abs{v_1}^2)  \leq  \frac{1}{C_1} (\zeta(2)-1). \end{equation} because otherwise \eqref{l-u-s} holds automatically. Combining \eqref{vm-assumption} and \eqref{v1-assumption} gives 
\begin{equation}\label{vm-sum-bound} \sum_{m=1}^{\infty} \frac{ \abs{v_m} }{m} = \abs{v_1} + \sum_{m=2}^{\infty} \frac{ \abs{v_m} }{m}\leq  \max \Biggl( \Bigl(  \delta_2^{-1} \frac{1}{C_1} (\zeta(2)-1)\Bigr)^{1/4}, \Bigl(  \delta_2^{-1} \frac{1}{C} (\zeta(2)-1)\Bigr)^{1/4} \Biggr) + \frac{1}{C_1} (\zeta(2)-1)  \end{equation} and choosing $C_1$ sufficiently large, the right hand side of \eqref{vm-sum-bound} is $<\frac{1}{2}$, and thus we may apply \cref{laplace-Taylor}, obtaining
\begin{equation}\label{laplace-Taylor-concluded}  \operatorname{Disc} (v_1,v_2,\dots) =  \sum_{m=2}^{\infty} \frac{ \abs{v_m}^2}{4m^2} +  \frac{ (v_1 \times v_1) \cdot v_2}{16} + O( u_1 u_2 u_3 + u_1^3 u_2).\end{equation}
We now simplify \eqref{laplace-Taylor-concluded} by bounding the terms appearing on the right hand side. To do this, we we use the facts clear from the definitions that $u_1 = \abs{v_1} + u_2$ and $u_3\leq u_2$ as well as the assumptions \eqref{vm-assumption} and \eqref{v1-assumption} that imply $u_2$ and $\abs{v_1}$, respectively, are bounded by constants. These facts imply
\[   \frac{ (v_1 \times v_1) \cdot v_2}{16}  \ll \abs{v_1}^2 \abs{v_2} \leq \abs{v_1}^2 u_2 \]
\[ O(u_1u_2 u_3) \ll u_1 u_2 u_3 \leq u_1 u_2^2 = (\abs{v_1}+ u_2) u_2^2 \ll u_2^2 \]
\[ O(u_1^3 u_2) \ll u_1^3 u_2 = (\abs{v_1} + u_2)^3 u_2 = \abs{v_1}^3 u_2 +3 \abs{v_1}^2 u_2^2 + 3 \abs{v_1}u_2^3 + u_2^4 \ll \abs{v_1}^2 u_2 + u_2^2\]
giving
\[ \operatorname{Disc} (v_1,v_2,\dots) =  \sum_{m=2}^{\infty} \frac{ \abs{v_m}^2}{4m^2} +  O (  \abs{v_1}^2 u_2 + u_2^2) . \]
Applying the completing-the-square-bound $\abs{v_1}^2 u_2 \leq \epsilon \abs{v_1}^4 + \frac{1}{4\epsilon} u_2^2$ for some sufficiently small $\epsilon$, we obtain
\[ \operatorname{Disc} (v_1,v_2,\dots) \leq  \sum_{m=2}^{\infty} \frac{ \abs{v_m}^2}{4m^2}    + \delta_2 \abs{v_1}^4+ O(u_2^2) .\]
Applying the Cauchy-Schwarz bound $u_2^2 \leq (\zeta(2)-1) \sum_{m=2}^{\infty} \abs{v_m}^2$ we obtain
\begin{equation}\label{l-u-final} \operatorname{Disc} (v_1,v_2,\dots) \leq  \delta_2 \abs{v_1}^4 + O (  \sum_{m=2}^{\infty} \abs{v_m}^2)  .\end{equation}
Finally \eqref{l-u-final} implies \eqref{l-u-s} because \eqref{v1-assumption} gives $\abs{v_1}\leq 1$ (for $C_1$ sufficiently large) so $\delta_2 \abs{v_1}^4 = \delta_2 \min ( \abs{v_1}^4, \abs{v_1}^2)$ and \eqref{vm-assumption} gives $ \sum_{m=2}^{\infty} \min ( C_1\abs{v_m}^2, \abs{v_m}/m )= \sum_{m=2}^\infty C_1 \abs{v_m}^2$ which dominates any expression of the form $O (  \sum_{m=2}^{\infty} \abs{v_m}^2)$ as long as $C_1$ is sufficiently large. \end{proof}

\subsection{Complex exponential integrals}

This subsection is devoted to bounding the integral  $\int_0^{2\pi}   e^{ \sum_{m=1}^{\infty}  \ex(m \theta) \cdot v_m/m + i\sum_{m=1}^{\infty}  \ex(m \theta) \cdot w_m/m}$.  We have three different estimates that roughly handle three different ranges for $\abs{w_1}$. When $\abs{w_1}$ is small we will apply Lemma \ref{hybrid-short-range} which is proven using a Taylor series argument. When $\abs{w_1}$ is large we will apply Lemma \ref{hybrid-long-range} which is proven using a stationary phase argument. When $\abs{w_1}$ is intermediate we will apply Lemma \ref{hybrid-medium-range} which is proven using a more elementary argument involving the range of values attained by the function $\sum_{m=1}^{\infty}  \ex(m \theta) \cdot w_m/m$. 

We then multiply the bounds together to obtain bounds for the expectation $\mathbb E [ e^{ \sum_{n=1}^k X_{n,\xi} \cdot v_n  +  i \sum_{n=1}^k X_{n,\xi} \cdot w_n } ] $, with the final bounds relevant to the remainder of the argument contained in Corollary \ref{hybrid-unified}.

\begin{lemma}\label{hybrid-short-range} Fix $\delta_3< \frac{1}{64}$. Let $(v_m)_{m=1}^\infty$ and $(w_m)_{m=1}^\infty$ be sequences of vectors in $\mathbb R^2$. If $ \abs{v_1} + \abs{w_1} + \sum_{m=2}^{\infty} \frac{ \sqrt{ \abs{v_m}^2 + \abs{w_m}^2} }{m} <  \frac{1}{2} $ then 
\begin{equation} \label{eq-hybrid-short-range} \begin{split} & \log  \abs{  \int_0^{2\pi}   e^{ \sum_{m=1}^{\infty}  \ex(m \theta) \cdot v_m/m + i\sum_{m=1}^{\infty}  \ex(m \theta) \cdot w_m/m }     \frac{d\theta}{2\pi}} \\   \leq & \frac{ \abs{v_1}^2}{4} - \frac{ \abs{w_1}^2}{4} - \delta_3 \abs{v_1}^4 + O_{\delta_3}  ( \abs{v_1}^6 + \abs{w_1}^3 + \sum_{m=2}^{\infty} (\abs{v_m}^2+ \abs{w_m}^2 )).\end{split}\end{equation}\end{lemma}

\begin{proof}  Let $W =  \sum_{m=2}^{\infty} \frac{ \sqrt{ \abs{v_m}^2 + \abs{w_m}^2} }{m} $. For any $\theta$,  
\[e^{\lambda \ex(m \theta) \cdot v_1 + i\lambda^2 \ex(\theta) \cdot w_1+ \lambda^3 \sum_{m=1}^{\infty}  \ex(m \theta) \cdot v_m/m + i\lambda^3 \sum_{m=1}^{\infty}  \ex(m \theta) \cdot w_m/m }\]  is a power series in $\lambda$ whose coefficient of $\lambda^n$ is bounded by the coefficient of $\lambda^n$ in  \begin{equation}\label{hsr-exp} e^{\lambda \abs{v_1} + \lambda^2 \abs{w_1} + \lambda^3 W}. \end{equation}

Hence the coefficient of $\lambda^n$ in \begin{equation}\label{taylor-hybrid-unlogged} \int_0^{2\pi} e^{\lambda \ex(m \theta) \cdot v_1 + i\lambda^2 \ex(\theta) \cdot w_1+ \lambda^3 \sum_{m=1}^{\infty}  \ex(m \theta) \cdot v_m/m + i\lambda^3 \sum_{m=1}^{\infty}  \ex(m \theta) \cdot w_m/m }\frac{d\theta}{2\pi} \end{equation} is also bounded by the coefficient of $\lambda^n$ in \eqref{hsr-exp}.

By Lemma \ref{weird-ps-comparison}, the coefficient of $\lambda^n$ in \begin{equation}\label{Taylor-hybrid-ps}  \log e^{\lambda \ex(m \theta) \cdot v_1 + i\lambda^2 \ex(\theta) \cdot w_1+ \lambda^3 \sum_{m=1}^{\infty}  \ex(m \theta) \cdot v_m/m + i\lambda^3 \sum_{m=1}^{\infty}  \ex(m \theta) \cdot w_m/m } \frac{d\theta}{2\pi} \end{equation}  is bounded by the coefficient of $\lambda^n$ in the power series  \[  - \log (2 - e^{\lambda \abs{v_1}+ \lambda^2 \abs{w_1} + \lambda^3 W}).\]  We now observe that Taylor's theorem applied to the function $- \log ( 2 -e^{ x +y^2+z^3 })$ implies that the sum of all terms of degree $\geq 6$ appearing in the power series for $- \log ( 2 -e^{ x +y^2+z^3 })$ is $O ( x^6+y^6+z^6)$ uniformly for $x,y,z\geq 0$ such that $x+y^2+z^3 \leq 1/2$. Indeed such $x,y,z$ lie in a compact region where the function is smooth so all derivatives are bounded. Plugging in $x=\lambda  \abs{v_1}, y =  \lambda\abs{w_1}^{\frac{1}{2}}, z =  \lambda W^{\frac{1}{3}}$ and then setting $\lambda=1$, we obtain the sum of the coefficients of $\lambda^n$ in \eqref{Taylor-hybrid-ps} for $n$ from $6$ to $\infty$. Hence the sum of the coefficient of $\lambda^n$ in \eqref{Taylor-hybrid-ps} for $n$ from $6$ to $\infty$ is \[ O (  \abs{v_1}^6 + \abs{w_1}^3 + W^2) = O ( \abs{v_1}^6 + \abs{w_1}^3 + \sum_{m=2}^{\infty} (\abs{v_m}^2+ \abs{w_m}^2 )). \]  

On the other hand, the coefficients of $\lambda, \lambda^2, \lambda^3,\lambda^4, \lambda^5$ in \eqref{taylor-hybrid-unlogged} are, respectively,
\[ 0\]
\[  \frac{\abs{v_1}^2}{4} \]
\[   i \frac{ v_1 \cdot w_1}{2}\]
\[   \frac{ \abs{v_1}^4}{64}  - \frac{\abs{w_1}^2}{4} \]
\[  \frac{ (v_1 \times v_1) \cdot v_2}{ 16}+ i \frac{ (v_1 \times v_1) \cdot w_2}{16} + i \frac{ \abs{v_1}^2 v_1 \cdot w_1}{16} .\]
Taking logarithms, the coefficients of $\lambda, \lambda^2, \lambda^3, \lambda^4, \lambda^5$ in \eqref{Taylor-hybrid-ps} are identical except the coefficient of $\lambda^4$ is
\[  - \frac{ \abs{v_1}^4}{64}  - \frac{\abs{w_1}^2}{4} .\]

Plugging $\lambda=1$ into \eqref{Taylor-hybrid-ps} and taking the real part, this gives
\[  \log  \abs{  \int_0^{2\pi}   e^{ \sum_{m=1}^{\infty}  \ex(m \theta) \cdot v_m/m + i\sum_{m=1}^{\infty}  \ex(m \theta) \cdot w_m/m }     \frac{d\theta}{2\pi}} \] \[= \frac{ \abs{v_1}^2}{4} - \frac{\abs{w_1}^2}{4} - \frac{ \abs{v_1}^4}{64} + \frac{ (v_1 \times v_1) \cdot v_2}{16}  +O (\abs{v_1}^6 + \abs{w_1}^3 + \sum_{m=2}^{\infty} (\abs{v_m}^2+ \abs{w_m}^2 )).\]
which is exactly the desired bound \eqref{eq-hybrid-short-range} except for the term $ \frac{ (v_1 \times v_1) \cdot v_2}{16} $, which can be controlled by observing that
\[  \abs{ \frac{ (v_1 \times v_1) \cdot v_2}{16} }\leq \frac{ \abs{v_1}^2 \abs{v_2}}{16} \leq \left( \frac{1}{64}-\delta_3\right) \abs{v_1}^4 +  \frac{ \frac{1}{1024}}{ \frac{1}{64} - \delta_3 } \abs{v_2}^2=\left( \frac{1}{64}-\delta_3\right) \abs{v_1}^4 + O_{\delta_3}( \abs{v_2}^2) .  \]  \end{proof}

\begin{lemma}\label{hybrid-medium-range}   Let $(v_m)_{m=1}^\infty$ and $(w_m)_{m=1}^\infty$ be sequences of vectors in $\mathbb R^2$. If $\sum_{m=1}^\infty \abs{v_m}/m<\infty$ and $\sum_{m=1}^{\infty} \abs{w_m}^2 < \infty$ then
\[\frac{ \abs{  \int_0^{2\pi}   e^{ \sum_{m=1}^{\infty}  \ex(m \theta) \cdot v_m/m + i\sum_{m=1}^{\infty}  \ex(m \theta) \cdot w_m/m }     \frac{d\theta}{2\pi}} }{\abs{  \int_0^{2\pi}   e^{ \sum_{m=1}^{\infty}  \ex(m \theta) \cdot v_m/m }     \frac{d\theta}{2\pi}} } \] \[ \leq 1- e^{ -2 \sum_{m=1}^{\infty} \abs{v_m}/m}  \min \left(  \sum_{m=1}^{\infty} \frac{\abs{w_m}^2}{\pi^2 m^2}, \frac{1}{ 2\sum_{m=1}^{\infty} \abs{w_m}^2} \right).\] \end{lemma}

\begin{proof} Let $\phi$ be the argument of $\int_0^{2\pi}   e^{ \sum_{m=1}^{\infty}  \ex(m \theta) \cdot v_m/m + i\sum_{m=1}^{\infty}  \ex(m \theta) \cdot w_m/m }     \frac{d\theta}{2\pi}$. We have
\[  \abs{  \int_0^{2\pi}   e^{ \sum_{m=1}^{\infty}  \ex(m \theta) \cdot v_m/m + i\sum_{m=1}^{\infty}  \ex(m \theta) \cdot w_m/m } \frac{d\theta}{2\pi}}\] \[= \Re \int_0^{2\pi} e^{ -i \phi + \sum_{m=1}^{\infty}  \ex(m \theta) \cdot v_m/m + i\sum_{m=1}^{\infty}  \ex(m \theta) \cdot w_m/m }     \frac{d\theta}{2\pi} \]
\[ = \int_0^{2\pi} e^{ \sum_{m=1}^{\infty} \ex(m\theta) \cdot v_m/m} \cos\Bigl( -\phi + \sum_{m=1}^{\infty} \ex(m\theta) \cdot w_m /m \Bigr) .\]

Using the bound \[ e^{ \sum_{m=1}^{\infty}  \ex(m \theta) \cdot v_m/m}  \in [e^{ - \sum_{m=1}^{\infty} \abs{v_m}/ m}, e^{ \sum_{m=1}^{\infty}  \abs{v_m} /m } ]\] valid for all $\theta$, we obtain in particular
\[ \abs{  \int_0^{2\pi}   e^{ \sum_{m=1}^{\infty}  \ex(m \theta) \cdot v_m/m }     \frac{d\theta}{2\pi}} \leq  e^{ \sum_{m=1}^{\infty}  \abs{v_m} /m } \]
so that it suffices to prove
\begin{equation}\label{hmr-suffices} \begin{split} &  \int_0^{2\pi} e^{ \sum_{m=1}^{\infty} \ex(m\theta) \cdot v_m/m} \cos\Bigl ( -\phi + \sum_{m=1}^{\infty} \ex(m\theta) \cdot w_m /m\Bigr ) \\ &  \leq \int_0^{2\pi}   e^{ \sum_{m=1}^{\infty}  \ex(m \theta) \cdot v_m/m }     \frac{d\theta}{2\pi} -  e^{ - \sum_{m=1}^{\infty} \abs{v_m}/m} \min \Bigl(  \sum_{m=1}^{\infty} \frac{\abs{w_m}^2}{\pi^2 m^2}, \frac{1}{ 2\sum_{m=1}^{\infty} \abs{w_m}^2} \Bigr)\end{split} \end{equation}

We split into two cases depending on whether $ e^{i \sum_{m=1}^{\infty}  \ex(m \theta) \cdot w_m/m} = - e^{i \phi}$ for some $\theta$ or not.

First, suppose $e^{i \sum_{m=1}^{\infty}  \ex(m \theta_0) \cdot w_m/m} = - e^{i \phi}$ for some $\theta_0$. Let $x =\sum_{m=1}^{\infty}  \ex(m \theta_0) \cdot w_m/m$ and let $I$ be the longest interval around $\theta_0$ on which \[ \sum_{m=1}^{\infty}  \ex(m \theta) \cdot w_m/m \in [x-\pi/2, x+\pi/2].\] Then on the boundary of $I$, we have $\sum_{m=1}^{\infty}  \ex(m \theta) \cdot w_m/m = x \pm \pi/2$ while at $\theta_0 \in I$ it takes the value $x$ so
\[ \pi \leq \int_I  \abs{ \frac{d}{d\theta} \sum_{m=1}^{\infty}  \ex(m \theta) \cdot w_m/m } d \theta = \int_I  \abs{ \sum_{m=1}^{\infty}  \ex(m \theta+ \frac{\pi}{2}) \cdot w_m } d \theta 
\] \[ \leq  \sqrt{ \abs{I} \int_0^{2\pi}  \abs{ \sum_{m=1}^{\infty}  \ex(m \theta+ \frac{\pi}{2})  \cdot w_m }^2 d \theta } = \sqrt{ \abs{I}  \pi \sum_{m=1}^{\infty} \abs{w_m}^2} \]
 so $\abs{I} \geq \frac{ \pi} { \sum_{m=1}^{\infty} \abs{w_m}^2} $ but for each $\theta \in I$ we have \[\cos\Bigl ( -\phi + \sum_{m=1}^{\infty} \ex(m\theta) \cdot w_m /m\Bigr )\leq 0\] so
 \[  \int_0^{2\pi} e^{ \sum_{m=1}^{\infty} \ex(m\theta) \cdot v_m/m} \cos\Bigl ( -\phi + \sum_{m=1}^{\infty} \ex(m\theta) \cdot w_m /m\Bigr )\]
 \[ \leq \int_0^{2\pi} e^{ \sum_{m=1}^{\infty}  \ex(m \theta) \cdot v_m/m }     \frac{d\theta}{2\pi} - \int_Ie^{ \sum_{m=1}^{\infty}  \ex(m \theta) \cdot v_m/m }     \frac{d\theta}{2\pi}\]
 \[ \leq  \int_0^{2\pi} e^{ \sum_{m=1}^{\infty}  \ex(m \theta) \cdot v_m/m }     \frac{d\theta}{2\pi} - \frac{\abs{I} }{2\pi} e^{ - \sum_{m=1}^{\infty} \abs{v_m}/m} \]
 \[ \leq  \int_0^{2\pi} e^{ \sum_{m=1}^{\infty}  \ex(m \theta) \cdot v_m/m }     \frac{d\theta}{2\pi} - \frac{1 }{2 \sum_{m=1}^{\infty} \abs{w_m}^2 } e^{ - \sum_{m=1}^{\infty} \abs{v_m}/m} \]
giving \eqref{hmr-suffices}.
 
 Next suppose that $e^{\sum_{m=1}^{\infty}  \ex(m \theta) \cdot w_m/m} \neq  - e^{i \phi}$ for any $\theta$, or in other words $\sum_{m=1}^{\infty}  \ex(m \theta) \cdot w_m/m \neq \phi + \pi n$ for any odd integer $n$. After shifting $\phi$ by an even integer multiple of $2\pi$, we may assume $\sum_{m=1}^{\infty}  \ex(m \theta) \cdot w_m/m \in (\phi-\pi, \phi+\pi)$ for all $\theta$. The simple trigonometric inequality $\cos \theta\leq   1 - \frac{2 \theta^2}{\pi^2}$ for $\theta \in (-\pi,\pi)$ gives
  \[  \int_0^{2\pi} e^{ \sum_{m=1}^{\infty} \ex(m\theta) \cdot v_m/m} \cos\Bigl ( -\phi + \sum_{m=1}^{\infty} \ex(m\theta) \cdot w_m /m\Bigr )\]
 \[\leq \int_0^{2 \pi} e^{ \sum_{m=1}^{\infty} \ex(m \theta) \cdot v_m/m}  \left(1 - \frac{2}{\pi^2} \left( \phi - \sum_{m=1}^{\infty}  \ex(m \theta) \cdot w_m/m \right)^2 \right) \frac{ d \theta}{2\pi} \]
  \[ = \int_0^{2\pi} e^{ \sum_{m=1}^{\infty} \ex(m \theta) \cdot v_m/m}  \frac{d\theta}{2\pi}  - \frac{2}{\pi^2} \int_0^{2\pi}  e^{ \sum_{m=1}^{\infty} \ex(m \theta) \cdot v_m/m}\left( \phi - \sum_{m=1}^{\infty}  \ex(m \theta) \cdot w_m/m \right)^2 \frac{d\theta}{2\pi}\]
  \[ \leq \int_0^{2\pi} e^{ \sum_{m=1}^{\infty} \ex(m \theta) \cdot v_m/m}  \frac{d\theta}{2\pi} - e^{ -\sum_{m=1}^{\infty} \abs{v_m}/m} \int_0^{2\pi}\left( \phi - \sum_{m=1}^{\infty}  \ex(m \theta) \cdot w_m/m \right)^2 \frac{d\theta}{2\pi}\]
 \[ = \int_0^{2\pi} e^{ \sum_{m=1}^{\infty} \ex(m \theta) \cdot v_m/m}  \frac{d\theta}{2\pi} - e^{ -\sum_{m=1}^{\infty} \abs{v_m}/m} \frac{1}{\pi^2} \left( \sum_{m=1}^{\infty} \frac{\abs{w_m}^2}{m^2} + \phi^2\right)\] \[ \leq \int_0^{2\pi} e^{ \sum_{m=1}^{\infty} \ex(m \theta) \cdot v_m/m}  \frac{d\theta}{2\pi} - e^{ -\sum_{m=1}^{\infty} \abs{v_m}/m}  \frac{1}{\pi^2} \sum_{m=1}^{\infty}  \frac{\abs{w_m}^2}{m^2}\]  
 giving \eqref{hmr-suffices}. \end{proof}

\begin{lemma}\label{hybrid-long-range}  Let $(v_m)_{m=1}^\infty$ and $(w_m)_{m=1}^\infty$ be sequences of vectors in $\mathbb R^2$. If $\sum_{m=1}^\infty  \abs{v_m}^2 + \sum_{m=1}^{\infty} \abs{w_m}^2 < \infty$ then
\begin{equation}\label{eq-hybrid-long-range} \begin{split} & \frac{ \abs{  \int_0^{2\pi}   e^{ \sum_{m=1}^{\infty}  \ex(m \theta) \cdot v_m/m + i\sum_{m=1}^{\infty}  \ex(m \theta) \cdot w_m/m }     \frac{d\theta}{2\pi}}}{ \sup_{\theta \in [0,2\pi]}   e^{ \sum_{m=1}^{\infty}  \ex(m \theta) \cdot v_m/m } }\\ \leq & \frac{1}{\sqrt{\abs{w_1}}} \left( \frac{4}{\pi} + 1 +   \frac{ \sqrt{ \sum_{m=1}^{\infty} \abs{v_m}^2 +  \sum_{m=2}^{\infty} \abs{w_m}^2}}{ \sqrt{\pi} \abs{w_1}^{1/4}} \right).\end{split} \end{equation}
\end{lemma}

 \begin{proof} By shifting $\theta$ we may assume $w_1 = (\abs{w_1},0)$.  Let 
\[ \mathcal F( \theta) = e^{ \sum_{m=1}^{\infty}  \ex(m \theta) \cdot v_m/m + i\sum_{m=2}^{\infty}  \ex(m \theta) \cdot w_m/m }  \] so that
\[  e^{ \sum_{m=1}^{\infty}  \ex(m \theta) \cdot v_m/m + i\sum_{m=1}^{\infty}  \ex(m \theta) \cdot w_m/m }  = e^{ i \abs{w_1} \cos \theta } \mathcal F(\theta)\]  and thus the integral to bound in \eqref{eq-hybrid-long-range} is $ \int_0^{2\pi}    e^{ i \abs{w_1} \cos \theta } \mathcal F(\theta)  \frac{d\theta}{2\pi}$. 

We now informally explain the strategy of proof: handle this integral by the method of stationary phase. Ignoring $\mathcal F$, one standard form of the stationary phase method is to change variables from $\theta$ to $\cos \theta$, apply integration by parts, and then reverse the change of variables, giving an integral where the derivative $\abs{w_1} \sin \theta$ appears in the denominator. Before doing this, we remove from the integral and handle separately the region where $\sin \theta$ is so small that this gives a worse bound. Since our desired bound \eqref{eq-hybrid-long-range} shrinks as $\abs{w_1}$ grows but grows in the other variables, we should think of $\abs{w_1}$ as large and the other variables as small. In other words, we think of $\mathcal F(\theta)$ as varying more slowly than $e^{ i \abs{w_1} \cos \theta }$. Because of this, we do not need to modify our change of variables strategy to account for $F$ (which would give a better bound but with a considerably more complicated formula accounting for multiple potential critical points of $\sum_{m=1}^{\infty} \ex(m\theta) \cdot w_m/m$). 
 
 
 We first handle the integral from $0$ to $\pi$.  For each $\delta$ with $0< \delta < \pi/2$ we have
\begin{equation}\label{cap-removal} \left| \int_0^{\pi}    e^{ i \abs{w_1} \cos \theta } \mathcal F(\theta) \frac{d\theta}{2\pi}  -  \int_{\delta}^{\pi-\delta }  e^{ i \abs{w_1} \cos \theta } \mathcal F(\theta) \frac{d\theta}{2\pi} \right| \leq  \frac{2 \delta}{2\pi} \norm{\mathcal F}_\infty \end{equation}
 and the change of variables $c = \cos \theta$ followed by integration by parts gives
 \begin{equation}\label{integration-by-parts} \begin{split}  \int_{\delta}^{\pi-\delta }  &e^{ i \abs{w_1} \cos \theta } \mathcal F(\theta) \frac{d\theta}{2\pi} \\
 =  \int_{- \cos \delta }^{\cos \delta  } &e^{ i \abs{w_1} c } \mathcal F(\arccos c )   \frac{1}{ 2 \pi \sqrt{1-c^2}}  dc \\
  =  - \int_{- \cos \delta }^{\cos \delta  }  &  \frac{ e^{ i \ \abs{w_1}c}}{ i\abs{ w_1} }  \frac{d}{dc} \left ( \mathcal F(\arccos c)    \frac{1}{ 2 \pi \sqrt{1-c^2}} \right)  dc  + \frac{ e^{ i \abs{w_1}c}}{ i\abs{ w_1} }   \mathcal F(\arccos c)   \frac{1}{ 2 \pi \sqrt{1-c^2}}  ]_{-\cos \delta}^{\cos \delta} \end{split} \end{equation}
 
We have \begin{equation}\label{stationary-phase-boundary} \abs{ \frac{ e^{ i \abs{w_1}c}}{ i\abs{ w_1} }  \mathcal F(\arccos c)  \frac{1}{ 2 \pi \sqrt{1-c^2}}  ]_{-\cos \delta}^{\cos \delta} } \leq  \frac{2}{ 2 \pi \abs{w_1 }\sin \delta } \norm{\mathcal F}_\infty \end{equation} since we may bound the value at $-\cos \delta$ and $\cos \delta$ separately.

Let  \begin{equation}\label{G-def} \mathcal G(\theta)  = \frac{ d\log \mathcal F(\theta)}{d\theta} =  \sum_{m=1}^{\infty} \ex(m\theta + \frac{\pi}{2}) \cdot v_m + i\sum_{m=2}^{\infty} \ex(m \theta + \frac{\pi}{2})\cdot w_m. \end{equation}
 
 We have \begin{equation}\label{stationary-phase-derivative} \begin{split} \frac{d}{dc}& \left (   \mathcal F( \arccos c)   \frac{1}{ 2 \pi \sqrt{1-c^2}} \right) = \mathcal F(\arccos c) \frac{1}{ 2 \pi \sqrt{1-c^2}}  \Bigl(  \frac{c}{1-c^2} -\frac{ \mathcal G(\arccos c) } {  \sqrt{1-c^2}} \Bigr) 
   \end{split} \end{equation} 
Respectively applying \eqref{stationary-phase-derivative} and reversing the change of variables $c= \cos \theta$, then applying Cauchy-Schwarz, and finally using the integral $\int_\delta^{\pi -\delta} \frac{1}{\sin^2 \theta} d\theta = \frac{2 \cos \delta}{\sin \delta}$ gives
\begin{equation}\label{rcv} \begin{split} &\abs{ \int_{- \cos \delta }^{\cos \delta  }    \frac{ e^{ i  \abs{w_1}c}}{ i\abs{ w_1} }  \frac{d}{dc} \left ( \mathcal F( \arccos c)   \frac{1}{ 2 \pi \sqrt{1-c^2}} \right) dc} \\
= & \abs{\int_{\delta}^{\pi-\delta} \frac{ e^{ i  \abs{w_1}\cos \theta}}{ i\abs{ w_1} }\mathcal F (\theta)  \Bigl(  \frac{\cos \theta }{\sin^2 \theta } + \frac{ \mathcal G(\theta) }{  \sin \theta } \Bigr)    \frac{d\theta}{2\pi}} \\
   \leq&  \frac{ \norm{\mathcal F}_\infty }{ \abs{w_1}}\int_{\delta}^{\pi - \delta} \abs{ \frac{\cos \theta}{\sin^2 \theta}} \frac{d\theta}{2\pi} + \frac{\norm{\mathcal F}_\infty }{\abs{w_1}} \sqrt{  \int_{\delta}^{\pi -\delta} \frac{1}{ \sin^2 \theta} \frac{d\theta}{2\pi} \cdot \int_{\delta}^{\pi-\delta}\abs{\mathcal G(\theta)}^2 \frac{d\theta}{2\pi } } \\
  \leq  &\frac{\cos \delta  \norm{\mathcal F}_\infty  }{ \abs{w_1} \pi \sin \delta } + \frac{ \norm{\mathcal F}_\infty }{\abs{w_1}} \sqrt{ \frac{\cos \delta}{ \pi \sin \delta} \int_{\delta}^{\pi-\delta} \left|\mathcal G(\theta) \right|^2 \frac{d\theta}{2\pi } }. \end{split} \end{equation}
 Plugging \eqref{rcv} and \eqref{stationary-phase-boundary} into \eqref{integration-by-parts}, and then applying \eqref{cap-removal}, we obtain
 \begin{equation}\label{stationary-top-side} \left| \int_0^{\pi}    e^{ i \abs{w_1} \cos \theta } \mathcal F(\theta)  \frac{d\theta}{2\pi} \right| \leq \frac{\delta \norm{\mathcal F}_\infty }{\pi} + \frac{\norm{\mathcal F}_\infty }{  \pi \abs{w_1} \sin \delta} + \frac{\cos \delta }{ \abs{w_1} \pi \sin \delta } + \frac{\norm{\mathcal F}_\infty }{\abs{w_1}} \sqrt{ \frac{\cos \delta}{ \pi \sin \delta} \int_{\delta}^{\pi-\delta}\abs{\mathcal G(\theta)}^2 \frac{d\theta}{2\pi } } .\end{equation}
 
 A symmetrical argument gives
\begin{equation}\label{stationary-bottom-side} \left| \int_\pi^{2\pi}  e^{ i \abs{w_1} \cos \theta } \mathcal F(\theta)  \frac{d\theta}{2\pi} \right| \leq \frac{\delta \norm{\mathcal F}_\infty}{\pi} + \frac{\norm{\mathcal F}_\infty}{  \pi \abs{w_1} \sin \delta} + \frac{\cos \delta \norm{\mathcal F}_\infty }{ \abs{w_1} \pi \sin \delta } + \frac{1}{\abs{w_1}} \sqrt{ \frac{\cos \delta}{ \pi \sin \delta} \int_{\pi + \delta}^{2\pi-\delta} \abs{\mathcal G(\theta)}^2 \frac{d\theta}{2\pi } } . \end{equation}
 
We have \begin{equation}\label{Gint-bound} \int_{\delta}^{\pi-\delta} \abs{\mathcal G(\theta)}^2\frac{d\theta}{2\pi }  + \int_{\pi+ \delta}^{2\pi-\delta}  \abs{\mathcal G(\theta)}^2\frac{d\theta}{2\pi }  \leq \int_0^{2\pi}  \abs{\mathcal G(\theta)}^2 \frac{d\theta}{2\pi }  =\frac{1}{2} \sum_{m=1}^{\infty} \abs{v_m}^2 +  \frac{1}{2} \sum_{m=2}^{\infty} \abs{w_m}^2 \end{equation} with the last equality using the definition \eqref{G-def} of $\mathcal G$. Combining \eqref{stationary-top-side}, \eqref{stationary-bottom-side}, and \eqref{Gint-bound} gives
\[ \frac{ \left| \int_0^{2\pi}  e^{ i \abs{w_1} \cos \theta } \mathcal F(\theta)   \frac{d\theta}{2\pi} \right| } { \norm{\mathcal F}_\infty}\]
\[ \leq \frac{2 \delta}{\pi} + \frac{2}{ \pi \abs{w_1} \sin \delta} + \frac{2 \cos \delta}{ \abs{w_1} \pi \sin \delta} + \frac{1}{\abs{w_1} } \sqrt{ \frac{ \cos \delta}{\pi \sin \delta} ( \sum_{m=1}^{\infty} \abs{v_m}^2+ \sum_{m=2}^{\infty} \abs{w_m}^2)} \]
\[ \leq \frac{2 \delta}{\pi} + \frac{1}{  \abs{w_1} \delta } + \frac{2 }{ \abs{w_1} \pi \delta } + \frac{1}{\abs{w_1} } \sqrt{ \frac{ 1}{\pi \delta }( \sum_{m=1}^{\infty} \abs{v_m}^2+ \sum_{m=2}^{\infty} \abs{w_m}^2)}. \]
Taking $\delta= \frac{1}{\sqrt{\abs{w_1}}}$ we obtain \eqref{eq-hybrid-long-range}.
 \end{proof}
 
 \begin{lemma}\label{squarification}| We have
  \[ \sum_{m=1}^{\infty} \abs{v_m}/m\leq\frac{\zeta(2)}{4} +  \sum_{m=1}^{\infty} \min ( \abs{v_m}^2, \abs{v_m}/m)  .\] \end{lemma}
  \begin{proof} For each $m$ we have  $\abs{v}_m/m \leq  \abs{v_m}^2 + \frac{1}{4m^2}$ by completing the square and thus  $\abs{v}_m/m \leq  \min ( \abs{v_m}^2, \abs{v_m}/m)   + \frac{1}{4m^2}$. Summing over $m$ gives the statement.\end{proof}

 \begin{proposition}\label{hybrid-unified} There exists $\delta_3>0$, constants $C_1, C_2>0$, and a function $\mathcal S \colon [0,\infty] \to [0,1]$ such that the following hold:

 Let $(v_m)_{m=1}^\infty$ and $(w_m)_{m=1}^\infty$ be sequences of vectors in $\mathbb R^2$. If $\sum_{m=1}^\infty m \abs{v_m}^2+ \sum_{m=1}^{\infty} \abs{w_m}^2 < \infty$ then
\begin{equation}\label{eq-hyb-un} \begin{split} &  \abs{  \int_0^{2\pi}   e^{ \sum_{m=1}^{\infty}  \ex(m \theta) \cdot v_m/m + i\sum_{m=1}^{\infty}  \ex(m \theta) \cdot w_m/m }     \frac{d\theta}{2\pi}}   \\ \leq & e^{   \frac{ \abs{v_1}^2}{4} - \delta_3 \min ( \abs{v_1}^4, \abs{v_1}^2) + \sum_{m=2}^{\infty} \min (C_1  \abs{v_m}^2, \abs{v_m}/m ) } (  1 + C_2 \sum_{m=2}^{\infty} \abs{w_m}^2 + C_2 \sum_{m=2}^{\infty} \abs{v_m}^2 ) \mathcal S( \abs{w_1} ) . \end{split}   \end{equation}

Furthermore, we have \begin{equation}\label{s-long-range} \mathcal S(y) = O \left( \frac{1}{ \sqrt{y} } \right),\end{equation} we have \begin{equation}\label{s-short-range} \mathcal S(y) \leq  e^{  - \frac{y^2}{4} + O( y^3)},\end{equation} and $\mathcal S(y)$ is bounded away from $1$ for $y$ in each fixed closed interval not containing $0$.

\end{proposition}

Here the function $\mathcal S$ describes how much savings is obtained in our estimate from cancellation induced by $w_1$. The advantage of writing the bound in this way is we can treat $\mathcal S(\abs{w_1})$ as a single quantity for calculations that are uniform in $\abs{w_1}$ but also easily break up into different ranges.

From this point on, we always take $\mathcal S$ to be a function as in \cref{hybrid-unified}.

\begin{proof}  Take $C_1$ as in \cref{laplace-unified}. Fix $\delta_3, C_2$ to be chosen later. We choose $\delta_3$ sufficiently small and $C_2$ sufficiently large. Neither depends on the other. We will always write a fixed closed interval not containing zero as $[C_3,C_4]$. (There is also no relation between $C_3,C_4$ and the other variables.) We let
 \begin{equation}S(v_1,\dots, w_1,\dots) = \label{min-for-S} \frac{ \abs{  \int_0^{2\pi}   e^{ \sum_{m=1}^{\infty}  \ex(m \theta) \cdot v_m/m + i\sum_{m=1}^{\infty}  \ex(m \theta) \cdot w_m/m }     \frac{d\theta}{2\pi}} }{  e^{   \frac{ \abs{v_1}^2}{4} - \delta_3 \min ( \abs{v_1}^4, \abs{v_1}^2) + \sum_{m=2}^{\infty} \min (C_1  \abs{v_m}^2, \abs{v_m}/m ) } (  1 + C_2 \sum_{m=2}^{\infty} \abs{w_m}^2 + C_2 \sum_{m=2}^{\infty} \abs{v_m}^2 ) }.\end{equation}
and define
\[ \mathcal S(y) = \inf_{ \substack{ (v_m)_{m=1}^\infty, (w_m)_{m=1}^\infty  \in (\mathbb R^2)^\mathbb N \\  \abs{w_1} =y}}   S(v_1,\dots, w_1,\dots) \]
so that \eqref{eq-hyb-un} holds by definition and the upper bounds on $\mathcal S(y)$ can be checked by checking corresponding upper bounds on $S(v_1,\dots, w_1,\dots) $. That is, for \eqref{s-long-range} it suffices to have \begin{equation}\label{slr} S(v_1,\dots, w_1,\dots) = O \left( \frac{1}{\sqrt{\abs{w_1}}}\right),\end{equation} for \eqref{s-short-range} it suffices to have \begin{equation}\label{ssr} S(v_1,\dots, w_1,\dots)  \leq e^{ - \frac{ \abs{w_1}^2}{4} + O(\abs{w_1}^3)} \end{equation} and for $\mathcal S(y)$ to be bounded away from $1$ for $y$ in an interval $[C_3,C_4]$ not containing $0$ it suffices to have $S(v_1,\dots, w_1,\dots) $ bounded away from $1$ for $\abs{w_1}\in [C_3,C_4]$

First note that we can always bound the integral  $\int_0^{2\pi}   e^{ \sum_{m=1}^{\infty}  \ex(m \theta) \cdot v_m/m + i\sum_{m=1}^{\infty}  \ex(m \theta) \cdot w_m/m }     \frac{d\theta}{2\pi}$ by its untwisted form $\int_0^{2\pi}   e^{ \sum_{m=1}^{\infty}  \ex(m \theta) \cdot v_m/m  }     \frac{d\theta}{2\pi}$ which is bounded by Proposition \ref{laplace-unified}
as \[ e^{   \frac{ \abs{v_1}^2}{4} - \delta_2 \min ( \abs{v_1}^4, \abs{v_1}^2) + \sum_{m=2}^{\infty} \min (C_1  \abs{v_m}^2, \abs{v_m}/m ) ) },\] which can be bounded by
\[ e^{   \frac{ \abs{v_1}^2}{4} - \delta_3 \min ( \abs{v_1}^4, \abs{v_1}^2) + \sum_{m=2}^{\infty} \min (C_1  \abs{v_m}^2, \abs{v_m}/m ) ) } \frac{1}{ 1+ \delta_3 \abs{v_1}^4} \] since  $\log (1+ \delta_3 \abs{v_1})^4 \leq (\delta_2 -\delta_3) \min ( \abs{v_1}^4, \abs{v_1}^2)$ for $\delta_3$ sufficiently small with respect to $\delta_2$.

It follows that \[ S(v_1,\dots, w_1,\dots)  \leq \frac{1} {   (1 + \delta_3 \abs{v_1}^4) ( 1 + C_2 \sum_{m=2}^{\infty} \abs{w_m}^2 + C_2 \sum_{m=2}^{\infty} \abs{v_m}^2 ) } .\] This in particular implies that $S(v_1,\dots, w_1,\dots) \leq 1$ and thus $\mathcal S(y) \leq 0$. Since $\mathcal S(y)$ is clearly nonnegative we see that $\mathcal S$ is indeed a function from $ [0,\infty]$ to $ [0,1]$.

Furthermore, as long as  $\delta_3 \abs{v_1}^4 + C_2 \sum_{m=2}^{\infty} \abs{ w_m}^2  + C_2 \sum_{m=2}^\infty \abs{v_m}^2 \geq \abs{w_1}^2/4$ we have  $S(v_1,\dots, w_1,\dots) \leq \frac{1}{ 1+ \abs{w_1}^2/4}$.  Since $\frac{1}{ 1+ \abs{w_1}^2/4} $ is $O (\frac{1}{\sqrt{\abs{w_1}}})$, is equal to $e^{ - \frac{ \abs{w_1}^2}{4}+ O(\abs{w_1}^3)}$, and is bounded away from $1$ for $C_3 \leq \abs{w_1} \leq C_4$, for the remainder of the argument we may assume that
\begin{equation}\label{hybrid-unified-assumption} \delta_3 \abs{v_1}^4 + C_2 \sum_{m=2}^{\infty} \abs{ w_m}^2  + C_2 \sum_{m=2}^\infty \abs{v_m}^2 < \abs{w_1}^2/4 \end{equation}
which notably implies
\[  \sum_{m=1}^{\infty} \abs{ w_m}^2  + \sum_{m=1}^\infty \abs{v_m}^2 < O ( \abs{w_1}^2) + O(1) \]
since $\delta_3 \abs{v_1}^4  \geq C_2 \abs{v_1}^2 + O(1)$  and $\abs{w_1}^2= O(\abs{w_1}^2)$.

First we check \eqref{ssr}. Since $S(v_1,\dots,w_1,\dots) \leq 1$, it suffices to check \eqref{ssr} for $y$ sufficiently small.  Note that \eqref{hybrid-unified-assumption} implies that, as long as $\abs{w_1}$ is sufficiently small, $\abs{v_1}$ is as small as desired, and the same holds for $\sum_{m=2}^{\infty}  \abs{v_m}^2 + \sum_{m=2}^{\infty} \abs{w_m}^2$. By Cauchy-Schwarz
\[\abs{v_1}+ \abs{w_1}+ \sum_{m=2}^{\infty}  \sqrt{ \abs{v_m}^2 + \abs{w_m}^2}/m  \leq  \abs{v_1}+\abs{w_1} + \sqrt{ (\zeta(2)-1)  \sum_{m=2}^{\infty}( \abs{v_m}^2 + \abs{w_m}^2)}  \]
which we can take to be as small as desired, in particular ensuring the assumption of \cref{hybrid-short-range} is satisfied, and we have
\[\log \abs{  \int_0^{2\pi}   e^{ \sum_{m=1}^{\infty}  \ex(m \theta) \cdot v_m/m + i\sum_{m=1}^{\infty}  \ex(m \theta) \cdot w_m/m} \frac{d\theta}{2\pi}} \] 
\[ \leq \frac{ \abs{v_1}^2}{4} - \frac{ \abs{w_1}^2}{4} - \delta_3 \abs{v_1}^4 + O_{\delta_3}  ( \abs{v_1}^6 + \abs{w_1}^3 + \sum_{m=2}^{\infty} (\abs{v_m}^2+ \abs{w_m}^2 )) \]
and by \eqref{hybrid-unified-assumption} we have $\abs{v_1}^6 = O ( \abs{w_1}^3)$ so that term can be ignored. Thus \eqref{min-for-S} is less than or equal to 
\[  \frac{ e^{ \frac{ \abs{v_1}^2}{4}- \delta_3 \abs{v_1}^4}} { e^{ \frac{ \abs{v_1}^2}{4} - \delta_3 \min ( \abs{v_1}^4, \abs{v_1}^2)}}   \frac{e^{ O(  \sum_{m=2}^{\infty} (\abs{v_m}^2+ \abs{w_m}^2 ))} }{1 + C_2 \sum_{m=2}^{\infty} \abs{w_m}^2 + C_2 \sum_{m=2}^{\infty} \abs{v_m}^2}   e^{ - \frac{ \abs{w_1}^2}{4} + O ( \abs{w_1}^3 )} .\] We have dropped the term $ \sum_{m=2}^{\infty} \min (C_1  \abs{v_m}^2, \abs{v_m}/m )$ from the denominator as it is always $\geq 1 $ but is unneeded.

We clearly have $  \frac{ e^{ \frac{ \abs{v_1}^2}{4}- \delta_3 \abs{v_1}^4}}{ e^{ \frac{ \abs{v_1}^2}{4} - \delta_3 \min ( \abs{v_1}^4, \abs{v_1}^2)}}   \leq 1$  and we have $\frac{e^{ O(  \sum_{m=2}^{\infty} (\abs{v_m}^2+ \abs{w_m}^2 ))} }{1 + C_2 \sum_{m=2}^{\infty} \abs{w_m}^2 + C_2 \sum_{m=2}^{\infty} \abs{v_m}^2} \leq 1$ as long as $\abs{w_1}$ is sufficiently small and $C_2$ is sufficiently large since we can bound $e^x $ by $1 + cx$ for any $c>1$ as long as $x$ is sufficiently small.

Thus \eqref{min-for-S} is less then or equal to $e^{ - \frac{ \abs{w_1}^2}{4} + O ( \abs{w_1}^3 )} $ for $\abs{w_1}$ sufficiently small which gives $\mathcal S(y) \leq  e^{ - \frac{y^2}{4} + O ( y^3) } $ for $y$ sufficiently small, and thus for all $y$, verifying \eqref{ssr}.

   Next we check \eqref{slr}. Before applying \cref{hybrid-long-range}, we observe that 
   \[ \sup_{\theta \in [0,2\pi]}   e^{ \sum_{m=1}^{\infty}  \ex(m \theta) \cdot v_m/m } \leq  e^{ \sum_{m=1}^\infty \abs{v_m}/m} \leq e^{ \frac{\zeta(2)}{4} + \sum_{m=1}^\infty \min( \abs{v_m}^2, \abs{v_m}/m)}\] \[ \ll   e^{   \frac{ \abs{v_1}^2}{4} - \delta_3 \min ( \abs{v_1}^4, \abs{v_1}^2) + \sum_{m=2}^{\infty} \min (C_1  \abs{v_m}^2, \abs{v_m}/m )} \] since $\delta_3 \leq \frac{1}{4}$ and $C_1\geq 1$.
   Next observe that (using \eqref{hybrid-unified-assumption} to handle the case $\abs{w_1}$ small in the first inequality)
   \[ \frac{4}{\pi} + 1 +   \frac{ \sqrt{\sum_{m=1}^{\infty} \abs{v_m}^2 +   \sum_{m=2}^{\infty} \abs{w_m}^2}}{ \sqrt{\pi} \abs{w_1}^{1/4}} \ll  1  +  \sqrt{ \sum_{m=1}^{\infty} \abs{v_m}^2 +   \sum_{m=2}^{\infty} \abs{w_m}^2}  \] \[\leq  \frac{5}{4} + \sum_{m=1}^{\infty} \abs{v_m}^2 +  \sum_{m=2}^{\infty} \abs{w_m}^2  \ll  1 + \delta_3 \abs{v_1}^4 + C_2  \sum_{m=2}^{\infty} \abs{w_m}^2 + C_2 \sum_{m=2}^{\infty} \abs{v_m}^2 . \] Putting these bounds together with \cref{hybrid-long-range}, we obtain
\[ \abs{  \int_0^{2\pi}   e^{ \sum_{m=1}^{\infty}  \ex(m \theta) \cdot v_m/m + i\sum_{m=1}^{\infty}  \ex(m \theta) \cdot w_m/m }     \frac{d\theta}{2\pi}}   \] \[ \leq \frac{\sup_{\theta \in [0,2\pi]}   e^{ \sum_{m=1}^{\infty}  \ex(m \theta) \cdot v_m/m } }{\sqrt{\abs{w_1}}} \left( \frac{4}{\pi} + 1 +   \frac{ \sqrt{ \sum_{m=1}^{\infty} \abs{v_m}^2 +  \sum_{m=2}^{\infty} \abs{w_m}^2}}{ \sqrt{\pi} \abs{w_1}^{1/4}} \right) \]
\[ \ll  e^{   \frac{ \abs{v_1}^2}{4} - \delta_3 \min ( \abs{v_1}^4, \abs{v_1}^2) + \sum_{m=2}^{\infty} \min (C_1  \abs{v_m}^2, \abs{v_m}/m )}  \frac{1}{ \sqrt{\abs{w_1}}  } ( 1 + \delta_3 \abs{v_1}^4 + C_2  \sum_{m=2}^{\infty} \abs{w_m}^2 + C_2 \sum_{m=2}^{\infty} \abs{v_m}^2 ) \]
\[  \leq  e^{   \frac{ \abs{v_1}^2}{4} - \delta_3 \min ( \abs{v_1}^4, \abs{v_1}^2) + \sum_{m=2}^{\infty} \min (C_1  \abs{v_m}^2, \abs{v_m}/m )}  \frac{1}{ \sqrt{\abs{w_1}}  }  (1 + \delta_3 \abs{v_1}^4)   ( 1 + C_2  \sum_{m=2}^{\infty} \abs{w_m}^2 + C_2 \sum_{m=2}^{\infty} \abs{v_m}^2 ) \]
which verifies \eqref{slr}.

  Next we consider $\abs{w_1}$ in an interval $I= [C_3,C_4]$ not containing $0$. Applying \cref{laplace-unified} and then \cref{hybrid-medium-range} we have
\[ S(v_1,\dots, w_1,\dots) \leq   \frac{ \abs{  \int_0^{2\pi}   e^{ \sum_{m=1}^{\infty}  \ex(m \theta) \cdot v_m/m + i\sum_{m=1}^{\infty}  \ex(m \theta) \cdot w_m/m }     \frac{d\theta}{2\pi}} }{\abs{  \int_0^{2\pi}   e^{ \sum_{m=1}^{\infty}  \ex(m \theta) \cdot v_m/m }     \frac{d\theta}{2\pi}} }\] \begin{equation}\label{hmr-rhs} \leq 1- e^{ -2 \sum_{m=1}^{\infty} \abs{v_m}/m}  \min \left(  \sum_{m=1}^{\infty} \frac{\abs{w_m}^2}{\pi^2 m^2}, \frac{1}{ 2\sum_{m=1}^{\infty} \abs{w_m}^2} \right).\end{equation}
Since \[\sum_{m=1}^{\infty} \abs{v_m}/m \leq \sqrt{ \zeta(2) \sum_{m=1}^\infty \abs{v_m}^2} \ll \abs{w_1} \leq C_4=O(1)\] and \[ \sum_{m=1}^{\infty} \abs{w_m}^2  \ll \abs{w_1}^2 \leq C_4^2 =O(1) \]
and \[  \sum_{m=1}^{\infty} \frac{\abs{w_m}^2}{\pi^2 m^2} \geq \frac{\abs{w_1}^2}{\pi^2m^2} \geq \frac{ C_3^2}{ \pi^2 m^2}, \] we see that \eqref{hmr-rhs} is at most $1-\epsilon$ for some $\epsilon>0$. Hence $S(v_1,\dots, w_1,\dots) \leq 1-\epsilon$ for $\abs{w_1}\in [C_3, C_4]$, as desired. \end{proof}

\begin{corollary}\label{hybrid-separated} Let $(v_m)_{m=1}^\infty$ and $(w_m)_{m=1}^\infty$ be sequences of vectors in $\mathbb R^2$. If $\sum_{m=1}^\infty m \abs{v_m}<\infty$ and $\sum_{m=1}^{\infty} w_m^2 < \infty$ then
\[ \abs{  \int_0^{2\pi}   e^{ \sum_{m=1}^{\infty}  \ex(m \theta) \cdot v_m/m + i\sum_{m=1}^{\infty}  \ex(m \theta) \cdot w_m/m }     \frac{d\theta}{2\pi}} \] \[\leq e^{   \frac{ \abs{v_1}^2}{4} - \delta_3 \min ( \abs{v_1}^4, \abs{v_1}^2)  } \mathcal S( \abs{w_1})   \prod_{m=2}^{\infty}\Bigl( e^{ \min ( C_1 \abs{v_m}^2, \abs{v_m}/m )}  (1 + C_2 \abs{w_m}^2 )( 1+ C_2 \abs{v_m}^2) \Bigr).\] \end{corollary}

\begin{proof} This follows from taking \cref{hybrid-unified} and separating terms, using the trivial bound
\[  1 + C_2 \sum_{m=2}^{\infty} \abs{w_m}^2 + C_2 \sum_{m=q}^{\infty} \abs{v_m}^2 \leq\prod_{m=2}^\infty (1 + C_2 \abs{w_m}^2 ) (1+ C_2 \abs{v_m}^2) . \qedhere\]\end{proof}

Write $A_n$ for $\sum_{d\mid n, d<n} E_d$ and $B_n$ for $\sum_{d\mid n, d<n} E_d d/n $.
\begin{corollary}\label{hybrid-combined} Let $v_1,\dots, v_k$ and $w_1,\dots, w_k$ be vectors in $\mathbb R^2$. Then\begin{equation}\label{eq-hybrid-combined} \begin{split} & \abs{\mathbb E [ e^{ \sum_{n=1}^k X_{n,\xi} \cdot v_n  +  i \sum_{n=1}^k X_{n,\xi} \cdot w_n } ] }\\  \leq  &  \prod_{n=1}^k \Bigl (  e ^{ E_n  \frac{ \abs{v_n}^2}{4}  +  \min (C_1 A_n \abs{v_n}^2, B_n  \abs{v_n}  )  - \delta_3 E_n \min ( \abs{v_n}^4, \abs{v_n}^2) } (1 + C_2 \abs{v_n}^2)^{ A_n  } (1+ C_2 \abs{w_n}^2)^{ A_n} \mathcal S(\abs{w_n} ) ^{E_n} \Bigr).\end{split} \end{equation}\end{corollary}

\begin{proof} Taking \cref{laplace-formula}, using \cref{hybrid-separated} to bound each factor, and then rearranging terms, we obtain
\[ \mathbb E [ e^{ \sum_{n=1}^k X_{n,\xi} \cdot v_n  +  i \sum_{n=1}^k X_{n,\xi} \cdot w_n } ] \] \[ \leq  \prod_{n=1}^k \Bigl( \Bigl(   e^{   \frac{ \abs{v_n}^2}{4} - \delta_3 \min ( \abs{v_n}^4, \abs{v_n}^2)  } \mathcal S( \abs{w_n} ) \Bigr)^{ E_n} \prod_{ d\mid n, d<n} \Bigl( e^{ \min ( C_1 \abs{v_n}^2, \abs{v_n}d/n )}  (1 + C_2 \abs{w_n}^2 )(1+ C_2 \abs{v_n}^2) \Bigr)^{ E_d} \Bigr) . \] 
Using \[ \sum_{ d\mid n, d<n} E_d \min ( C_1 \abs{v_n}^2, \abs{v_n}d/n ) \leq \min (C_1 \sum_{ d\mid n, d<n} E_d \abs{v_n}^2, \sum_{d\mid n} E_d \abs{v_n} d/n ) \] we obtain \eqref{eq-hybrid-combined}. \end{proof}

The following facts about $A_n, B_n,$ and $E_n$ will be useful in the remainder of the argument.

\begin{lemma} We have the following identities and inequalities for $A_n, B_n, E_n$:
 \begin{equation}\label{bnnn} B_n + E_n = \frac{q^n}{n}. \end{equation}
 \begin{equation}\label{an-bound} A_n = O ( q^{n/2}/n) .\end{equation}
  \begin{equation}\label{bn-bound} B_n = O ( q^{n/2}/n).\end{equation}
  \begin{equation}\label{nn-asymptotic} E_n = q^n/n + O( q^{n/2}/n) .\end{equation}
\begin{equation}\label{q5} E_n> 2A_n+4 \textrm{ as long as } q>5 .\end{equation}
There exists a positive constant $c$ such that 
\begin{equation}\label{q11}  \frac{E_n}{2}- 2 A_n - \frac{q}{2}-1> c(A_n + E_n)  \textrm{ as long as } q>11\textrm{ and } n>1. \end{equation}
\end{lemma}

\begin{proof} For \eqref{bnnn} by definition we have $B_n+ E_n = \sum_{d\mid n} E_d d/n$ and $\sum_{d\mid n } E_d  d= q^n$ by either counting elements of the finite field $\mathbb F_{q^n}$ or a zeta function argument. This in particular implies $E_n \leq \frac{q^n}{n}$, which we will use repeatedly in the remaining proofs. 

For \eqref{an-bound}, we observe that the largest possible $d$ satisfying $d\mid n$ and $d<n$ is $n/2$ and every other solution is at most $n/3$, so  
\[ A_n=\sum_{d\mid n, d<n} E_d \leq  E_{n/2} + \sum_{d \leq n/3} E_{n/3}  \leq q^{n/2}/ (n/2) + \sum_{d\leq n/3} q^{d} = O(q^{n/2}/n) + O(q^{n/3})= O(q^{n/2}/n).\]

\eqref{bn-bound} follows from \eqref{an-bound} upon observing $B_n\leq A_n$. 

\eqref{nn-asymptotic} follows from \eqref{bnnn} and \eqref{bn-bound}. 

To obtain  \eqref{q5} and \eqref{q11} we redo the above proofs with explicit constants to prove the inequalities for all $q^n$ sufficiently large and then use exact formulas to handle the finitely many remaining possible values of $(q,n)$. 

 The proof of \eqref{an-bound} gives
 \[ A_n=\sum_{d\mid n, d<n} E_d \leq  E_{n/2} + \sum_{d \leq n/3} E_{n/3}  \leq q^{n/2}/ (n/2) + \sum_{d\leq n/3} q^{d} \leq 2 q^{n/2}/n + \frac{1}{1-q^{-1} } q^{n/3} .\]

For $X> 0$ we have $2X^{1/6}> \frac{1}{1-7^{-1} }\log_7 X$ so we have \begin{equation}\label{exp-bounds-n} 2q^{n/6} > \frac{1}{1-7^{-1}} \log_7 q^n \geq \frac{1}{1-q^{-1} } \log_q q^n =  \frac{1}{1-q^{-1}} n \end{equation}  so
\begin{equation}\label{fussy-an-bound} A_n \leq 4 q^{n/2}/n.\end{equation} 
Then we have
\begin{equation}\label{fussy-bn-bound}   B_n = A_n \leq 4  q^{n/2}/n \end{equation}
and
\begin{equation}\label{fussy-nn-bound} E_n= \frac{q^n}{n}-B_n \geq \frac{q^n}{n} - 4 \frac{q^{n/2}}{n} . \end{equation}

Thus \eqref{q5} is satisfied as long as  \[  \frac{q^n}{n} - 4\frac{q^{n/2}}{n}  >  4 \frac{q^{n/2}}{n} +4 \] i.e. as long as  \[ 1 >  8 q^{-n/2}  +4 n q^{-n} \] which by \eqref{exp-bounds-n} follows from \[ 1> 8 q^{-n/2} + 8 q^{-5n/6} \] which holds for $q^n> 95.2$. Because $q>7$ this holds for all $n>2$. But for $n=1$ we have $E_n=q$ and $A_n=0$ so \eqref{q5} becomes $q>4$ which is satisfied for all $q>5$ and for $n=2$ we have $E_n = \frac{q^2-q}{2}$ and $A_n=q$ so \eqref{q5} becomes $\frac{q^2-5q}{2} >4$ which is satisfied for all $q>5$.

For \eqref{q11} it follows from \eqref{an-bound}, \eqref{bn-bound}, and \eqref{nn-asymptotic} that $  \frac{E_n}{2}- 2 A_n - \frac{q}{2}-1 = q^n/(2n) + O(q^{n/2}/n) +O(q)$ and $A_n+ E_n =q^n/n + O(q^{n/2}/2) $ so \eqref{q11} is satisfied for $q^n$ sufficiently large. Thus it suffices to prove \begin{equation}\label{q11-simplified} \frac{E_n}{2}- 2 A_n - \frac{q}{2}-1> 0\end{equation}  as then we can choose $c$ small enough to ensure \eqref{q11} is satisfied for the finitely many remaining values of $q,n$.  Then by \eqref{fussy-an-bound}, \eqref{fussy-nn-bound}, and \eqref{exp-bounds-n} it suffices to prove
\[ \frac{q^n}{2n} - 2 \frac{q^{n/2}}{n} - 8 \frac{q^{n/2}}{n} - \frac{q^{1+ \frac{n}{6}} }{n}- \frac{ 2 q^{n/6}}{n}>0 \]
or equivalently
\[ \frac{1}{2} > 10 q^{-n/2}  + q^{ 1 - \frac{5n}{6}}  + 2 q^{- \frac{5n}{6}}.\]
which since $n \geq 2$ follows from
\[ \frac{1}{2} > 10 q^{-n/2}  + q^{ -\frac{n}{3}}  + 2 q^{- \frac{5n}{6}}\]
which holds for $q^n>697.4$. Because $q>11$ this holds for all $n>2$. For $n=2$ we have $E_n = \frac{q^2-q}{2}$ and $A_n=q$ so \eqref{q11-simplified} becomes $\frac{q^2-11q}{4}-1 >0$ which is satisfied for all $q>11$.  \end{proof}

\begin{lemma}\label{v-term-bound} For any $n>0$ and $v\in \mathbb R^2$, we have \begin{equation}\label{v-dependence} e ^{ - B_n  \frac{\abs{v}^2}{4} +  \min (C_1 A_n \abs{v}^2, B_n  \abs{v}  )  - \delta_3 E_n \min ( \abs{v}^4, \abs{v}^2) } (1 + C_2 \abs{v}^2)^{ A_n  }\leq e^{O ( \min ( q^{n/2} \abs{v}^2 /n, 1/n))} .\end{equation}\end{lemma}

\begin{proof}  We note first that  $\min (C_1 A_n \abs{v}^2, B_n  \abs{v}  ) = O ( A_n \abs{v}^2)$ and also $ \log (1 + C_2 \abs{v}^2)^{ A_n  }= O ( A_n \abs{v}^2)$ so the left-hand side of \eqref{v-dependence} is always $e^{ O ( A_n \abs{v}^2)} = e^{ O ( q^{n/2} \abs{v}^2/n)}$ by \eqref{an-bound}.

We next check that the left-hand side of \eqref{v-dependence} is $\leq e^{ O(1/n)} $. First in the range $\abs{v}\leq 1$, it suffices to check that   $e^{ O ( A_n \abs{v}^2) - \delta_3 E_n  \abs{v}^4 } \leq e^{ O(1/n)} $ but the exponent is $\leq O \left( \frac{A_n^2}{\delta_3 E_n} \right)$ and therefore is $\leq O(1/n)$ by  \eqref{an-bound} and \eqref{nn-asymptotic}. For $\abs{v}\geq 1$, it suffices to check that  $e^{ O ( A_n \abs{v}^2) - \delta_3 E_n \abs{v}^4} \leq  e^{ O(1/n)} $ which is automatic as long as  $O(A_n) - \delta_3 E_n \leq 0$ which happens for all but finitely many $n$, again by  \eqref{an-bound} and \eqref{nn-asymptotic}. For these finitely many $n$, it suffices to check that \eqref{v-dependence} goes to $0$ as $n$ goes to $\infty$, which is clear as $e^{ \min ( C_1 A_n \abs{v}^2, \abs{v})}$ is merely exponential in a  linear function of $\abs{v}$ while $(1 + C_2 \abs{v}^2)^{ A_n  }$ is polynomial and these are both dominated by $e^{ -\delta_3 E_n \min ( \abs{v}^4, \abs{v}^2) }$ which is exponential in a quadratic function.
\end{proof}

\begin{corollary}\label{hybrid-simplified} Let $v_1,\dots, v_k$ and $w_1,\dots, w_k$ be vectors in $\mathbb R^2$. Then\begin{equation}\label{eq-hybrid-simplified} \begin{split} & \abs{\mathbb E [ e^{ \sum_{n=1}^k X_{n,\xi} \cdot v_n  +  i \sum_{n=1}^k X_{n,\xi} \cdot w_n } ] } \leq    \prod_{n=1}^k \Bigl (  e ^{ \frac{q^n}{n}  \frac{ \abs{v_n}^2}{4}   + O ( \min ( q^{n/2} \abs{v_n}^2 /n, 1/n))}(1+ C_2 \abs{w_n}^2)^{ A_n} \mathcal S(\abs{w_n} ) ^{E_n} \Bigr).\end{split} \end{equation}\end{corollary}
 
\begin{proof} This follows from plugging $E_n = q^n/n-B_n$ from \eqref{bnnn} into \cref{hybrid-combined} and then plugging in \cref{v-term-bound}. \end{proof}

\subsection{Pointwise bounds}

Recall that $F(x_1,\dots, x_k)$ is the probability density function of $X_{1,\xi},\dots, X_{n,\xi}$.

\begin{proposition}\label{pointwise-wint} Assume $q>5$. Let $x_1,\dots, x_k$ be vectors in $\mathbb R^2$.   Then 
\begin{equation}\label{eq-pointwise-wint} \frac{ F(x_1,\dots, x_k)}{ \prod_{n=1}^k \left( e^{ - \frac{ n\abs{x_n}^2}{q^n}}   \frac{n}{q^n \pi} \right)} \leq  O\Bigl(  e^{ O ( \sum_{n=1}^k \min (  n q^{-3n/2} \abs{x_n}^2, n^{-1} ) ) } \Bigr).\end{equation}\end{proposition}

\begin{proof} Let  \begin{equation}\label{vn-def} v_n = 2 n x_n  /q^n\end{equation} for all $n$ from $1$ to $k$.  By the Fourier inversion formula we have 
\begin{equation}\label{pointwise-wint-first} (2\pi)^{2k} e^{ \sum_{n=1}^k x_n \cdot v_n}  F( x_1,\dots, x_n) = \int_{w_1,\dots, w_k \in \mathbb R^2}  e^{ -i \sum_{n=1}^k x_n \cdot w_n } \mathbb E [ e^{ \sum_{n=1}^k X_{n,\xi} \cdot v_n  +  i \sum_{n=1}^k X_{n,\xi} \cdot w_n } ] dw_1 \dots dw_k .\end{equation}
\eqref{pointwise-wint-first} and \cref{hybrid-simplified}, give
\[ (2\pi)^{2k} e^{ \sum_{n=1}^k x_n \cdot v_n}  F( x_1,\dots, x_n) \leq  \int_{w_1,\dots, w_k \in \mathbb R^2}  \abs{ \mathbb E [ e^{ \sum_{n=1}^k X_{n,\xi} \cdot v_n  +  i \sum_{n=1}^k X_{n,\xi} \cdot w_n } ]} dw_1 \dots dw_k \]
\[ \hspace{-.5in} \leq \prod_{n=1}^k e^{\frac{q^n}{n} \frac{ \abs{v_n}^2}{4}+ O ( \min ( q^{n/2} \abs{v_n}^2 /n, 1/n))} \int_{w_1,\dots, w_k \in \mathbb R^2}   \prod_{n=1}^k  (1+ C_2 \abs{w_n}^2)^{ A_n}   \mathcal S( \abs{w_n} )^{E_n}   dw_1 \dots dw_k\]
\begin{equation}\label{pointwise-wint-setup} = \prod_{n=1}^k e^{ \frac{q^n}{n} \frac{ \abs{v_n}^2}{4}+  O ( \min ( q^{n/2} \abs{v_n}^2 /n, 1/n))}  \prod_{n=1}^k \int_{w\in \mathbb R^2} (1+ C_2 \abs{w}^2)^{ A_n}  \mathcal S( \abs{w} )^{E_n} dw. \end{equation}

We first tackle the inner integral \begin{equation}\label{pointwise-wint-inner} \int_{w\in \mathbb R^2}  (1+ C_2 \abs{w}^2)^{ A_n} \mathcal S( \abs{w} )^{E_n}   dw.\end{equation} First note that for $w$ large, we have $(1+ C_2 \abs{w}^2)^{ A_n} = O(\abs{w}^{2A_n})$ and $S( \abs{w} )^{E_n} = O ( \abs{w}^{- \frac{E_n}{2}})$, so for the integral to converge, it is necessary and sufficient to have $E_n/2 > A_n+2$, which follows from \eqref{q5}.

Since we can absorb the integrals \eqref{pointwise-wint-inner} for small $n$ into the implicit constant, we will be focused on the asymptotic evaluation of \eqref{pointwise-wint-inner} for large $n$.

For $\abs{w}$ bounded we have \[\log (1+ C_2 \abs{w}^2) = 1+ C_2 \abs{w}^2 + O (\abs{w}^4)\] and \[ \log  \mathcal S(\abs{w})  \leq \frac{\abs{w}^2}{4} + O( \abs{w})^3\] so
\[ \ (1+ C_2 \abs{w}^2)^{ A_n}  \mathcal S(\abs{w})^{E_n}  \leq  e^{ \left(C_2 A_n - \frac{E_n}{4} \right) \abs{w}^2 +O ( q^n \abs{w}^3/n)} \]  (using \eqref{an-bound} and \eqref{nn-asymptotic})  which for $\abs{w} \leq (q^n/n)^{-2/5}$ is $\leq e^{ \left(C_2 A_n - \frac{E_n}{4} \right) \abs{w}^2 }  e^{ ( q^n/n)^{-1/5}} $ so the integral over $\abs{w}\leq (q^n/n)^{-2/5}$ is bounded by $e^{ ( q^n/n)^{-1/5}} $ times
\[ \int_{w\in \mathbb R^2} e^ {\left(C_2 A_n - \frac{E_n}{4} \right) w^2} dw =  \frac{ \pi} {   \frac{E_n}{4} - C_2 A_n } = \frac{ \pi}{ \frac{q^n}{4n} - O ( q^{n/2} /n) } =\frac{4\pi n }{q^n} + O \left(  \frac{n}{q^{3n/2}} \right) \]
(for $n$ large), with the $e^{ ( q^n/n)^{-1/5}} $ factor itself contributing an error term of size $O (  (q^n/n)^{-6/5})$.

The integral \eqref{pointwise-wint-inner} over $\abs{w}> (q^n/n)^{-2/5}$ will give additional error terms.

First in the range where $\abs{w}> (q^n/n)^{-2/5}$ but $\abs{w}$ is bounded by a fixed small constant, we have  $(C_2 A_n - \frac{E_n}{4}) \abs{w}^2= - \left( \frac{q^n}{4n}+ O (q^{n/2}/n)\right) \abs{w}^2$ which is larger by a constant factor than $O ( q^n \abs{w}^3/n)$, so the integrand of \eqref{pointwise-wint-inner} in this range is at most $e^{ - c q^n \abs{w}^2/n} \leq e^{ - c (q^n/n)^{1/5}}$ for a small constant $c$. Since the area of this range is $O(1)$, this range gives an error term of size decreasing superexponentially in $q^n/n$.

For $\abs{w}$ greater than a large fixed constant $R$, we have $1 + C_2 \abs{w}^2 = O( \abs{w})^2 $and $\mathcal S(w) = O ( \abs{w}^{-1/2}) $. This gives a bound of $O(1)^{A_n + E_n} \abs{w}^{ 2A_n - E_n/2}$ for the integrand or $O(1)^{ A_n + E_n }  R^{ 2+ 2A_n - E_n/2}$ for the integral \eqref{pointwise-wint-inner}, and since both $A_n + E_n$ and $E_n -4 A_n -4$ are asymptotic to $q^n/n$ by \eqref{an-bound} and \eqref{nn-asymptotic}, we can choose $R$ large enough that the second term dominates and the error term decays superexponentially in $q^n/n$.

For the intermediate range of $\abs{w}$ between two fixed constants, we also get superexponential decay simply by observing that $(1 + C_2 \abs{w}^2)^2= O(1)^{A_n}$ and $\mathcal S(\abs{w})^{E_n} \leq (1-\epsilon)^{E_n}$ for some $\epsilon>1$, while $A_n = o(E_n)$ and $E_n$ increases exponentially by \eqref{an-bound} and \eqref{nn-asymptotic}, so the integrand has superexponential decay and the length of the integral on this range is $O(1)$.

So all these error terms are dominated by the  $O (  (q^n/n)^{-6/5})$, giving 
\[ \int_{w\in \mathbb R^2}  (1+ C_2 \abs{w}^2)^{ A_n}  \mathcal S(\abs{w})^{E_n} dw = \frac{4\pi n }{q^n} +O (  (q^n/n)^{-6/5})\]
which implies 
\begin{equation}\label{pointwise-w-integral} \prod_{n=1}^k \int_{w\in \mathbb R^2} (1+ C_2 \abs{w}^2)^{ A_n}   \mathcal S(\abs{w})^{E_n}   dw = O \Bigl( \prod_{n=1}^k \frac{4 \pi n}{q^n} \Bigr) .\end{equation}

Since $v_n = 2 n x_n / q^n$ we have \begin{equation}\label{dot-splitting} v_n \cdot x_n =   \frac{q^n}{ 4n} \abs{v_n}^2 + \frac{n}{q^n} \abs{x_n}^2 \end{equation} and plugging \eqref{dot-splitting} and \eqref{pointwise-w-integral} into \eqref{pointwise-wint-setup} we obtain
\[ (2\pi)^{2k} e^{ \sum_{n=1}^k \left(\frac{q^n}{ 4n} \abs{v_n}^2 + \frac{n}{q^n} \abs{x_n}^2 \right)}  F( x_1,\dots, x_n) \leq \prod_{n=1}^k \Bigl (  e ^{ \frac{q^n}{n} \frac{ \abs{v_n}^2}{4}  + O( \min ( q^{n/2} \abs{v_n}^2 /n, 1/n)) } \Bigr) O ( \prod_{n=1}^k \frac{4 \pi n}{q^n} ) \]
and solving for $F(x_1,\dots, x_n)$  gives
\[   F( x_1,\dots, x_n) \leq O(1)   \prod_{n=1}^k \Biggl (  e ^{O( \min ( q^{n/2} \abs{v_n}^2 /n, 1/n))} \Biggr) \prod_{n=1}^k  \left( (2\pi)^{-2} e^{ - \frac{n}{q^n} \abs{x_n}^2}  \frac{4 \pi n}{q^n} \right) \]
\[ \ll e^{ O ( \sum_{n=1}^k \min ( q^{n/2} \abs{v_n}^2/n, 1/n))}  \sum_{n=1}^k  \prod_{n=1}^k  \left( e^{ - \frac{n}{q^n} \abs{x_n}^2}  \frac{n }{\pi q^n} \right).\]
Plugging in the definition \eqref{vn-def} of $v_n$ and dividing by $\prod_{n=1}^k  \left( e^{ - \frac{n}{q^n} \abs{x_n}^2}  \frac{n }{\pi q^n} \right)$ gives \eqref{eq-pointwise-wint}. \end{proof}

\begin{corollary}\label{pointwise-wint-simplified} Assume $q>5$. Let $x_1,\dots, x_k$ be vectors in $\mathbb R^2$.   Then 
\begin{equation}\label{eq-pointwise-wint-simplified} \frac{ F(x_1,\dots, x_k)}{ \prod_{n=1}^k \left( e^{ - \frac{ n\abs{x_n}^2}{q^n}}   \frac{n}{q^n \pi} \right)} \leq  O( k^{O(1)} ) = O(N^{O(1)}). \end{equation}\end{corollary}

\begin{proof} We have  \begin{equation}\label{harmonic-series} \sum_{n=1}^k \min (  n q^{-3n/2} \abs{x_n}^2, n^{-1} ) \leq \sum_{n=1}^k n^{-1} = O(\log k)  \end{equation} and plugging \eqref{harmonic-series} into \eqref{eq-pointwise-wint}, together with the trivial bound $k\leq N$, gives \eqref{eq-pointwise-wint-simplified}. \end{proof}

\subsection{Hermite polynomial expansion bounds}

The goal of this subsection is to prove a formula, Corollary \ref{hermite-expansion}, for $F(x_1,\dots, x_k)$ as the product of a Gaussian probability density function times a sum of Hermite polynomials weighted by certain coefficients $h_{a_{1,1},\dots, a_{k,2}}$, together with bounds for the coefficients $h_{a_{1,1},\dots, a_{k,2}}$. The bounds for the coefficients will start in \cref{hermite-start} with a complicated bound expressed in terms of an integral and conclude in \cref{hermite-end} which bounds a sum of squares of the  $h_{a_{1,1},\dots, a_{k,2}}$ which exactly equals the error, in $L^2$ norm integrated against the Gaussian measure, of a low-degree polynomial approximation for $\frac{ F(x_1,\dots,x_k)}{\prod_{n=1}^k \left( e^{ - \frac{ n\abs{x_n}^2}{q^n}}   \frac{n}{q^n \pi} \right) }$ obtained using these Hermite polynomials. We begin with a brief review of Hermite polynomials.

The (probabilist's) Hermite polynomials are defined as:
\[  \mathit{He}_n ( x) = (-1)^n e^{ \frac{x^2}{2}} \frac{d^n}{ dx^n} e^{ - \frac{x^2}{2}} \]
and their key property is the orthogonality when integrated against the Gaussian measure
\begin{equation} \int_{-\infty}^{\infty}  \mathit{He}_n ( x)  \mathit{He}_m ( x)  \frac{ e^{ - \frac{x^2}{2}} }{\sqrt{2\pi}} dx = \begin{cases} n! & \textrm{if } n=m \\ 0 & \textrm{if } n \neq m \end{cases}. \end{equation}

After a change of variables, this implies for any $\sigma>0$
\begin{equation}\label{hermite-orthogonality} \int_{-\infty}^{\infty}  \mathit{He}_n \left( \frac{x}{\sigma} \right) \mathit{He}_m \left( \frac{x}{\sigma} \right)   \frac{ e^{ - \frac{x^2}{2\sigma^2}} }{\sqrt{2\pi} \sigma} dx = \begin{cases} n! & \textrm{if } n=m \\ 0 & \textrm{if } n \neq m \end{cases}. \end{equation}

We have
\[ \int_{-\infty}^{\infty}  e^{ i xy}  \mathit{He}_n ( x)   \frac{ e^{ - \frac{x^2}{2}} }{\sqrt{2\pi}} dx= \int_{-\infty}^{\infty}  e^{ i xy} (-1)^n \left( \frac{d^n}{ dx^n} \frac{e^{ - \frac{x^2}{2}}}{\sqrt{2\pi}}\right)  dx =  \int_{-\infty}^{\infty} \left(  \frac{d^n}{ dx^n}e^{ i xy} \right) \frac{ e^{ - \frac{x^2}{2}}}{\sqrt{2\pi}}  dx = y^n e^{ \frac{y^2}{2}} \]
and a change of variables gives
\begin{equation}\label{Hermite-Fourier}  \int_{-\infty}^{\infty} e^{ i xy}  \mathit{He}_n \left( \frac{x}{\sigma} \right)   \frac{ e^{ - \frac{x^2}{2\sigma^2 }} }{\sqrt{2\pi}\sigma} dx=  \sigma^n y^n e^{ \frac{\sigma^2 y^2}{2}}. \end{equation}

Now we introduce the notation that will be needed for our first bound on the coefficients.

Let $C_5$ be the implicit constant in the big $O$ in \cref{hybrid-simplified}.   Let $n$ be a positive integer. 

For $a$ a positive integer and $r$ a positive real number, write \[ \mathcal I_{n}(a,r)=  \frac{1}{\pi} \int_{-\infty}^\infty \frac{ e^{   C_5 \min ( q^{n/2} \abs{v }^2 /n, 1/n)} } {\abs{ v+ i  r}^{a+1} } dv .\] For $a=0$, write \[ \mathcal I_{n}(a,r)=1.\]

For $a_{n,1},a_{n,2}$ two nonnegative integers, let
\[ \mathcal L_n (a_{n,1},a_{n,2})= \inf_{ \substack{r_{n,1}, r_{n,2} \geq 0 \\ r_{n,j}=0 \textrm{ if and only if } a_{n,j}=0 }} \frac{ (1+ C_2 (r_{n,1}^2+r_{n,2}^2))^{A_n}   \mathcal S(\sqrt{ r_{n,1}^2+r_{n,2}^2 })^{E_n}  }{ e^{ - \frac{q^n}{n}\frac{r_{n,1}^2 + r_{n,2}^2 }{4}} } \mathcal I_n(a_{n,1},r_{n,1}) \mathcal I_n(a_{n,2},r_{n,2}).\]

For $w_n\in \mathbb R^2$, write $w_{n,1}$ for its first coordinate and $w_{n,2}$ for its second coordinate.

\begin{lemma}\label{hermite-start}  There exists a tuple $(h_{a_{1,1},\dots, a_{k,2}})_{ a_{1,1},\dots, a_{k,2} \in \mathbb Z^{\geq 0}}$ of complex numbers indexed by $2k$-tuples of nonnegative integers such that \[\mathbb E [ e^{  i \sum_{n=1}^k X_{n,\xi} \cdot w_n } ]  = e^{ - \sum_{n=1}^k \frac{q^n}{n} \frac{ \abs{w_n}^2}{4}} \sum_{ a_{1,1}, \dots, a_{k,2} \in \mathbb Z^{\geq 0}} h_{a_{1,1},\dots, a_{k,2}} \prod_{n=1}^k (w_{n,1}^{a_{n,1}} w_{n,2}^{a_{n,2}}) \]
and for each tuple $a_{1,1},\dots, a_{k,2}$ of nonnegative integers we have
\begin{equation}\label{eq-hln} \abs{h_{a_{1,1},\dots, a_{k,2}} } \leq \prod_{n=1}^k  \mathcal L_n ( a_{n,1}, a_{n,2} ).\end{equation} \end{lemma}

We will shortly see in \cref{hermite-expansion} that the $h_{a_{1,1},\dots, a_{k,2}}$ are the coefficients of an expansion of $F$ by Hermite polynomials. The remaining results in this subsection will be devoted to proving more straightforward upper bounds on the $h_{a_{1,1},\dots, a_{k,2}}$, culminating in \cref{hermite-end} which bounds a certain sum of $h_{a_{1,1},\dots, a_{k,2}}$ which will appear in the proof of \cref{intro-hf}.

\begin{proof} The fact that $h_{a_{1,1},\dots, a_{k,2}}$ exist and the sum is absolutely convergent follows from the fact that \[\frac{ \mathbb E [ e^{  i \sum_{n=1}^k X_{n,\xi} \cdot w_n } ] }{  e^{ - \sum_{n=1}^k \frac{q^n}{n} \frac{ \abs{w_n}^2}{4}} }\] is an entire function which is clear as the random variables $X_{n,\xi}$ are bounded so the numerator is entire while the denominator is entire and nowhere vanishing.

To estimate the coefficients, we use the Cauchy integral formula. We explain the argument in detail only in the case that $a_{1,1},\dots,a_{k,2}$ are all nonzero. The general case follows the same ideas, but is notationally more complicated.

 For $f$ a function of a complex variable $z$, which is bounded on loci in the complex plane where the imaginary part of $z$ is bounded, the coefficient of $z^a$ in the Taylor expansion at $0$ of $f$ is given for any $r>0$ by 
\[ \frac{1}{2\pi}  \left( \int_{ r i + \infty}^{r i - \infty} f(z) \frac{dz}{z^{a+1}} + \int_{ -ri- \infty}^{-ri +\infty} f(z) \frac{dz}{z^{a+1}}\right) \]
or writing $z= x+iy$, by
\[ \frac{1}{2\pi}  \left( \int_{ -\infty}^{\infty } f(x+ ir ) \frac{dx}{ (x+ir)^{a+1}} + \int_{ -\infty}^{\infty} f(x-ir ) \frac{dx}{(x-ir)^{a+1}}\right) .\] On the other hand, for $a=0$, the value is simply $f(0)$.

Thus for $f$ a function of complex variables $z_{1,1}, \dots, z_{k,2} $ which is bounded on loci in $\mathbb C^{2k}$ where the imaginary parts of all coordinates are bounded, the coefficient of $\prod_{n=1}^k ( z_{n,1}^{a_{n,1}} z_{n,2}^{a_{n,2}})$ in $f$ is given for any $r_{1,1},\dots, r_{k,2}>0$ by
\[ \frac{1}{(2\pi)^{2k} }\sum_{\epsilon_{1,1},\dots, \epsilon_{n,k}\in \pm 1} \int_{\mathbb R^{2k}} f( x_{1,1}+  i \epsilon_{1,1} r_{1,1},\dots, x_{k,2} + i \epsilon_{k,2} )  \frac{ dx_{1,1} \dots d x_{k,2}}{ \prod_{n=1}^k ( (x_{n,1}+  i \epsilon_{n,1} r_{n,1})^{a_{n,1}+1} (x_{n,2}+  i \epsilon_{n,2} r_{n,2})^{a_{n,2}+1} ) }. \]
If some of the $a_{n,j}$ are $0$, we can drop the corresponding variables $x_{n,j}$ from the integral as well as drop the sums over $\epsilon_{n,j}$ and a corresponding number of factors of $(2\pi)$.  We apply this to the function of  $u_1,\dots, u_k\in \mathbb C^2$, with $u_n = v_n + i w_n$, given by
\[ \frac{ \mathbb E [ e^{   \sum_{n=1}^k X_{n,\xi} \cdot u_n } ]  }{ e^{ - \sum_{n=1}^k \frac{q^n}{n} \frac{ u_n \cdot u_n}{4}}} \]
to obtain
\[ i^{ -\sum_{n=1}^k (a_{n,1} + a_{n,2})} h_{a_{1,1},\dots, a_{k,2}} \] \[  \hspace{-.5in}= \frac{1}{(2\pi)^{2k} }\sum_{\epsilon_{1,1},\dots, \epsilon_{k,2}\in \pm 1}  \int_{v_1,\dots, v_n\in \mathbb R^2} \frac{ \mathbb E [ e^{  \sum_{n=1}^k X_{n,\xi} \cdot v_n + i \sum_{n=1}^k X_{n,\xi} \cdot \tilde{w}_n } ]  } { e^{ \sum_{n=1}^k \frac{q^n}{n} \frac{ \abs{v_n}^2 + 2 i v_n\cdot \tilde{w}_n - \abs{\tilde{w}_n}^2}{4}}} \frac{ dv_1\dots dv_k}{\prod_{n=1}^k (( v_{n,1} + i \epsilon_{n,1} r_{n,1})^{a_{n,1}+1} ( v_{n,2} + i \epsilon_{n,2} r_{n,2})^{a_{n,2}+1} )}\]
where $\tilde{w}_n = ( \epsilon_{n,1} r_{n,1}, \epsilon_{n,2} r_{n,2})$. Taking absolute values and applying \cref{hybrid-simplified} gives
\[  \abs{ h_{a_{1,1},\dots, a_{k,2}} } \] \[ \hspace{-.5in}\leq  \frac{1}{(2\pi)^{2k} } \sum_{\epsilon_{1,1},\dots, \epsilon_{k,2}\in \pm 1}    \int_{v_1,\dots, v_n\in \mathbb R^2} \frac{ \abs{\mathbb E [ e^{  \sum_{n=1}^k X_{n,\xi} \cdot v_n + i \sum_{n=1}^k X_{n,\xi} \cdot \tilde{w}_n } ]}  } { e^{ \sum_{n=1}^k \frac{q^n}{n} \frac{ \abs{v_n}^2 - \abs{\tilde{w}_n}^2}{4}}} \frac{ dv_1\dots dv_k}{\prod_{n=1}^k ( \abs{ v_{n,1} + i  r_{n,1}}^{a_{n,1}+1} \abs{ v_{n,2} + i  r_{n,2}}^{a_{n,2}+1}) } \]
\[  \leq \frac{1}{(2\pi)^{2k} } \sum_{\epsilon_{1,1},\dots, \epsilon_{n,k}\in \pm 1}    \int_{v_1,\dots, v_n\in \mathbb R^2} \frac{\prod_{n=1}^k \Bigl (  e ^{ \frac{q^n}{n}   \frac{ \abs{v_n}^2}{4}  + C_5 \min ( q^{n/2} \abs{v_n}^2 /n, 1/n) }   (1+ C_2 \abs{\tilde{w}_n}^2)^{ A_n}    \mathcal S(\abs{\tilde{w}_n})^{E_n}  \Bigr) dv_1\dots dv_k } { e^{ \sum_{n=1}^k \frac{q^n}{n} \frac{ \abs{v_n}^2 - \abs{\tilde{w}_n}^2}{4}}{\prod_{n=1}^k ( \abs{ v_{n,1} + i  r_{n,1}}^{a_{n,1}+1} \abs{ v_{n,2} + i  r_{n,2}}^{a_{n,2}+1}}) } \]
\begin{equation}\label{tildew-to-w}=  \frac{1}{(2\pi)^{2k} } \sum_{\epsilon_{1,1},\dots, \epsilon_{k,2}\in \pm 1}  \prod_{n=1}^k  \frac{ (1+ C_2 \abs{\tilde{w}_n}^2)^{ A_n}  \mathcal S(\abs{\tilde{w}_n})^{E_n}  }{ e^{ - \frac{q^n}{n}\frac{ \abs{\tilde{w}_n}^2}{4}} }    \int_{v_1,\dots, v_n\in \mathbb R^2} \frac{\prod_{n=1}^k e ^{  C_5  \min ( q^{n/2} \abs{v_n}^2 /n, 1/n) }  dv_1\dots dv_k } {\prod_{n=1}^k ( \abs{ v_{n,1} + i  r_{n,1}}^{a_{n,1}+1} \abs{ v_{n,2} + i  r_{n,2}}^{a_{n,2}+1}) }. \end{equation}
For $w_n=  (r_{n,1}, r_{n,2})$, we have $\abs{\tilde{w}_n}=\abs{w_n}$ and since $\tilde{w}_n$ only appears in \eqref{tildew-to-w} via its absolute value, we may simplify by replacing $\tilde{w}_n$ by $w_n$ and then removing the sum over $\epsilon_{n,j}$, obtaining
\begin{equation}\label{c-medium-point}  \abs{ h_{a_{1,1},\dots, a_{k,2}} } \leq \frac{1}{\pi^{2k} }   \prod_{n=1}^k  \frac{ (1+ C_2 \abs{w_n}^2)^{ A_n}  \mathcal S(\abs{w_n})^{E_n}  }{ e^{ - \frac{q^n}{n}\frac{ \abs{w_n}^2}{4}} }   \prod_{n=1}^{k} \int_{v_n\in \mathbb R^2} \frac{ e^{  C_5 \min ( q^{n/2} \abs{v_n}^2 /n, 1/n) } } { \abs{ v_{n,1} + i  r_{n,1}}^{a_{n,1}+1} \abs{ v_{n,2} + i  r_{n,2}}^{a_{n,2}+1}} dv_n.\end{equation}
If some of the $a_{n,j}$ are $0$, we drop the corresponding variables $v_{n,j}$ from the integral in \eqref{c-medium-point} as well as a corresponding number of factors of $\pi$.

We can bound $e^{   C_5  \min ( q^{n/2} \abs{v_n}^2 /n, 1/n) }$ by $e^{   C_5  \min ( q^{n/2} \abs{v_{n,1} }^2 /n, 1/n) } e^{ C_5 \min ( q^{n/2} \abs{v_{n,2}}^2 /n, 1/n) }$ so the inner integral splits as a product
\[  \int_{v_n\in \mathbb R^2} \frac{ e^{   C_5 \min ( q^{n/2} \abs{v_n}^2 /n, 1/n) } } { \abs{ v_{n,1} + i  r_{n,1}}^{a_{n,1}+1} \abs{ v_{n,2} + i  r_{n,2}}^{a_{n,2}+1}} dv_n \leq \prod_{j=1}^2\int_{-\infty}^\infty \frac{ e^{  C_5  \min ( q^{n/2} \abs{v_{n,j} }^2 /n, 1/n) } } {\abs{ v_{n,j} + i  r_{n,j}}^{a_{n,j}+1} } dv_{n,j} \] which matches $\pi^2 \mathcal I_n ( a_{n,1} , r_{n,1}) \mathcal I_n (a_{n,2},r_{n,2})$, giving a bound of $ \prod_{n=1}^k  \mathcal L_n ( a_{n,1}, a_{n,2} )$ once we choose $r_{n,1},r_{n,2}$ to be values where the minimum in the definition of $\mathcal L_n$ is attained (or comes arbitrarily close to being attained). If some of the variables are $0$, since we drop the integral and the factor of $\pi$, we obtain $1$, again maching the definition of $\mathcal I_n ( a_{n,j} , r_{n,j })$. Here we use the fact that $e^{   C_5  \min ( q^{n/2} \abs{v_{n,1} }^2 /n, 1/n) } =1$ if $v_{n,j}=0$ . \end{proof}

We are now ready to state our Hermite expansion. The proof relies on \cref{hermite-end} below, but there is no circularity as Corollary \ref{hermite-expansion} is not used until the next section -- we state it here for motivation.
\begin{corollary}\label{hermite-expansion}  For $(h_{a_{1,1},\dots, a_{k,2}})_{ a_{1,1},\dots, a_{k,2} \in \mathbb Z^{\geq 0}}$  as in \cref{hermite-start} we have
 \[\hspace{-.4in} F(x_1,\dots, x_k) = \prod_{n=1}^k \left( e^{ - \frac{ n\abs{x_n}^2}{q^n}}   \frac{n}{q^n \pi} \right)  \sum_{ a_{1,1}, \dots, a_{k,2} \in \mathbb Z^{\geq 0}} h_{a_{1,1},\dots, a_{k,2}}  \prod_{n=1}^k  \textit{He}_{a_{n,1}} \left( x_{n,1} \sqrt{ \frac{2n}{q^n}} \right) \textit{He}_{a_{n,2}} \left(x_{n,2}  \sqrt{ \frac{2n}{q^n}} \right )  \sqrt{ \frac{2n}{q^n}}^{ a_{n,1}+ a_{n,2} }.\] \end{corollary}
 
 \begin{proof} We
 take Fourier transforms of both sides. Using \cref{hermite-start} to compute the Fourier transform of the left-hand side and \eqref{Hermite-Fourier} to compute the Fourier transform of the right-hand side, we see the Fourier transforms are equal. the The left-hand side is an $L^2$ function by \cref{pointwise-wint-simplified} and the right-hand side is $L^2$ by \cref{hermite-end} below and \cref{hermite-orthogonality}). By invertibility of the Fourier transform for $L^2$ functions, both sides are equal. \end{proof}

We are now ready to estimate the inegral $\mathcal I_n$ and local terms $\mathcal L_n$. The easiest, but most important, estimate is the following:

\begin{lemma}\label{L-constant-term}  We have \[ \mathcal L_n(0,0)\leq 1\] for all $n$. \end{lemma}

\begin{proof} We set $r_{n,1}=r_{n,2}=0$ and all the terms are manifstly equal to $1$ in this case, except for $\mathcal S(0)$, which is $\leq 1$ by \cref{hybrid-unified}. (In fact one can also check $\mathcal S(0)=1$ using \cref{hybrid-unified} but we don't need this.) \end{proof}

For $(a_{n,1},a_{n,2} ) \neq (0,0)$ it will suffice to bound $\mathcal L_{n}(a_{n,1},a_{n,2})$ to within a constant factor. To that end, we have the following bound for $\mathcal I_n$:

\begin{lemma}\label{ln-bound}  We have \[ \mathcal I_n(r,a) \ll \frac{1}{ r^a \sqrt{a+1}} \] where we adopt the convention $0^0=1$. \end{lemma}

\begin{proof} The case $a=0$ is $ 1 \ll 1 $ which is clear. For $a>0$ convexity of the logarithm gives the lower bound
\[ \abs{ v+ i  r}^{a+1} = ( v^2 + r^2 )^{ \frac{a+1}{2}}  = r^{a+1}    \left(\frac{v^2}{r^2} + 1\right)^{ \frac{a+1}{2}} \geq r^{ a+1}  \left( 1 + \frac{a+1}{2} \frac{v^2}{r^2} \right).\]
We also have \[ e^{   C_5 \min ( q^{n/2} \abs{v }^2 /n, 1/n)} \leq e^{C_5/n} \ll 1.\] Thus
\[ \mathcal I_n( a,r) \ll \int_{-\infty}^\infty \frac{1 } {\abs{ v+ i  r}^{a+1} } dv  \leq \int_{-\infty}^\infty \frac{1}{ r^{a+1}  \left( 1 + \frac{a+1}{2} \frac{v^2}{r^2} \right)} dv. \] The change of variables $x= \sqrt{ \frac{a+1}{2}} \frac{v}{r}$ gives
\[  \int_{-\infty}^\infty \frac{1}{ r^{a+1}  \left( 1 + \frac{a+1}{2} \frac{v^2}{r^2} \right)} dv= \frac{1}{r^a} \sqrt{\frac{2}{a+1}} \int_{\infty}^\infty \frac{1}{1+x^2} \ll  \frac{1}{ r^a \sqrt{a+1}}. \qedhere\]
\end{proof}

This allows us to prove the following general bound, where we have reintroduced $w_n$ for compactness of notation.

\begin{lemma}\label{Ln-uni} Fix integers $n>0$ and $ a_{n,0}, a_{n,1} \geq 0$. Let $r_{n,1},r_{n,2}$ be nonnegative real numbers such that $r_{n,j}=0$ if and only if $a_{n,j}=0$. Let $w_n=(r_{n,1},r_{n,2})$. Then
\[ \mathcal L_n (a_{n,1}, a_{n,2}) \ll \frac{ (1+ C_2 \abs{w_n}^2)^{A_n}   \mathcal S( \abs{w_n} )^{E_n}  }{ e^{ - \frac{q^n}{n}\frac{\abs{w_n}^2 }{4}}  a_{n,1}^{r_{n,1} } a_{n,2}^{r_{n,2}} \sqrt{a_{n,1}+1 }  \sqrt{a_{n,2}+1} }.\]
\end{lemma}
\begin{proof} This follows from the definition of $\mathcal L_n$, after observing that a minimum is bounded by its value at any point, applying \cref{ln-bound} to bound $\mathcal I_n(a_{n,j}, r_{n,j})$, and substituting $\abs{w_n}$ for $\sqrt{r_{n,1}^2+r_{n,2}^2}$. \end{proof}

We will specialize Lemma \ref{Ln-uni} at different values of $a_{n,1},a_{n,2}$ in different ranges. First we give a lemma helpful for $a_{n,1}+a_{n,2}$ small:

\begin{lemma}\label{Ln-short-range} For integers $n>0$ and $a_{n,1},a_{n,2}\geq 0$ we have 
\[   \mathcal L_n(a_{n,1}, a_{n,2} )  \ll e^{ O ( a_{n,1} + a_{n,2})} \left( \frac{q^n}{n} \right)^{ \frac{a_{n,1}+ a_{n,2}}{3}}  a_{n,1}^{ - \frac{a_{n,1}}{3}} a_{n,2}^{- \frac{a_{n,2}}{3}}\frac{1}{ \sqrt{a_{n,1}+1}  \sqrt{a_{n,2} +1}     }   .\]\end{lemma}

\begin{proof}  We apply Lemma \ref{Ln-uni} and set $r_{n,j}=  \sqrt[3] { \frac{na_{n,j}} {q^{n} } } $ . We  have \begin{equation*}
\mathcal S( \abs{w_n} ) \leq e^{ - \frac{ \abs{w_n}^2}{4} + O (\abs{w_n}^3)} \end{equation*} and \begin{equation*}
1 + C_2 \abs{w_n}^2 \leq e^{ C_2 \abs{w_n}^2} \end{equation*} so
\begin{equation}\label{mb-1-3} \begin{split} &  \frac{ (1+ C_2 \abs{w_n}^2)^{ A_n}  \mathcal S(\abs{w_n})^{E_n}  }{ e^{ - \frac{q^n}{n}\frac{ \abs{w_n}^2}{4}} } \leq e^{ A_n C_2 \abs{w_n}^2 + \frac{q^n}{n} \frac{\abs{w_n}^2}{4} - E_n \frac{\abs{w_n}^2}{4} + O ( E_n \abs{w_n}^3)} \\ & =  e^{ A_n C_2 \abs{w_n}^2 + B_n  \frac{\abs{w_n}^2}{4} + O ( E_n \abs{w_n}^3)} = e^{ O ( \frac{q^{n/2} }{n} \abs{w_n}^2 + \frac{q^n}{n} \abs{w_n}^3 )}  = e^{ O (  \frac{q^n}{n} \abs{w_n}^3 )} \end{split}\end{equation}  since if $a_{n,1}=a_{n,2}=0$ the exponent is $0$ and otherwise $\abs{w_n} \geq \sqrt[3]{ \frac{n}{q^{n}}} \geq q^{-n/2}$.
We have \begin{equation}\label{mb-1-4} \abs{w_n}^3 = (r_{n,1}^2 + r_{n,2}^2)^{3/2} = \frac{n}{q^n}  (  a_{n,1}^{2/3} + a_{n,2}^{2/3})^{3/2}  = O \left( \frac{n}{q^n} (a_{n,1} + a_{n,2} ) \right).\end{equation} 
Combining \eqref{mb-1-3} and \eqref{mb-1-4}, we have
\[   \frac{ (1+ C_2 \abs{w_n}^2)^{ A_n}  \mathcal S(\abs{w_n})^{E_n}  }{ e^{ - \frac{q^n}{n}\frac{ \abs{w_n}^2}{4}} r_{n,1}^{a_{n,1}} r_{n,2}^{a_{n,2} } } = \frac{ e^{ O ( a_{n,1} + a_{n,2})}}{ r_{n,1}^{a_{n,1} } r_{n,2}^{a_{n,2}}} =e^{ O ( a_{n,1} + a_{n,2})} \left( \frac{q^n}{n} \right)^{ \frac{a_{n,1}+ a_{n,2}}{3}}  a_{n,1}^{ - \frac{a_{n,1}}{3}} a_{n,2}^{- \frac{a_{n,2}}{3}} .\] Adding the $\sqrt{a_{n,1}+1}\sqrt{a_{n,2}+1}$ term, we obtain the statement.
\end{proof} 

Next we give a lemma for  $a_{n,1}+a_{n,2}$ large.

\begin{lemma}\label{Ln-long-range} There exists an absolute constant $C_6$ such that for integers $n>0$ and $a_{n,1},a_{n,2}\geq 0$ with  $a_{n,1} + a_{n,2}$ larger than $C_6 \frac{q^n}{n}$, we have
\[   \mathcal L_n(a_{n,1}, a_{n,2} )  \ll O(1)^{ A_n+E_n} \left( \frac{q^n}{2n} \right)^{ \frac{E_n}{4}- A_n+ \frac{a_{n,1}+ a_{n,2}}{2}} (a_{n,1}+a_{n,2})^{ A_n - \frac{E_n}{4}}  e^{  \frac{ a_{n,1} + a_{n,2}}{2}}a_{n,1}^{ - \frac{a_{n,1}}{2}} a_{n,2}^{ - \frac{a_{n,2}}{2}}\frac{1}{ \sqrt{a_{n,1}+1}  \sqrt{a_{n,2} +1}    }   .\] \end{lemma}

\begin{proof}  We apply Lemma \ref{Ln-uni} and set $r_{n,j}= = \sqrt{ \frac{2n}{q^n} a_{n,j} }$  We have \begin{equation*}
\abs{w_n} = \sqrt{ r_{n,1}^2 + r_{n,2}^2} = \sqrt{ \frac{2n}{q^n } (a_{n,1} + a_{n,2})} \end{equation*} which is larger than an absolute constant, so \begin{equation}\label{mb-2-2}  1+ C_2 \abs{w_n}^2= O ( \abs{w_n}^2) = O \left(  \frac{2n}{q^n } (a_{n,1} + a_{n,1}) \right) \end{equation} and \begin{equation}\label{mb-2-3} \mathcal S( \abs{w_n}) = O ( \abs{w_n}^{-\frac{1}{2}} ) = O\left ( \left( \frac{q^n}{2n} \right)^{\frac{1}{4}} (a_{n,1}+a_{n,2})^{-\frac{1}{4}} \right) \end{equation} while \begin{equation}\label{mb-2-4} e^{ - \frac{q^n}{n}\frac{ \abs{w_n}^2}{4}} = e^{ -  \frac{ a_{n,1} + a_{n,2}}{2}}\end{equation} and \begin{equation}\label{mb-2-5} r_{n,j}^{a_{n,j}} = \left( \frac{q^n}{2n}\right)^{ - \frac{a_{n,j}}{2}}  a_{n,j}^{  \frac{a_{n,j}} {2} } .\end{equation} Putting \eqref{mb-2-2}, \eqref{mb-2-3}, \eqref{mb-2-4}, and \eqref{mb-2-5} all together gives
\[   \frac{ (1+ C_2 \abs{w_n}^2)^{ A_n}  \mathcal S(\abs{w_n})^{E_n}  }{ e^{ - \frac{q^n}{n}\frac{ \abs{w_n}^2}{4}} r_{n,1}^{a_{n,1}} r_{n,2}^{a_{n,2} } } = O(1)^{ A_n+E_n} \left( \frac{q^n}{2n} \right)^{ \frac{E_n}{4}- A_n + \frac{a_{n,1}+ a_{n,2}}{2}} (a_{n,1}+a_{n,2})^{ A_n - \frac{E_n}{4}}  e^{  \frac{ a_{n,1} + a_{n,2}}{2}}a_{n,1}^{ - \frac{a_{n,1}}{2}} a_{n,2}^{ - \frac{a_{n,2}}{2}} .\] Adding the $\sqrt{a_{n,1}+1}\sqrt{a_{n,2}+1}$ term, we obtain the statement.\end{proof}

Our final lemma will be used for $a_{n,1}+a_{n,2}$ in an intermediate range.

\begin{lemma}\label{Ln-medium-range} For each $C_7$, there exists $\epsilon>0$ such that for integers $n>0$ and $a_{n,1},a_{n,2}\geq 0$ with $a_{n,1} + a_{n,2}$ less than $C_7 \frac{q^n}{n}$, we have
\[   \mathcal L_n(a_{n,1}, a_{n,2} )  \ll  e^{ \left( \frac{1}{2} - \epsilon \right) (a_{n,1}+ a_{n,2})}  \left( \frac{q^n}{2n} \right)^{  \frac{a_{n,1} + a_{n,2}}{2} } a_{n,1}^{ - \frac{a_{n,1}}{2}} a_{n,2}^{ - \frac{a_{n,2}}{2}} \frac{1}{ \sqrt{a_{n,1}+1}  \sqrt{a_{n,2} +1}    }   .\]  \end{lemma}

\begin{proof} For small values of $n$, there are finitely many possibilities and it suffices to apply \cref{Ln-short-range}, absorbing any discrepancies into the implicit constant, so we may assume $n$ is large. We apply Lemma \ref{Ln-uni} and set $r_{n,j}= = \sqrt{ \frac{2n}{q^n} a_{n,j} }$.  We have $\abs{w_n} <C_7 $.  There exists $\epsilon>0$ such that $\mathcal S( x) < e^{ -   \epsilon x^2}$ for all $x \leq C_7$ (since any $\epsilon< \frac{1}{4}$ works for $x$ sufficiently small and some $\epsilon$ works on every bounded interval away from $0$). We furthermore have 
\begin{equation}\label{mb-3-1} (1+ C_2 \abs{w_n}^2)^{ A_n}  \mathcal S(\abs{w_n})^{E_n} \leq e^{ C_2 A_n \abs{w_n}^2 - \epsilon E_n \abs{w_n}^2 } = e^{ C_2 A_n \abs{w_n}^2 + \epsilon B_n \abs{w_n}^2 -  \epsilon \frac{q^n}{n} \abs{w_n}^2 }    \leq e^{ -  \frac{\epsilon }{2} \frac{q^n}{n}  \abs{w_n}^2} \end{equation} for $n$ is sufficiently large. We also have
\begin{equation}\label{mb-3-5} r_{n,j}^{a_{n,j}} = \left( \frac{q^n}{2n}\right)^{ - \frac{a_{n,j}}{2}}  a_{n,j}^{  \frac{a_{n,j}} {2} } .\end{equation}  Using \eqref{mb-3-1} and \eqref{mb-3-5}, we have
\[   \frac{ (1+ C_2 \abs{w_n}^2)^{ A_n}  \mathcal S(\abs{w_n})^{E_n}  }{ e^{ - \frac{q^n}{n}\frac{ \abs{w_n}^2}{4}} r_{n,1}^{a_{n,1}} r_{n,2}^{a_{n,2} } }\leq  e^{ \left( \frac{1}{4} - \frac{ \epsilon}{2} \right) \frac{q^n}{n} \abs{w_n}^2 }   \left( \frac{q^n}{2n} \right)^{  \frac{a_{n,1} + a_{n,2}}{2} }  a_{n,1}^{ - \frac{a_{n,1}}{2}} a_{n,2}^{ - \frac{a_{n,2}}{2}} \] 
\[ =  e^{ \left( \frac{1}{2} - \epsilon \right) (a_{n,1}+ a_{n,2})}  \left( \frac{q^n}{2n} \right)^{  \frac{a_{n,1} + a_{n,2}}{2} }  a_{n,1}^{ - \frac{a_{n,1}}{2}} a_{n,2}^{ - \frac{a_{n,2}}{2}} . \] Adding the $\sqrt{a_{n,1}+1}\sqrt{a_{n,2}+1}$ term, we obtain the statement.\end{proof}

We are now ready to give our final bound we need on the $h_{a_{1,1},\dots, a_{k,2}}$s.

\begin{lemma}\label{hermite-end} If $q>11$ then for $(h_{a_{1,1},\dots, a_{k,2}})_{ a_{1,1},\dots, a_{k,2} \in \mathbb Z^{\geq 0}}$  as in \cref{hermite-start} we have
\[ \sum_{ \substack{ a_{1,1}, \dots, a_{k,2} \in \mathbb Z^{\geq 0} \\ \sum_{n=1}^k n (a_{n,1}+ a_{n,2}) > k}}  \abs{h_{a_{1,1},\dots, a_{k,2}} }^2 \prod_{n=1}^k a_{n,1}! a_{n,2}! \left( \frac{2n}{q^n} \right)^{ a_{n,1}+ a_{n,2} }  \ll_q  k ^{ - \frac{q-2}{2}} . \]  \end{lemma}

\begin{proof} It suffices to prove

\[ \sum_{ \substack{ a_{1,1}, \dots, a_{k,2} \in \mathbb Z^{\geq 0} } }  \Bigl( \min(k, \sum_{n=1}^k n (a_{n,1}+ a_{n,2}) \Bigr)^{ \frac{q}{2}}    \abs{h_{a_{1,1},\dots, a_{k,2}} }^2 \prod_{n=1}^k a_{n,1}! a_{n,2}! \left( \frac{2n}{q^n} \right)^{ a_{n,1}+ a_{n,2} }  = O_q(k)  . \]

The inequality $\sum_{n=1}^k  n (a_{n,1}+ a_{n,2}) \leq \prod_{n=1}^k (1 +  n (a_{n,1}+ a_{n,2}) )$ and \eqref{eq-hln} gives
\[ \sum_{ \substack{ a_{1,1}, \dots, a_{k,2} \in \mathbb Z^{\geq 0} } }  \Bigl( \min (k, \sum_{n=1}^k n (a_{n,1}+ a_{n,2}) )\Bigr)^{ \frac{q}{2}}    \abs{h_{a_{1,1},\dots, a_{k,2}} }^2 \prod_{n=1}^k a_{n,1}! a_{n,2}! \left( \frac{2n}{q^n} \right)^{ a_{n,1}+ a_{n,2} } \]
\[ \leq\sum_{ \substack{ a_{1,1}, \dots, a_{k,2} \in \mathbb Z^{\geq 0} } }  \prod_{n=1}^k \Bigl( \min(k,  1+  n (a_{n,1}+ a_{n,2}) )\Bigr)^{ \frac{q}{2}}    \abs{h_{a_{1,1},\dots, a_{k,2}} }^2 \prod_{n=1}^k a_{n,1}! a_{n,2}! \left( \frac{2n}{q^n} \right)^{ a_{n,1}+ a_{n,2} } \]
\[ \ll  \sum_{ \substack{ a_{1,1}, \dots, a_{k,2} \in \mathbb Z^{\geq 0} } }  \prod_{n=1}^k  \Biggl( \Bigl( \min(k,  1+  n (a_{n,1}+ a_{n,2})) \Bigr)^{ \frac{q}{2}} \mathcal L_n ( a_{n,1}, a_{n,2})    a_{n,1}! a_{n,2}! \left( \frac{2n}{q^n} \right)^{ a_{n,1}+ a_{n,2} }   \Biggr)\]
\[ =   \prod_{n=1}^k  \Biggl( \sum_{a_{n,1}, a_{n,2} =0}^{\infty} \Bigl( \min(k,  1+  n (a_{n,1}+ a_{n,2}) )\Bigr)^{ \frac{q}{2}} \mathcal L_n ( a_{n,1}, a_{n,2})    a_{n,1}! a_{n,2}! \left( \frac{2n}{q^n} \right)^{ a_{n,1}+ a_{n,2} }   \Biggr)\] so it suffices to show 
\begin{equation}\label{hermite-end-sts} \sum_{a_{n,1}, a_{n,2} =0}^{\infty} \Bigl( \min(k,  1+  n (a_{n,1}+ a_{n,2})) \Bigr)^{ \frac{q}{2}} \mathcal L_n ( a_{n,1}, a_{n,2})    a_{n,1}! a_{n,2}! \left( \frac{2n}{q^n} \right)^{ a_{n,1}+ a_{n,2} } = \begin{cases} O(k)  & \textrm{if } n=1 \\ 1 + O_q(1/n^2) & \textrm{ if } n>1 \end{cases}.\end{equation}
Removing the $a_{n,1},a_{n,2}=0$ term handled by \cref{L-constant-term}, this is equivalent to
\[ \sum_{\substack{a_{n,1}, a_{n,2} \in \mathbb Z^{\geq 0}\\ (a_{n,1},a_{n,2})\neq (0,0) }} \Bigl( \min (k,  n (a_{n,1}+ a_{n,2})) \Bigr)^{ \frac{q-2}{2}} \mathcal L_n ( a_{n,1}, a_{n,2})    a_{n,1}! a_{n,2}! \left( \frac{2n}{q^n} \right)^{ a_{n,1}+ a_{n,2} } =  \begin{cases} O(k)  & \textrm{if } n=1 \\  O_q(1/n^2) & \textrm{ if } n>1 \end{cases}.\]
Stirling's formula gives 
\begin{equation}\label{weak-stirling} a_{n,j}! \ll a_{n,j}^{a_{n,j} } e^{ - a_{n,j}} \sqrt{a_{n,j}+1}. \end{equation}

We split $a_{n,1} + a_{n,2}$ into three ranges.  We fix constants $C_8$ sufficiently small and $C_9$ sufficiently large. For $ a_{n,1}+ a_{n,2} \leq C_8 \frac{q^n}{n^7}$ we apply \cref{Ln-short-range} to bound $\mathcal L_n ( a_{n,1}, a_{n,2})$. For $a_{n,1}+ a_{n,2} \geq C_9 \frac{q^n}{n}$ we apply \cref{Ln-long-range}. For $a_{n,1},a_{n,2} \in (C_8 \frac{q^n}{n^7}, C_9 \frac{q^n}{n})$ we apply \cref{Ln-medium-range}. 

Applying \cref{Ln-long-range} and \eqref{weak-stirling} to the terms in \eqref{hermite-end-sts} with $a_{n,1}+ a_{n,2} \geq C_9 \frac{q^n}{n}$, we obtain

\begin{equation}\label{hermite-l2-long}  \sum_{\substack{a_{n,1}, a_{n,2} \in \mathbb Z^{\geq 0}\\ a_{n,1}+ a_{n,2}  \geq C_9 \frac{q^n}{n} }} \Bigl( \min(k,   n (a_{n,1}+ a_{n,2})) \Bigr)^{ \frac{q}{2}} \mathcal L_n ( a_{n,1}, a_{n,2})    a_{n,1}! a_{n,2}! \left( \frac{2n}{q^n} \right)^{ a_{n,1}+ a_{n,2} } \end{equation}
\[ \ll  \sum_{\substack{a_{n,1}, a_{n,2} \in \mathbb Z^{\geq 0}\\ a_{n,1}+ a_{n,2}  \geq C_9 \frac{q^n}{n} }} \Bigl( \min(k,   n (a_{n,1}+ a_{n,2})) \Bigr)^{ \frac{q}{2}} O(1)^{ A_n+E_n} \left( \frac{q^n}{2n} \right)^{ \frac{E_n}{2}- 2A_n} (a_{n,1}+a_{n,2})^{ 2A_n - \frac{E_n}{2}}   \frac{1}{ \sqrt{a_{n,1}+1}  \sqrt{a_{n,2} +1}  }\]
\[ \ll \sum_{\substack{a \in \mathbb Z^{\geq 0}\\ a \geq C_9 \frac{q^n}{n} }} \Bigl( \min(k,   n a) \Bigr)^{ \frac{q}{2}} O(1)^{ A_n+E_n} \left( \frac{q^n}{2n} \right)^{ \frac{E_n}{2}- 2A_n}  a^{ 2A_n - \frac{E_n}{2} }. \]

We now handle the cases $n=1$ and $n>1$ separately. For $n=1$, we have $E_n=q$ and $A_n=0$ and the terms depending on $q$ and $n$ may be absorbed into the implicit constant. This gives
\[\sum_{\substack{a \in \mathbb Z^{\geq 0}\\ a \geq C q }} \Bigl( \min (k,    a) \Bigr)^{ \frac{q}{2}}  a^{ - \frac{q}{2} } .\]

The terms where $a \leq k$ contribute at most $ \sum_{ a=1}^k 1 =k $ and the remaining terms contribute $k^{ \frac{q}{2}} \sum_{a=k+1}^\infty  a^{- \frac{q}{2}} =O(k )$, so this indeed gives $O(k)$.

For $n>1$ and at all subsequent points in the argument, we will ignore the $\min(k , \cdot )$. This gives 
\[\sum_{\substack{a \in \mathbb Z^{\geq 0}\\ a \geq C_9 \frac{q^n}{n} }}  n^{ \frac{q}{2}}  O(1)^{ A_n+E_n} \left( \frac{q^n}{2n} \right)^{ \frac{E_n}{2}- 2A_n}  a^{ 2A_n - \frac{E_n}{2} + \frac{q}{2}} \]
\[ \ll n^{ \frac{q}{2}}  O(1)^{ A_n+E_n} \left( \frac{q^n}{2n} \right)^{ \frac{E_n}{2}- 2A_n}  \left( C_9 \frac{q^n}{n} \right)^{ 2 A_n - \frac{E_n}{2}+ \frac{q}{2} +1} \]
\[ = n^{ \frac{q}{2}}  O(1)^{ A_n+E_n}  \left( C_9 \right)^{ 2 A_n - \frac{E_n}{2}+ \frac{q}{2}+1 }  \left( \frac{q^n}{n}  \right)^{ \frac{q}{2}+1}. \]
Since $q> 11$,  by \eqref{q11}, the exponent $ \frac{E_n}{2}- 2 A_n - \frac{q}{2}-1$ is greater than a constant multiple of $A_n + E_n $ so choosing $C_9$ sufficiently large the $\left( C_9 \right)^{ 2 A_n - \frac{E_n}{2}+ \frac{q}{2}} $ term dominates $O(1)^{A_n+E_n}  $ by a factor that is doubly exponential in $n$. Since $ \left( \frac{q^n}{n}  \right)^{ \frac{q}{2}}$ and $n^{ \frac{q-2}{2}} $ are at most singly exponential in $n$, they are easily dominated and the product is $O(1/n^2)$. So indeed \eqref{hermite-l2-long} is $O(1/n^2)$ for $n>1$ and $O( k)$ for $n=1$.

Applying \cref{Ln-medium-range} and \eqref{weak-stirling} to the terms in \eqref{hermite-end-sts} with $a_{n,1}+ a_{n,2} \in (C_8 \frac{q^n}{n^7}, C_9 \frac{q^n}{n})$, we obtain 
\begin{equation}\label{hermite-l2-medium}  \sum_{\substack{a_{n,1}, a_{n,2} \in \mathbb Z^{\geq 0}\\ a_{n,1}+ a_{n,2}  \in ( C_8 \frac{q^n}{n^7},  C_9 \frac{q^n}{n} ) }} \left( n (a_{n,1}+ a_{n,2})  \right)^{ \frac{q}{2}} \mathcal L_n ( a_{n,1}, a_{n,2})    a_{n,1}! a_{n,2}! \left( \frac{2n}{q^n} \right)^{ a_{n,1}+ a_{n,2} } \end{equation}
\[ \ll \sum_{\substack{a_{n,1}, a_{n,2} \in \mathbb Z^{\geq 0}\\ a_{n,1}+ a_{n,2}  \in ( C_8 \frac{q^n}{n^7},  C_9 \frac{q^n}{n} ) }} \left( n (a_{n,1}+ a_{n,2})  \right)^{ \frac{q}{2}} e^{ - 2\epsilon  (a_{n,1}+ a_{n,2})}   \frac{1}{ \sqrt{a_{n,1}+1}  \sqrt{a_{n,2} +1}    } \]
\[ \leq   \sum_{\substack{a_{n,1}, a_{n,2} \in \mathbb Z^{\geq 0}\\ a_{n,1}+ a_{n,2}  \in ( C_8 \frac{q^n}{n^7},  C_9 \frac{q^n}{n} ) }}  \left( n C_9 \frac{q^n}{n} \right)^{\frac{q}{2}} e^{-2 \epsilon C_8 \frac{q^n}{n^7}}  \]
\[ \ll \left(\frac{q^n}{2} \right)^2 \left( n C_9 \frac{q^n}{n} \right)^{\frac{q}{2}} e^{-2 \epsilon C_8 \frac{q^n}{n^7}}\]
and the term $e^{-2 \epsilon C_8 \frac{q^n}{n^7}}$ decreases doubly exponential in $n$ while the remaining terms increase singly-exponentially so the product decreases doubly-exponentially and in particular is $O(1/n^2)$. 

Applying \cref{Ln-short-range} and \eqref{weak-stirling} to the terms in \eqref{hermite-end-sts} with $a_{n,1}+ a_{n,2} \leq C_8 \frac{q^n}{n^7}$, we obtain
\begin{equation}\label{hermite-l2-short}  \sum_{\substack{a_{n,1}, a_{n,2} \in \mathbb Z^{\geq 0}\\ a_{n,1}+ a_{n,2}  \in (0,  C_8 \frac{q^n}{n^7}]  }} \left( n (a_{n,1}+ a_{n,2})  \right)^{ \frac{q}{2}} \mathcal L_n ( a_{n,1}, a_{n,2})    a_{n,1}! a_{n,2}! \left( \frac{2n}{q^n} \right)^{ a_{n,1}+ a_{n,2} } \end{equation}
\[ \ll \sum_{\substack{a_{n,1}, a_{n,2} \in \mathbb Z^{\geq 0}\\ a_{n,1}+ a_{n,2}  \in (0,  C_8 \frac{q^n}{n^7}]  }} e^{ O ( a_{n,1} + a_{n,2})} \left( \frac{q^n}{n} \right)^{-  \frac{a_{n,1}+ a_{n,2}}{3}}  a_{n,1}^{ \frac{a_{n,1}}{3}} a_{n,2}^{ \frac{a_{n,2}}{3}}\frac{1}{ \sqrt{a_{n,1}+1}  \sqrt{a_{n,2} +1}     }   \]
\[ \leq \sum_{\substack{a_{n,1}, a_{n,2} \in \mathbb Z^{\geq 0}\\ a_{n,1}+ a_{n,2}  \in (0,  C_8 \frac{q^n}{n^7}]  }} e^{ O ( a_{n,1} + a_{n,2})} \left( \frac{q^n}{n} \right)^{-  \frac{a_{n,1}+ a_{n,2}}{3}}  (C_8 q^n/n^7) ^{ \frac{a_{n,1}}{3}} (C_8 q^n/n^7) ^{ \frac{a_{n,2}}{3}}  \]
\[ = \sum_{\substack{a_{n,1}, a_{n,2} \in \mathbb Z^{\geq 0}\\ a_{n,1}+ a_{n,2}  \in (0,  C_8 \frac{q^n}{n}]  }}  ( C_8^{ 1/3} e^{O(1)} n^{-2} )^{ a_{n,1} + a_{n,2}} \leq  \sum_{\substack{a_{n,1}, a_{n,2} \in \mathbb Z^{\geq 0}\\ a_{n,1}+ a_{n,2}  >0  }}  ( C_8^{ 1/3} e^{O(1)} n^{-2} )^{ a_{n,1} + a_{n,2}} .\]
We have \[  \sum_{\substack{a_{n,1}, a_{n,2} \in \mathbb Z^{\geq 0}\\ a_{n,1}+ a_{n,2}  >0  }} x^{ a_{n,1} + a_{n,2}} = \frac{ 2x-x^2}{ (1-x)^2 } = O(x) \] for $x$ sufficiently small so, taking $C_8$ sufficiently small, this is $O ( C_8^{1/3} e^{O(1)} n^{-2} ) = O(n^{-2})$, as desired.\end{proof}

\section{The chimera}

In this section, we prove \cref{intro-support}, \cref{intro-lf}, and \cref{intro-hf}. The proof of \cref{intro-support} is direct and independent of the prior results. To prove \cref{intro-lf} and \cref{intro-hf}, we combine estimates from the previous section on the function $F$ with estimates from the literature on the measure $\mu_{\textrm{rm}}$.

\begin{proof}[Proof of \cref{intro-support}]  It suffices to check first that the support of $\mu_{\textrm{ch}}$ is contained in the support of $\mu_{\textrm{rm}}$ and second that the support of $\mu_{\textrm{ch}}$ is contained in the support of $\mu_{\textrm{ep}}$.

Since $\mu_{\textrm{rm}}$ is the pushforward of the Haar measure on $U(N)$, and the support of Haar measure is all of $U(N)$, the support of $\mu_{\textrm{rm}}$ is (the closure of) the image of $U(N)$. Since $\mu_{\textrm{ch}}$ is also the pushforward of a measure on $U(N)$, its support is also contained in (the closure of) the image of $U(N)$ and hence in the support of  $\mu_{\textrm{rm}}$.

For a point $L^* \in \Cqs$ to be contained in the support of $\mu_{\textrm{ep}}$, each neighborhood of $L^*$ must contain a random Euler product $L_\xi$ with positive probability. Since the topology is the product topology, a basis for the neighborhood consists of the sets of power series in $q^{-s}$ whose first $n$ coefficients are all within $\epsilon$ of the first $n$ coefficients of $L^*$. Whether $L_\xi$ lies in this neighborhood only depends on $\xi(\mathfrak p)$ for $\mathfrak p$ of degree $\leq n$. The set of functions from $\mathfrak p$ of degree $\leq n$ to the circle a finite-dimensional manifold, the uniform measure on this manifold is supported everywhere, and the map from this to the first $n$ coefficients of $L_\xi$, so it suffices for $L^*$ to be in the image of this manifold. In other words, it suffices to check there is a single function $\xi$ from primes to the unit circle such that first $n$ coefficients of $L_\xi$ agree with the first $n$ coefficients of $L^*$.

If $L^*$ lies in the support of $\mu_{\textrm{ch}}$ then there is certainly a function $\xi$ where the first $k$ coefficients of $L_\xi$ agree with $L^*$ since  $L^*= \det ( I - q^{ \frac{1}{2} - s } M) $ for some $M$ with \[ F( - q^{1/2} \tr(M),\dots, -q^{k/2} \tr (M^k)/k)>0\] where $F$ is the probability density function of the first $k$ coefficients of the logarithm of a random $L_\xi$. Since $F( - q^{1/2} \tr(M),\dots, -q^{k/2} \tr (M^k)/k)>0$, the density is nonzero, so there exists $\xi$ such that the first $k$ coefficients of $\log L_\xi$  match $ - q^{1/2} \tr(M),\dots, -q^{k/2} \tr (M^k)/k$.

We now check by induction on $n$ that for each $n \geq k$ there exists a function $\xi_n$ such that the first $n$ coefficients of  $\log L_{\xi_n}$ match  $ - q^{1/2} \tr(M),\dots, -q^{n/2} \tr (M^n)/n$. For $n=k$ this is what we just checked. For $n=k$, we choose $\xi_n$ so that $\xi_n(f) =\xi_{n-1}(f)$ for all $f$ of degree $<n$, so that the first $n-1$ coefficients of $L_{\xi_n}$ match the first $n-1$ coefficients of $L_{\xi_{n-1}}$. The $n$th coefficient of $\log L_\xi$ is then 
\[ \sum_{ \substack{ \mathfrak p \in \mathbb F_q[u]^+ \\ \textrm{irreducible} \\ \deg \mathfrak p=n}} \xi_n(\mathfrak p) + \sum_{ \substack{ \mathfrak p \in \mathbb F_q[u]^+ \\ \textrm{irreducible} \\ \deg \mathfrak p \mid n \\ \deg \mathfrak p \neq n }} \xi_n(\mathfrak p)/ (n/\deg \mathfrak p) \] so it suffices to choose $\xi_n$ so that
\[ \sum_{ \substack{ \mathfrak p \in \mathbb F_q[u]^+ \\ \textrm{irreducible} \\ \deg \mathfrak p=n}} \xi_n(\mathfrak p) = -q^{n/2} \tr (M^n)/n -  \sum_{ \substack{ \mathfrak p \in \mathbb F_q[u]^+ \\ \textrm{irreducible} \\ \deg \mathfrak p \mid n \\ \deg \mathfrak p \neq n }} \xi_n(\mathfrak p)/ (n/\deg \mathfrak p) \]
We can choose $\sum_{ \substack{ \mathfrak p \in \mathbb F_q[u]^+ \\ \textrm{irreducible} \\ \deg \mathfrak p=n}} \xi_n(\mathfrak p)$ to be any complex number of absolute value $\leq  E_n$ so it suffices to check the right hand side has absolute value $\leq E_n$. We have \[ \sum_{ \substack{ \mathfrak p \in \mathbb F_q[u]^+ \\ \textrm{irreducible} \\ \deg \mathfrak p \mid n \\ \deg \mathfrak p \neq n }} \xi_n(\mathfrak p)/ (n/\deg \mathfrak p) = O( q^{n/2})\] and \[ \abs{ q^{n/2} \tr (M^n)/n} \leq q^{n/2} N/n \]  

Since $E_n$ is greater than a constant times $q^n/n$, the right hand side is is $\leq E_n$ as long as $q^{n/2} > O(N + n)$ which happens as long as  $n/ \log N$ is sufficiently large. Since $n>k$, this happens if $k/\log N$ is sufficiently large, and since $k = \lfloor N^{\beta} \rfloor \geq \lfloor N^{1/4} \rfloor$, this happens for $N$ sufficiently large.\end{proof}

For this section, a convenient coordinate system for $\Cqs$ consists of the variables $b_n$ defined so that $b_n(L)$ is $\sqrt{\frac{n}{q^n}}$ times the coefficient of $q^{-ns}$ in $\log L$, so that \[ L = e^{ \sum_{n=1}^{\infty} \sqrt{\frac{q^n}{n}} b_n(L) q^{-s}}.\] Thus \begin{equation}\label{anll} b_n(L_\xi ) = \sqrt{\frac{n}{q^n}}X_{n,\xi}\end{equation} and because the definition \eqref{LM-def} of $L_M$ gives \[L_M (s)= \det ( I - q^{ \frac{1}{2} - s } M) = e^{ \sum_{n=1}^{\infty} q^{ \frac{n}{2} } q^{-ns} \tr(M^n)/n } ,\] we have \begin{equation}\label{anlm} b_n(L_M) = - \frac{  \tr( M^n)}{\sqrt{n}} .\end{equation}

Recall the measures $\mu_{\textrm{ep}}$ and $\mu_{\textrm{rm}}$ on $\Cqs$.  We define another measure $\mu_{\textrm{g}}$ on $\Cqs$ as the unique measure where the $b_n$ are independent complex standard normal random variables and the constant coefficient is $1$. The utility of $\mu_{\textrm{g}}$ for our purposes is that it serves as an approximation for both $\mu_{\textrm{ep}}$ and $\mu_{\textrm{rm}}$.

In particular, define a projection map $\eta\colon \mathbb C[[q^{-s} ]] \to \mathbb C^k$ that sends $L$ to $b_1(L),\dots, b_k(L)$.

The next two results give strong estimates, in different forms, comparing $\mu_{\textrm{rm}}$ and $\mu_{g}$. 

\begin{theorem}[]\cite[Proposition 1.2]{JohanssonLambert}\label{JL} For any $\beta<\frac{1}{2}$, setting $k = \lfloor N^\beta \rfloor$, as long as $N$ is sufficiently large in terms of $\beta$, the total variation distance between the pushforward measures $\eta_* \mu_{\textrm{rm}}^N$ and $\eta_* \mu_{\textrm{g}}$ is \[ \leq e^{- (1- o_N(1)) N^{1-\beta} \log (N^{1-\beta})} \] where $o_N(1)$ goes to $0$ as $N$ goes to $\infty$ for any fixed $\beta$. \end{theorem}

\begin{proof} This is a restatement of \cite[Proposition 1.2]{JohanssonLambert}, with a simplified but weaker bound. We explain how our notation compares. 

It is immediate from the definitions that $\eta_* \mu_{\textrm{g}}$ is a product of $k$ independent standard complex Gaussian distributions. Taking the real and imaginary parts, and multiplying by $-\sqrt{2}$, we obtain a product of $2k$ independent standard real Gaussians, exactly what is called $\mathbf{G}$ in \cite{JohanssonLambert}. (Their $m$ is our $k$).

Similarly, $\eta_* \mu_{\textrm{rm}}$ is the distribution of $b_1(M),\dots, b_k(M) =  -\tr(M^1)/\sqrt{1},\dots, -  \tr(M^k)/\sqrt{k}$. Taking real and imaginary parts and multiplying by $- \sqrt{2}$, this is exactly the distribution called $\mathbf X$ in \cite{JohanssonLambert}.

The total variation bound follows from \cite[Proposition 1.2]{JohanssonLambert}, noting that the factor $1.4 \cdot 10^{-13} n^{ {3\beta}-\frac{3}{2}}$ is $\leq 1$ and can be ignored. \end{proof}

\begin{lemma}\label{DS}\cite[Theorem 2]{DiaconisShahshahani} Let $\phi \in \mathbb C[ c_0, c_1,\dots, \overline{c_0},\overline{c_1},\dots] $ have degree $\leq N $. Then \[ \int_{\Cqs} \phi \mu_{\textrm{g} } = \int_{\Cqs} \phi \mu_{\textrm{rm}}^N.\] \end{lemma}

We use the total variation distance between measures to control differences between integrals against the measures in a couple different ways:

\begin{lemma} For probability measures $\mu_1,\mu_2$ with total variation distance $\delta$ and a function $f$, we have
\begin{equation}\label{tv-infty} \abs{ \int f \mu_1 - \int f \mu_2} \leq  2 \delta \sup \abs{f} \end{equation}
\begin{equation}\label{tv-2} \abs{ \int f \mu_1 - \int f \mu_2} \leq \sqrt{\delta} \left( \sqrt{ \int \abs{f}^2 \mu_1} + \sqrt{ \int \abs{f}^2 \mu_2} \right) \end{equation} \end{lemma}

\begin{proof}By the definition of total variation distance, we can write $\mu_1 - \mu_1'  =\mu_2 -\mu_2'$ where $\mu_1' \leq \mu_1 $ and $\mu_2'\leq \mu_2$ are measures with total mass $\delta$. We obtain \eqref{tv-infty} by noting the integral of $f$ against any measure is bounded by the sup-norm of $f$ times the total mass of that measure.  We obtain \eqref{tv-2} by applying Cauchy-Schwarz to $f$ and the density functions $\frac{\mu_1'}{\mu_1}$ and $\frac{\mu_2'}{\mu_2}$, which are $1$-bounded and integrate to $\delta$ and thus have $L^2$ norm at most $\sqrt{\delta}$. \end{proof}

We repeatedly use the following lemma to compare integrals against different measures:

\begin{lemma} Let $ G$ be a measurable function of complex variables $b_1,\dots, b_k$. We have
\begin{equation}\label{matrix-change-of-integral} \int_{U(N) }  G( - \tr(M),\dots, - \tr(M^k)/\sqrt{k} ) \mu_{\textrm{Haar}} =  \int_{ \Cqs}  G( b_1(L),\dots, b_k(L))  \mu_{\textrm{rm}} = \int_{\mathbb C^k}  G(b_1,\dots,b_k)  \eta_* \mu_{\textrm{rm}}.\end{equation} 
and
\begin{equation}\label{mf-change-of-integral} \int_{\Cqs }  G(b_1(L),\dots, b_k(L)) \mu_{\textrm{ep}} = \int_{\mathbb C^k}  G(b_1,\dots,b_k) \eta_* \mu_{\textrm{ep}} = \int_{\mathbb C^k}  G(b_1,\dots,b_k) \frac{ F( q^{1/2} b_1 ,\dots,  q^{k/2}  b_k/ \sqrt{k} )}{  \prod_{j=1}^k  ( e^{  -  |b_j |^2 }  \frac{ j }{ q^j \pi} ) }  \eta_* \mu_{\textrm{g}} .\end{equation}

\end{lemma}

\begin{proof} We prove \eqref{matrix-change-of-integral} first. Both equalities follow from the fact that the integral of a function against the pushforward of a measure is the integral of the pullback against the measure. For the first equality, we also use that $\mu_{\textrm{rm}}$ is the pushforward of $\mu_{\textrm{Haar}}$ along $M \mapsto L_M$ by definition, and use \eqref{anlm} to compute the pullback of $ G$ along $M \mapsto L_M$. 

We now prove \eqref{mf-change-of-integral}. The first equality is similar to \eqref{matrix-change-of-integral}. For the second equality, we use  that the probability density function of $\eta_* \mu_{\textrm{g}}$, in the variables $b_1,\dots, b_k$, is  $\prod_{j=1}^k  \left( e^{  -  \abs{b_j }^2 }  \frac{1 }{ \pi } \right)$, so we have
\[ \int_{\mathbb C^k}   G(b_1,\dots, b_k) \frac{ F( q^{1/2} b_1 ,\dots,  q^{k/2}  b_k/ \sqrt{k} )}{  \prod_{j=1}^k  ( e^{  -  |b_j |^2 }  \frac{ j }{ q^j \pi} ) } \eta_* \mu_{\textrm{g}} = \int_{\mathbb C^k}   G(b_1,\dots, b_k) \frac{ F( q^{1/2} b_1 ,\dots, -q^{k/2}  b_k/ \sqrt{k} )}{  \prod_{j=1}^k  \left(  \frac{ j }{ q^j } \right) } db_1 \dots db_k.\]

Because $F$ is the probability density function of the tuple of random variables $X_{1,\xi},\dots, X_{k,\xi} $, and by \eqref{anll} we have $b_n(L_\xi) = \sqrt{\frac{n}{q^n}}X_{n,\xi}$, it follows that $\frac{ F( q^{1/2} b_1 ,\dots, -q^{k/2}  b_k/ \sqrt{k} )}{  \prod_{j=1}^k  \left(  \frac{ j }{ q^j } \right) }$ is the probability density function of the tuple of random variables $b_1(L_\xi),\dots, b_k(L_\xi)$, i.e. of the measure $\eta_* \mu_{\textrm{ep}}$, giving
\[ \int_{\mathbb C^k}   G(b_1,\dots, b_k) \frac{ F( q^{1/2} b_1 ,\dots, -q^{k/2}  b_k/ \sqrt{k} )}{  \prod_{j=1}^k  \left(  \frac{ j }{ q^j } \right) } db_1 \dots db_k=  \int_{\mathbb C^k}  G(b_1,\dots,b_k) \eta_* \mu_{\textrm{ep}}. \qedhere \] \end{proof}

The first step to proving the main theorems is to evaluate $\gamma$:

\begin{lemma}\label{total-mass}  If $q>5$ then for $N$ sufficiently large in terms of $\beta$, \begin{equation}\label{eq-tm} \int_{U(N)}   \frac{ F( -q^{1/2} \tr(M),\dots, -q^{k/2} \tr (M^k)/k)}{  \prod_{j=1}^k  \left( e^{  -  \frac{  \abs{ \tr (M^j)}^2 }{  j } }  j q^j / \pi \right) } \mu_{\textrm{Haar} } = 1 + O ( e^{ -(1- o_N(1)) N^{1-\beta} \log (N^{1-\beta})}  ) . \end{equation} \end{lemma}
In other words, the $\gamma$ in the definition \eqref{weighted-def} of $\mu_{\textrm{weighted}}$ is $1 + O ( e^{ -(1- o_N(1)) N^{1-\beta} \log (N^{1-\beta})}  ) $ for $N$ sufficiently large. 

\begin{proof}  \eqref{matrix-change-of-integral} gives
\begin{equation}\label{tm-vc} \int_{U(N)} \frac{ F( - q^{1/2} \tr(M),\dots, -q^{k/2} \tr (M^k)/k)}{  \prod_{j=1}^k  \left( e^{  -  \frac{  \abs{ \tr (M^j)}^2 }{  j } }   \frac{ j }{ q^j \pi} \right) }  \mu_{\textrm{Haar} }= \int_{\mathbb C^k} \frac{ F(  q^{1/2} b_1 ,\dots, q^{k/2}  b_k / \sqrt{k} )}{  \prod_{j=1}^k  \left( e^{  -  |b_j| ^2  }  \frac{ j }{ q^j \pi} \right) } \eta _* \mu_{\textrm{rm}}. \end{equation}

Similarly, \eqref{mf-change-of-integral} gives
\begin{equation}\label{tm-tm} \int_{\mathbb C^k} \frac{ F(  q^{1/2} b_1 ,\dots, q^{k/2}  b_k / \sqrt{k} )}{  \prod_{j=1}^k  \left( e^{  -  |b_j| ^2  }  \frac{ j }{ q^j \pi} \right) } \eta _* \mu_{\textrm{g}} = \int_{\mathbb C^k} 1 \mu_{\textrm{ep}} = 1 .\end{equation}

Next we will prove \begin{equation}\label{tm-error} \int_{\mathbb C^k} \frac{ F(  q^{1/2} b_1 ,\dots, q^{k/2}  b_k / \sqrt{k} )}{  \prod_{j=1}^k  \left( e^{  -  |b_j| ^2  }  \frac{ j }{ q^j \pi} \right) } \eta _*\mu_{\textrm{rm}}  -\frac{ F(  q^{1/2} b_1 ,\dots, q^{k/2}  b_k / \sqrt{k} )}{  \prod_{j=1}^k  \left( e^{  -  |b_j| ^2  }  \frac{ j }{ q^j \pi} \right) }\eta _* \mu_{\textrm{g}} =  O ( e^{ -(1- o_N(1)) N^{1-\beta} \log (N^{1-\beta})}) . \end{equation} To check \eqref{tm-error}, we apply \eqref{tv-infty}. We use \cref{JL} to bound the total variation distance by $e^{- (1- o_N(1)) N^{1-\beta} \log (N^{1-\beta})}$. We apply \cref{pointwise-wint-simplified} to bound the sup-norm by $O(N^{O(1)})$, and note that the $O(N^{O(1)})$ can be absorbed into the $o_N(1)$ in the exponent.

Combining \eqref{tm-vc}, \eqref{tm-tm}, and \eqref{tm-error}, we obtain \eqref{eq-tm}.\end{proof}

The main part of proving \cref{intro-lf} is the following:

\begin{lemma}\label{lf-basic-comparison} Assume that $q>5$. Let $\phi \in \mathbb C[ c_1,c_2,\dots, \overline{c_1}, \overline{c_2},\dots] $ have degree $\leq k$. For $N$ sufficiently large in terms of $\beta$, we have
\begin{equation}\label{lf-bc-1} \int_{\Cqs }   \frac{ F(  q^{1/2} b_1 ,\dots, q^{k/2}  b_k / \sqrt{k} )}{  \prod_{j=1}^k  \left( e^{  -  |b_j| ^2  }  \frac{ j }{ q^j \pi} \right) }  \phi  \mu_{\textrm{g}} = \int_{\Cqs }  \phi  \mu_{\textrm{ep}}  \end{equation} and
\begin{equation}\label{lf-bc-2} \int_{\Cqs }  \phi  \mu_{\textrm{ep}}  =  \int_{\Cqs }    \frac{ F(  q^{1/2} b_1 ,\dots, q^{k/2}  b_k / \sqrt{k} )}{  \prod_{j=1}^k  \left( e^{  -  |b_j| ^2  }  \frac{ j }{ q^j \pi} \right) }  \phi    \mu_{\textrm{rm}}  + O (e^{ -(\frac{1}{2} - o_N(1)) N^{1-\beta} \log (N^{1-\beta})} \norm{\phi}_2 ) .\end{equation}
\end{lemma}

\begin{proof} First note that, since $\phi$ has degree $\leq k$, we can express $\phi$ as a polynomial only in terms of $c_1,c_2,\dots, c_k, \overline{c_1} , \overline{c_2},\dots, \overline{c_k}$. Since $c_1,\dots, c_k$ can be expressed as polynomials in $b_1,\dots,b_k$, it follows that $\phi$ can be expressed as a polynomial in $b_1,\dots,b_k, \overline{b_1},\dots, \overline{b_k}$, which we will refer to as $\tilde{\phi}$. Then we have
\begin{equation}\label{lf-bc-proof-1} \int_{\Cqs }  \phi  \mu_{\textrm{ep}} = \int_{\mathbb C^k}   \frac{ F(  q^{1/2} b_1 ,\dots, q^{k/2}  b_k / \sqrt{k} )}{  \prod_{j=1}^k  \left( e^{  -  |b_j| ^2  }  \frac{ j }{ q^j \pi} \right) } \tilde{\phi}  \eta _* \mu_{\textrm{g}} \end{equation} by \eqref{mf-change-of-integral},
\begin{equation}\label{lf-bc-proof-2} \int_{\Cqs }   \frac{ F(  q^{1/2} b_1 ,\dots, q^{k/2}  b_k / \sqrt{k} )}{  \prod_{j=1}^k  \left( e^{  -  |b_j| ^2  }  \frac{ j }{ q^j \pi} \right) }  \phi  \mu_{\textrm{g}} = \int_{\mathbb C^k}   \frac{ F(  q^{1/2} b_1 ,\dots, q^{k/2}  b_k / \sqrt{k} )}{  \prod_{j=1}^k  \left( e^{  -  |b_j| ^2  }  \frac{ j }{ q^j \pi} \right) } \tilde{\phi}  \eta _* \mu_{\textrm{g}}\end{equation}
by compatibility of integration with pushforward of measures, and
\begin{equation}\label{lf-bc-proof-3} \int_{\Cqs }    \frac{ F(  q^{1/2} b_1 ,\dots, q^{k/2}  b_k / \sqrt{k} )}{  \prod_{j=1}^k  \left( e^{  -  |b_j| ^2  }  \frac{ j }{ q^j \pi} \right) }  \phi    \mu_{\textrm{rm}} =\int_{\Cqs }    \frac{ F(  q^{1/2} b_1 ,\dots, q^{k/2}  b_k / \sqrt{k} )}{  \prod_{j=1}^k  \left( e^{  -  |b_j| ^2  }  \frac{ j }{ q^j \pi} \right) }  \tilde{\phi}    \eta _* \mu_{\textrm{rm}} \end{equation}
by \eqref{matrix-change-of-integral}.

Combining \eqref{lf-bc-proof-1} and \eqref{lf-bc-proof-2}, we obtain \eqref{lf-bc-1}. From \eqref{lf-bc-proof-1} and \eqref{lf-bc-proof-3}, we see that to prove \eqref{lf-bc-2}, it suffices to prove that
\begin{equation}\hspace{-.5in} \label{lf-bc-simp} \int_{\mathbb C^k}   \frac{ F(  q^{1/2} b_1 ,\dots, q^{k/2}  b_k / \sqrt{k} )}{  \prod_{j=1}^k  \left( e^{  -  |b_j| ^2  }  \frac{ j }{ q^j \pi} \right) } \tilde{\phi}  \eta _* \mu_{\textrm{g}} = \int_{\Cqs }    \frac{ F(  q^{1/2} b_1 ,\dots, q^{k/2}  b_k / \sqrt{k} )}{  \prod_{j=1}^k  \left( e^{  -  |b_j| ^2  }  \frac{ j }{ q^j \pi} \right) }  \tilde{\phi}    \eta _* \mu_{\textrm{rm}} + O (e^{- (\frac{1}{2} - o_N(1)) N^{1-\beta} \log (N^{1-\beta})} \norm{\phi}_2 ) .\end{equation}
We will do this by applying \eqref{tv-2} to $\eta _* \mu_{\textrm{rm}}$ and $\eta _* \mu_g$. We have by \cref{pointwise-wint-simplified} and \cref{DS}
\begin{equation}\label{lf-bc-g} \begin{split} & \int_{ \mathbb C^k} \abs{    \frac{ F(  q^{1/2} b_1 ,\dots, q^{k/2}  b_k / \sqrt{k} )}{  \prod_{j=1}^k  \left( e^{  -  |b_j| ^2  }  \frac{ j }{ q^j \pi} \right) }}^2 |\tilde{\phi}|^2 \eta _* \mu_{\textrm{g}} \\   \ll N^{O(1)}& \int_{\mathbb C^k}  |\tilde{\phi}|^2 \eta _* \mu_{\textrm{g}} = N^{O(1)}  \int_{ \Cqs}  \abs{\phi}^2 \mu_{\textrm{g} }= N^{O(1)} \int_{ \Cqs}  \abs{\phi}^2 \mu_{\textrm{rm} } = N^{O(1)} \norm{\phi} \end{split} \end{equation}
and an identical argument, except skipping the \cref{DS} step, gives
\begin{equation}\label{lf-bc-rm} \int_{ \mathbb C^k} \abs{    \frac{ F(  q^{1/2} b_1 ,\dots, q^{k/2}  b_k / \sqrt{k} )}{  \prod_{j=1}^k  \left( e^{  -  |b_j| ^2  }  \frac{ j }{ q^j \pi} \right) }}^2 |\tilde{\phi}|^2 \eta _* \mu_{\textrm{rm}} \ll N^{O(1)} \norm{\phi} .\end{equation}
Plugging \eqref{lf-bc-g}, \eqref{lf-bc-rm}, and \cref{pointwise-wint-simplified} into \cref{tv-2}, we obtain 
\[\int_{\mathbb C^k}   \frac{ F(  q^{1/2} b_1 ,\dots, q^{k/2}  b_k / \sqrt{k} )}{  \prod_{j=1}^k  \left( e^{  -  |b_j| ^2  }  \frac{ j }{ q^j \pi} \right) } \tilde{\phi}  \eta _* \mu_{\textrm{g}} \] \[= \int_{\Cqs }    \frac{ F(  q^{1/2} b_1 ,\dots, q^{k/2}  b_k / \sqrt{k} )}{  \prod_{j=1}^k  \left( e^{  -  |b_j| ^2  }  \frac{ j }{ q^j \pi} \right) }  \tilde{\phi}    \eta _* \mu_{\textrm{rm}} + O (N^{O(1)} e^{ -(\frac{1}{2} - o_N(1)) N^{1-\beta} \log (N^{1-\beta})} \norm{\phi}_2 )\] which, absorbing the $N^{O(1)}$ into the $o_{N(1)}$, gives \eqref{lf-bc-simp}. \end{proof}

For the next two proofs,  we observe that for $\phi \in \mathbb C[c_0,c_1,\dots, \overline{c_0},\overline{c_1},\dots]$ we have
\begin{equation}\label{int-gamma} \int_{\Cqs }  \phi    \mu_{\textrm{ch}}  =  \int_{U(N)}  \phi(L_M)    \mu_{\textrm{weighted}}  = \gamma \int_{U(N) } \phi(L_M)    \frac{ F( -q^{1/2} \tr(M),\dots, -q^{k/2} \tr (M^k)/k)}{  \prod_{j=1}^k  \left( e^{  -  \frac{  \abs{ \tr (M^j)}^2 }{  j } }  j q^j / \pi \right) } \mu_{\textrm{Haar} } \end{equation} by the definition \eqref{weighted-def} of $\mu_{\textrm{weighted}}$.

\begin{proof}[Proof of \cref{intro-lf}]
We have \begin{equation}\label{lf-key-integral}  \begin{split} &\int_{U(N) } \phi(L_M)    \frac{ F( -q^{1/2} \tr(M),\dots, -q^{k/2} \tr (M^k)/k)}{  \prod_{j=1}^k  \left( e^{  -  \frac{  \abs{ \tr (M^j)}^2 }{  j } }  j q^j / \pi \right) } \mu_{\textrm{Haar} }  = \int_{\Cqs }\phi    \frac{ F(  -q^{1/2} b_1 ,\dots, -q^{k/2}  b_k / \sqrt{k} )}{  \prod_{j=1}^k  \left( e^{  -  \abs{b_j}^2  }  \frac{ j }{ q^j \pi} \right) }  \mu_{\textrm{rm}}  \\ =   &\int_{\Cqs }  \phi  \mu_{\textrm{ep}}  +  O (e^{ -(\frac{1}{2} - o_N(1)) N^{1-\beta} \log (N^{1-\beta})} \norm{\phi}_2 ) \end{split}  \end{equation} by \eqref{matrix-change-of-integral} and \cref{lf-basic-comparison}.  Combining \eqref{int-gamma} and \eqref{lf-key-integral} gives exactly the main term and error term of \eqref{eq-intro-lf} except with an extra factor of $\gamma$. Using \cref{total-mass} to estimate $\gamma$, 
we see that multiplying by $\gamma$ introduces an additional error term of size  $\abs{ \int_{\Cqs }  \phi  \mu_{\textrm{ep}} } O( e^{- (1 - o_N(1) )N^{1-\beta} \log (N^{1-\beta})} )$. But 
\[ \abs{ \int_{\Cqs }  \phi  \mu_{\textrm{ep}} } \leq  \int_{\Cqs } \abs{ \phi}   \mu_{\textrm{ep}} = \int_{\Cqs } \abs{ \phi} \frac{ F(  -q^{1/2} b_1 ,\dots, -q^{k/2}  b_k / \sqrt{k} )}{  \prod_{j=1}^k  \left( e^{  -  \abs{b_j}^2  }  \frac{ j }{ q^j \pi} \right) }     \mu_{\textrm{g}}  \leq O ( k^{O(1) } ) \int_{\Cqs } \abs{ \phi} \mu_{\textrm{g}} \leq \] \[ O ( k^{O(1) } ) \sqrt{\int_{\Cqs } \abs{\phi}^2 \mu_{\textrm{g}}}= O ( k^{O(1) } ) \sqrt{\int_{\Cqs } \abs{\phi}^2 \mu_{\textrm{rm}}}= O( k^{O(1)}) \norm{\phi}_2 \] by the trivial bound, \eqref{mf-change-of-integral}, \cref{pointwise-wint-simplified}, Cauchy-Schwarz, \cref{DS}, and definition, so this error term can be absorbed into the $O (e^{ -(\frac{1}{2} - o_N(1)) N^{1-\beta} \log (N^{1-\beta})} \norm{\phi}_2 )$ error term, giving \eqref{eq-intro-lf}.
\end{proof} 

\begin{proof}[Proof of \cref{intro-hf}] Fix $\phi \in \mathbb C[ c_0, c_1,\dots, \overline{c_0},\overline{c_1},\dots]$ such that for all $\psi\in \mathbb C[ c_0, c_1,\dots, \overline{c_0},\overline{c_1},\dots]$ of degree $\leq k$ we have
\[\int_{U(N)} \phi(L_M) \overline{\psi(L_M)} \mu_{\textrm{Haar}} =0 .\]

Then for any real-valued $\psi\in \mathbb C[ c_0, c_1,\dots, \overline{c_0},\overline{c_1},\dots]$ of degree $\leq k$, we have
\begin{equation}\label{l2-error} \begin{split}
&\int_{U(N)} \phi(L_M)  \frac{ F( -q^{1/2} \tr(M),\dots, -q^{k/2} \tr (M^k)/k)}{  \prod_{j=1}^k  \left( e^{  -  \frac{  \abs{ \tr (M^j)}^2 }{  j } }  j q^j / \pi \right) }  \mu_{\textrm{Haar}} \\
= &  \int_{U(N)} \phi(L_M) \Biggl( \frac{ F( -q^{1/2} \tr(M),\dots, -q^{k/2} \tr (M^k)/k)}{  \prod_{j=1}^k  \left( e^{  -  \frac{  \abs{ \tr (M^j)}^2 }{  j } }  j q^j / \pi \right) } - \psi (L_M)  \Biggr)  \mu_{\textrm{Haar}} + \int_{U (N)} \phi(L_M) \psi(L_M) \mu_{\textrm{Haar}} \\
 = &\int_{U(N)} \phi(L_M) \Biggl( \frac{ F( -q^{1/2} \tr(M),\dots, -q^{k/2} \tr (M^k)/k)}{  \prod_{j=1}^k  \left( e^{  -  \frac{  \abs{ \tr (M^j)}^2 }{  j } }  j q^j / \pi \right) } - \psi (L_M)  \Biggr)  \mu_{\textrm{Haar}} +0 \\
 \leq & \norm{\phi}_2 \sqrt{ \int_{U(N)} \Biggl(  \frac{ F( -q^{1/2} \tr(M),\dots, -q^{k/2} \tr (M^k)/k)}{  \prod_{j=1}^k  \left( e^{  -  \frac{  \abs{ \tr (M^j)}^2 }{  j } }  j q^j / \pi \right) } - \psi (L_M) \Biggr)^2 \mu_{\textrm{Haar}} }.
 \end{split}\end{equation}

We furthermore have
\[  \int_{U(N)} \Biggl( \frac{ F( -q^{1/2} \tr(M),\dots, -q^{k/2} \tr (M^k)/k)}{  \prod_{j=1}^k  \left( e^{  -  \frac{  \abs{ \tr (M^j)}^2 }{  j } }  j q^j / \pi \right) } - \psi (L_M) \Biggr)^2 \mu_{\textrm{Haar}} \]
\[=  \int_{\Cqs} \Biggl( \frac{ F( -q^{1/2} b_1 ,\dots, -q^{k/2}  b_k/\sqrt{k} )}{  \prod_{j=1}^k  \left( e^{  -  \abs{b_j}^2 }  j q^j / \pi \right) } - \psi  \Biggr)^2 \mu_{\textrm{rm}} \]
\[ \hspace{-.4in}= \int_{\Cqs} \Biggl( \frac{ F( -q^{1/2}b_1,\dots, -q^{k/2}b_k /\sqrt{k})}{  \prod_{j=1}^k  \left( e^{  -   \abs{ b_k}^2}  j q^j / \pi \right) }\Biggr)^2 \mu_{\textrm{rm} }- 2 \int_{\Cqs} \frac{ F( -q^{1/2} b_1 ,\dots, -q^{k/2} b_k/\sqrt{k} )}{  \prod_{j=1}^k  \left( e^{  -\abs{b_j}^2 }  j q^j / \pi \right) }  \psi \mu_{\textrm{rm} }+ \int_{\Cqs} \psi^2 \mu_{\textrm{rm}}.\]

We now compare each of these integrals against $\mu_{\textrm{rm}}$ to corresponding integrals against $\mu_{\textrm{g}}$. First, we have 
\begin{equation}\label{hf-third-term} \int_{\Cqs} \psi^2 \mu_{\textrm{rm}} = \int_{\Cqs} \psi^2 \mu_{\textrm{g}}\end{equation} by \cref{DS} since $\psi^2$ has degree $2k \leq N$.

Second, using \eqref{matrix-change-of-integral} and then \eqref{tv-infty} (inputting \cref{JL} and \cref{pointwise-wint-simplified}), we obtain
\begin{equation}\label{hf-first-term} \begin{split} 
 &\int_{\Cqs} \Biggl( \frac{ F( -q^{1/2}b_1,\dots, -q^{k/2}b_k /\sqrt{k})}{  \prod_{j=1}^k  \left( e^{  -   \abs{ b_k}^2}  j q^j / \pi \right) }\Biggr)^2 \mu_{\textrm{rm} }  \\
 =& \int_{\mathbb C^k} \Biggl( \frac{ F( -q^{1/2}b_1,\dots, -q^{k/2}b_k /\sqrt{k})}{  \prod_{j=1}^k  \left( e^{  -   \abs{ b_k}^2}  j q^j / \pi \right) }\Biggr)^2 \eta _* \mu_{\textrm{rm} }  \\
 =& \int_{\mathbb C^k} \Biggl( \frac{ F( -q^{1/2}b_1,\dots, -q^{k/2}b_k /\sqrt{k})}{  \prod_{j=1}^k  \left( e^{  -   \abs{ b_k}^2}  j q^j / \pi \right) }\Biggr)^2 \nu_* \mu_{\textrm{g} }  +O ( N^{O(1)}  e^{ -(1- o_N(1)) N^{1-\beta} \log(N^{1-\beta} )})   \\
 = & \int_{\Cqs} \Biggl( \frac{ F( -q^{1/2}b_1,\dots, -q^{k/2}b_k /\sqrt{k})}{  \prod_{j=1}^k  \left( e^{  -   \abs{ b_k}^2}  j q^j / \pi \right) }\Biggr)^2 \mu_{\textrm{g} }+O ( e^{ -(1- o_N(1)) N^{1-\beta} \log(N^{1-\beta} )}) .\end{split} \end{equation}
Third, \cref{lf-basic-comparison} gives
\begin{equation}\label{hf-second-term}  \begin{split} & \int_{\Cqs} \frac{ F( -q^{1/2} b_1 ,\dots, -q^{k/2} b_k/\sqrt{k} )}{  \prod_{j=1}^k  \left( e^{  -\abs{b_j}^2 }  j q^j / \pi \right) }  \psi \mu_{\textrm{rm} }\\
= & \int_{\Cqs} \frac{ F( -q^{1/2} b_1 ,\dots, -q^{k/2} b_k/\sqrt{k} )}{  \prod_{j=1}^k  \left( e^{  -\abs{b_j}^2 }  j q^j / \pi \right) }  \psi \mu_{\textrm{g} }+ O (e^{ - (\frac{1}{2} - o_N(1)) N^{1-\beta} \log (N^{1-\beta})} \norm{\psi}_2 ) . \end{split} \end{equation} 

Combining \eqref{hf-third-term}, \eqref{hf-first-term}, and \eqref{hf-second-term}, we obtain
\[   \int_{\Cqs} \Biggl( \frac{ F( -q^{1/2} b_1 ,\dots, -q^{k/2}  b_k/\sqrt{k} )}{  \prod_{j=1}^k  \left( e^{  -  \abs{b_j}^2 }  j q^j / \pi \right) } - \psi  \Biggr)^2 \mu_{\textrm{rm}} \]
\begin{equation}\label{two-approximation-errors} \hspace{-.4in} =  \int_{\Cqs} \Biggl( \frac{ F( -q^{1/2} b_1 ,\dots, -q^{k/2}  b_k/\sqrt{k} )}{  \prod_{j=1}^k  \left( e^{  -  \abs{b_j}^2 }  j q^j / \pi \right) } - \psi  \Biggr)^2 \mu_{\textrm{g}} + O ( e^{ -(1- o_N(1)) N^{1-\beta} \log(N^{1-\beta} )}) + O (e^{ - (\frac{1}{2} - o_N(1)) N^{1-\beta} \log (N^{1-\beta})} \norm{\psi}_2 ). \end{equation}
To minimize \eqref{two-approximation-errors}, we should choose $\psi$ to be a good approximation in $L^2$ to $\frac{ F( -q^{1/2} b_1 ,\dots, -q^{k/2}  b_k/\sqrt{k} )}{  \prod_{j=1}^k  \left( e^{  -  \abs{b_j}^2 }  j q^j / \pi \right) }$. To do this we follow Corollary \ref{hermite-expansion} and set
\[ \psi =  \sum_{ \substack{ a_{1,1}, \dots, a_{k,2} \in \mathbb Z^{\geq 0}\\ \sum_{n=1}^k n (a_{n,1} + a_{n,2}) \leq k }} h_{a_{1,1},\dots, a_{k,2}}  \prod_{n=1}^k  \textit{He}_{a_{n,1}} \left( x_{n,1} \sqrt{ \frac{2n}{q^n}} \right) \textit{He}_{a_{n,2}} \left(x_{n,2}  \sqrt{ \frac{2n}{q^n}} \right )  \sqrt{ \frac{2n}{q^n}}^{ a_{n,1}+ a_{n,2} }\]
so that
\[ \frac{ F( -q^{1/2} b_1 ,\dots, -q^{k/2}  b_k/\sqrt{k} )}{  \prod_{j=1}^k  \left( e^{  -  \abs{b_j}^2 }  j q^j / \pi \right) } - \psi  \] \[ =  \sum_{ \substack{ a_{1,1}, \dots, a_{k,2} \in \mathbb Z^{\geq 0}\\ \sum_{n=1}^k n (a_{n,1} + a_{n,2}) > k }} h_{a_{1,1},\dots, a_{k,2}}  \prod_{n=1}^k  \textit{He}_{a_{n,1}} \left( x_{n,1} \sqrt{ \frac{2n}{q^n}} \right) \textit{He}_{a_{n,2}} \left(x_{n,2}  \sqrt{ \frac{2n}{q^n}} \right )  \sqrt{ \frac{2n}{q^n}}^{ a_{n,1}+ a_{n,2} }\]
and then
\[ \int_{\Cqs} \Biggl( \frac{ F( -q^{1/2} b_1 ,\dots, -q^{k/2}  b_k/\sqrt{k} )}{  \prod_{j=1}^k  \left( e^{  -  \abs{b_j}^2 }  j q^j / \pi \right) } - \psi  \Biggr)^2 \mu_{\textrm{g}} \]
\[ = \int_{\Cqs}  \Biggl(  \sum_{ \substack{ a_{1,1}, \dots, a_{k,2} \in \mathbb Z^{\geq 0}\\ \sum_{n=1}^k n (a_{n,1} + a_{n,2}) > k }} h_{a_{1,1},\dots, a_{k,2}}  \prod_{n=1}^k  \textit{He}_{a_{n,1}} \left( x_{n,1} \sqrt{ \frac{2n}{q^n}} \right) \textit{He}_{a_{n,2}} \left(x_{n,2}  \sqrt{ \frac{2n}{q^n}} \right )  \sqrt{ \frac{2n}{q^n}}^{ a_{n,1}+ a_{n,2} } \Biggr)^2 \mu_{\textrm{g}} \]
\[ = \sum_{ \substack{ a_{1,1}, \dots, a_{k,2} \in \mathbb Z^{\geq 0}\\ \sum_{n=1}^k n (a_{n,1} + a_{n,2}) > k }} \abs{h_{a_{1,1},\dots, a_{k,2}} }^2 \prod_{n=1}^k a_{n,1}! a_{n,2}! \left( \frac{2n}{q^n} \right)^{ a_{n,1}+ a_{n,2} }   \ll k^{ - \frac{q-2}{2}}  \ll N^{ - \beta \frac{q-2}{2}} \] by \eqref{hermite-orthogonality} and \cref{hermite-end}.

We also have
\[ \norm{\psi}_2^2= \int_{\Cqs} \abs{\psi}^2 \mu_{\textrm{rm}} = \int_{\Cqs} \abs{\psi}^2 \mu_{\textrm{g}} \] \[  \leq \int_{\Cqs} \Biggl( \frac{ F( -q^{1/2} b_1 ,\dots, -q^{k/2}  b_k/\sqrt{k} )}{  \prod_{j=1}^k  \left( e^{  -  \abs{b_j}^2 }  j q^j / \pi \right) }  \Biggr)^2  \mu_{\textrm{g}} \leq \int_{\Cqs} N^{O(1) }\mu_{\textrm{g}} = N^{O(1)}. \]

Plugging these into \eqref{two-approximation-errors}, we obtain
\[  \int_{\Cqs} \Biggl( \frac{ F( -q^{1/2} b_1 ,\dots, -q^{k/2}  b_k/\sqrt{k} )}{  \prod_{j=1}^k  \left( e^{  -  \abs{b_j}^2 }  j q^j / \pi \right) } - \psi  \Biggr)^2 \mu_{\textrm{rm}} \]
\[ = O ( N^{ - \beta \frac{q-2}{2}}) + O ( e^{ -(1- o_N(1)) N^{1-\beta} \log(N^{1-\beta} )}) + O (e^{ - (\frac{1}{2} - o_N(1)) N^{1-\beta} \log (N^{1-\beta})}N^{O(1)} )\]
\[ = O ( N^{ - \beta \frac{q-2}{2}}) \]
as the polynomial error term dominates the exponential ones. Plugging this into \eqref{l2-error} and using \eqref{int-gamma} and the fact from \cref{total-mass} that $\gamma=O(1)$, this gives \eqref{eq-intro-hf}.
\end{proof}

\section{Representations of the unitary group and moments}

This section is devoted to the proof of \cref{intro-cfkrs}. We begin in \S\ref{ss-reps} by describing the relationship between irreducible representations and polynomials in $\mathbb C[c_0,c_1,\dots, \overline{c_0},\overline{c_1},\dots]$, and between highest weights of representations and the degrees of polynomials. Using this, in \S\ref{ss-error} we will prove \cref{lem-mip}, a variant of \cref{intro-cfkrs} with the same error term but a main term expressed very differently in terms of a sum over irreducible representations. The proof of \cref{lem-mip} is very general and should apply with minimal modification beyond moments to other statistics of $L$-functions such as zero densities and ratios.  The next steps are to give in \S\ref{ss-main} an explicit expression, avoiding the language of representation theory, for the main term in \cref{lem-mip}, and to compare the main term of \cref{lem-mip} with the main term of \cref{intro-cfkrs}, leading to a proof of \cref{intro-cfkrs} at the end of this section. It would be possible to avoid the language of representation theory entirely, only writing down explicit polynomials and using the Weyl integration formula, but doing this would make our calculations less motivated.

\subsection{Preliminaries on representations and polynomials}\label{ss-reps}

Irreducible representations of the unitary group $U(N)$ are classified by their highest weight, a nonincreasing $N$-tuple of integers. An irreducible representation of $GL_N(\mathbb C)$ has highest weight $\omega_1,\dots,\omega_N$ if and only if it contains a vector on which upper-triangular unipotent matrices act trivially and on which the diagonal matrix with diagonal entries $\lambda_1,\dots,\lambda_N$ acts by multiplication by $\prod_{\ell=1}^N \lambda_\ell^{\omega_\ell}$. An irreducible representation of $U(N)$ has highest weight $\omega_1,\dots,\omega_N$ if and only if it extends to a representation of $GL_N(\mathbb C)$ with highest weight $\omega_1,\dots,\omega_N$. We denote the highest weight of $V$ by $\weight(V)$. 

We say the norm of the highest weight $\omega_1,\dots,\omega_N$ is $\sum_{\ell=1}^N \abs{\omega_\ell}$. The norm of $\weight(V)$ is denoted by $\norm{\weight(V)}$.

In this subsection, we will check that polynomials of degree $\leq k$ in $\mathbb C[c_0,c_1,\dots, \overline{c_0},\overline{c_1},\dots]$ may be expressed as linear combinations of the characters of irreducible representations with highest weight of norm $\leq k$ and, conversely, characters of such irreducible representations may be expressed as low-degree polynomials. It follows that characters of irreducible representations with highest weight of norm $>k$ are orthogonal to all low-degree polynomials, an important criterion for applying Theorem \ref{intro-hf}.

\begin{lemma}\label{tensor-product-weights} Let $V_1$ and $V_2$ be irreducible representations of $U(N)$. Then $V_1\otimes V_2$ is a sum of irreducible representations of $U(N)$ with highest weights of norms $\leq \norm{\weight(V_1)} + \norm{\weight(V_2)}$. \end{lemma}

\begin{proof} If $V_1 \otimes V_2$ contains an irreducible summand with highest weight $\omega_1,\dots, \omega_N$ then it contains an eigenvector of the diagonal torus with weights $\omega_1,\dots, \omega_N$. Since $V_1$ and $V_2$ split as sums of eigenspaces of the diagonal torus, this is only possible if $V_1$ and $V_2$ each contain an eigenvector whose weights sum to $\omega_1,\dots, \omega_N$. By linearity of the norm, it suffices to prove that the weights of eigenvectors of $V_j$ have norm at most $\norm{\weight(V_j)}$. If this were not so, since the set of weights is $S_N$-invariant, there would have to be a vector whose weight had greater norm with weights in decreasing order, which could not be a sum of the highest weights and negative roots, contradicting the fact that the representation is generated by the highest weight under the negative roots. \end{proof}

For $M \in U(N)$ the definition \eqref{LM-def} of $L_M$ implies that \begin{equation}\label{l-wedge} L_M (s) = \sum_{d=0}^n (-1)^n \tr (M, \wedge^d \std ) q^{d \left(\frac{1}{2}-s \right)}\end{equation} and \begin{equation}\label{l-bar-wedge} \overline{L_M(s)}= \sum_{d=0}^n (-1)^n \tr (M, \wedge^d \std^\vee ) q^{d\left( \frac{1}{2}  -\overline{s}\right)}.\end{equation} 

\begin{lemma}\label{low-degree-is-low-weight} For $\phi \in \mathbb C[c_0,c_1,\dots, \overline{c_0},\overline{c_1},\dots]$, there exists a finite set of irreducible representations $V_o$ and coefficients $\kappa_o$ such that for all $M\in U(N)$ we have
\[ \phi(L_M) = \sum_o \kappa_o \tr (M, V_o) \]
and if $\phi$ has degree $\leq d$ then we can assume that $\norm{\weight(V_o)}\leq d$ for all $o$.
\end{lemma}
\begin{proof} Since all polynomials are linear combinations of monomials, it suffices to prove this for monomials.

We first check for the coefficients of $L_M$ and their complex conjugates. The coefficient of $q^{-ds}$ is $(-q^{1/2})^d \tr(M, \wedge^d \std )$ and its complex conjugate is $(-q^{1/2})^d \tr(M, \wedge^d \std^\vee )$ by \eqref{l-wedge} and \eqref{l-bar-wedge}. The highest weight of $\wedge^d \std$ has $d$ ones followed by $N-d$ zeroes, while the highest weight of $\wedge^d \std^\vee$ has $N-d$ zeroes followed by $d$ negative ones, and both of these have norm $d$.

Any monomial is a product of $c_d$s and $\overline{c_d}$s, hence equal to a constant multiple of a product of traces. The product of traces is the trace of the tensor product, which is the sum of the traces on the irreducible summands of the tensor product. By \cref{tensor-product-weights}, these all have weights with norms bounded by the degree of the monomial. \end{proof}

 \begin{lemma}\label{high-weight-orthogonal} Let $\phi \in \mathbb C[c_0,c_1,\dots, \overline{c_0},\overline{c_1},\dots]$ be a polynomial and $V$ an irreducible representation of $U(N)$ such that for all $M\in U(N)$, \[ \phi(L_M) =  \tr (M, V) .\] If $\norm{\weight(V)}>k$ then for all polynomials $\psi \in \mathbb C[c_0,c_1,\dots, \overline{c_0},\overline{c_1},\dots]$ of degree $\leq k$ we have
 \[ \int_{U(N)} \phi(L_M) \overline{\psi(L_M)} \mu_{\textrm{Haar}}=0.\] \end{lemma}
 
 \begin{proof} We apply \cref{low-degree-is-low-weight} to $\psi$ to obtain
 \[ \int_{U(N)} \phi(L_M) \overline{\psi(L_M)} \mu_{\textrm{Haar}}= \int_{U(N)} \tr(M,V) \sum_o \overline{\kappa_o } \overline{\tr (M, V_o) }\mu_{\textrm{Haar}}=\sum_o \overline{\kappa_o }  \int_{U(N)} \tr(M,V) \overline{\tr (M, V_o) }\mu_{\textrm{Haar}} = 0 \] by orthogonality of characters, since $V$ cannot be among the $V_o$ as $\norm{\weight(V)}> k \geq \norm{\weight(V_o)}$ for all $o$. \end{proof}
 
 Note that the conclusion of \cref{high-weight-orthogonal} is the assumption \eqref{hf-hypothesis} of \cref{intro-hf}. To obtain a supply of polynomials to which we can apply \cref{intro-hf}, it suffices to find polynomials $\phi$ satisfying the hypothesis $ \phi(L_M) = \tr (M, V) $ of \cref{high-weight-orthogonal}. We can do this using the Jacobi-Trudi identity for Schur polynomials.

 We always take $\wedge^d \std = \wedge^d \std^\vee=0$ if $ d\notin[0,N]$. For a power series $L$ in $q^{-s}$ and arbitrary integer $d$, let $c_d(L)$ be the coefficient of $q^{-ds}$ in $L$, so that $c_d=0$ for $d<0$.

\begin{lemma}\label{rep-to-poly} Let $V$ be a representation of $U(N)$ with highest weight $\omega_1,\dots,\omega_N$. Let  $a,b \in \mathbb Z^{\geq 0}$ be integers satisfying $a \geq \omega_1$ and $\omega_n \geq -b$. Then

\begin{enumerate}

\item $\tr (M,V \otimes \det^{b})$ is the determinant of the $a+b \times a+b$ matrix whose $ij$th entry is $\tr(M, \wedge^{ \# \{ \ell \mid \omega_\ell \geq i-b \} +j-i } \std) $.

\item Let $d_{ij} = \# \{ \ell \mid \omega_\ell < i-b \} + i - j  $ for $ i\leq b$ and $d_{ij} =  \# \{ \ell \mid \omega_\ell \geq i-b \} +j-i $ for $i > b$. Then $\tr (M,V )$ is the determinant of the $a+b \times a+b$ matrix whose $ij$th entry is $\tr(M, \wedge^{ d_{ij} }\std)$ for $i>b$ and $\tr(M, \wedge^{d_{ij} } \std^\vee )$ for $j>b$.

\item $\tr (M,V )$ is the determinant of the $a+b \times a+b$ matrix whose $ij$th entry is $ (-q^{-1/2})^{d_{ij}} c_{d_{ij}} (L_M)$ for $i>b$ and $ (-q^{-1/2})^{d_{ij}} \overline{c_{d_{ij}} (L_M)} $ for $i \leq b$.

\item The determinant of the $a+b \times a+b$ matrix whose $ij$th entry is $ (-q^{-1/2})^{d_{ij}} c_{d_{ij}} $ for $i>b$ and $ (-q^{-1/2})^{d_{ij}} c_{d_{ij}} $ for $i \leq b$ is a polynomial of degree at most the norm of $\omega_1,\dots, \omega_n$. 

\end{enumerate} \end{lemma}

\begin{proof} For part (1), if we let $\eigen_1,\dots, \eigen_N$ be the eigenvalues of $M$, then $\tr (M,V \otimes \det^{b})$ is the Schur polynomial in $\eigen_1,\dots, \eigen_N$ associated to the partition $(\omega_1+b,\dots, \omega_N+b)$ and \[\tr(M, \wedge^{ \# \{ \ell \mid \omega_\ell \geq i-b \} +j-i }  \std) \] is the $\# \{ \ell \mid \omega_\ell \geq i-b \} +j-i$th elementary symmetric polynomial in $\eigen_1,\dots, \eigen_N$. The claim is then a statement of the Jacobi-Trudi identity for Schur polynomials~\cite[Formula A6]{FH}.

For part (2), we have $\tr(M,V) = \tr (M,V, \otimes \det^{b}) \det(M)^{-b}$. We take the formula of part (1) and multiply the first $b$ rows by $\det(M)^{-1}$ to multiply the determinant by $\det(M)^{-b}$. This fixes the $ij$ entry for $i>b$ and changes the $ij$ entry for $ i \leq b$ to
\[ \tr(M, \wedge^{ \# \{ \ell \mid \omega_\ell \geq i-b \} +j-i } \std)  \det(M)^{-1} =\tr(M, \wedge^{ \# \{ \ell \mid \omega_\ell \geq i-b \} +j-i } \std \otimes \det^{-1}) \] \[=\tr(M, \wedge^{ N- (\# \{ \ell \mid \omega_\ell \geq i-b \} +j-i )}\std^\vee ) = \tr(M, \wedge^{ \# \{ \ell \mid \omega_\ell < i-b \} +i-j)}\std^\vee )  = \tr(M, \wedge^{d_{ij} } \std^\vee ).\]\

Part (3) follows from part (2) when we observe that $\tr(M, \wedge^{ d_{ij} }\std)= (-q^{-1/2})^{d_{ij}} c_{d_{ij}} (L_M)$ because of \eqref{l-wedge} and $\tr(M, \wedge^{ d_{ij} }\std^\vee )= (-q^{-1/2})^{d_{ij}} \overline{ c_{d_{ij}} (L_M)}$ because of \eqref{l-bar-wedge}.

For part (4), let \[ d(i) = \begin{cases}  \# \{ \ell \mid \omega_\ell <  i-b \} & \textrm{if } i \leq b \\  \# \{ \ell \mid \omega_\ell \geq i-b \} & \textrm{if } i>b \end{cases} \] and let \[ v(i) = \begin{cases}  2b -i & \textrm{if } i \leq b \\ i  & \textrm{if } i>b \end{cases}.\] Then if $i,j> b$ we have $v(j)-v(i)=j-i$, if $j\leq b < i$ we have $ v(j)-v(i) = 2b-j -i  \geq j-i $, if $i,j \leq b$ we have $v(j)-v(i) = i-j $, and if $i \leq b < j$ we have $v(j)-v(i)=  j + i-2b  \geq  i-j$ so in all cases we have \[ d(i) + v(j)-v(i) \geq d_{ij} .\] The $ij$-entry is a polynomial of degree $d_{ij}$. This implies the determinant has degree $\leq \sum_{i=1}^{a+b} d(i)$ since when we calculate the degree of each term in the Leibniz expansion, the $v(j)$ and $v(i)$ cancel. Finally $\sum_{i=1}^{a+b} d(i)$ is the norm of $\omega_1,\dots, \omega_N$ since if $\omega_\ell \geq 0$ then $\omega_\ell$ contributes to $d(b+1),\dots, d(b+ \omega_\ell)$ and thus contributes $\omega_\ell$ to the sum while if $\omega_\ell \leq 0$ then $\omega_\ell$ contributes to $d(b+1+\omega_\ell),\dots, d(b)$ and thus contributes $-\omega_\ell$ to the sum. \end{proof}

\subsection{Handling the error term}\label{ss-error}

 \begin{lemma} Fix $\rp, \ry$, and $N$. There exist coefficients $\kappa_o$ (depending on $s_1,\dots, s_{\rp+\ry}$) and irreducible representations $V_o$ such that for each $M\in U(N)$ we have
 \begin{equation}\label{g-moment-rep-sum} \prod_{i=1}^{\rp} L_M(s_i ) \prod_{i=\rp+1}^{\rp + \ry} \overline{L _M( s_i ) }=  \sum_o \kappa_o \tr(M, V_o).\end{equation}
\end{lemma} \begin{proof} By \eqref{l-wedge} and \eqref{l-bar-wedge}, $L_M(s)$ and $\overline{L_M(s)}$ may be expressed as complex-linear combinations of characters of $U(N)$. Multiplying these expressions, it follows that 
$ \prod_{i=1}^{\rp} L_M(s_i ) \prod_{i=\rp+1}^{\rp + \ry} \overline{L _M( s_i ) }$ is a complex-linear combinations of characters of $U(N)$. \end{proof}

Using \cref{rep-to-poly}(3), we associate to each $V_o$ a polynomial $\psi_o\in \mathbb C[c_0, c_1,\dots, \overline{c_0},\overline{c_1},\dots]$ such that $\psi_o (L_M) = \tr(M, V_o)$. This gives 
\begin{equation}\label{L-psi}  \prod_{i=1}^{\rp} L(s_i ) \prod_{i=\rp+1}^{\rp + \ry} \overline{L ( s_i ) }=  \sum_o \kappa_o \psi_o(L)\end{equation} as long as $L= L_M$ for some $M\in U(N)$.

We define \[ \phi_{\textrm{lf}} = \sum_{ o ,  \norm{ \weight (V_o)}\leq k} \kappa_o \psi_o \] and \[ \phi_{\textrm{hf}} = \sum_{ o ,  \norm{ \weight (V_o)}> k} \kappa_o\psi_o. \]  

Note that all implicit constants in big $O$ notation used in this section will be allowed to depend on $\rp,\ry,q$ but not on $N$ (since the final goal is to prove \cref{intro-cfkrs}, an estimate whose error term depends on $\rp,\ry,q$ but not on $N$).

\begin{lemma}\label{phi-norm-bound} Both $\norm{\phi_{\textrm{hf}} }_2$ and $\norm{\phi_{\textrm{lf}} }_2$ are $O( N ^{ \frac{(\rp+\ry)^2 }{2}})$.\end{lemma}
\begin{proof} We have \[  \norm{\phi_{\textrm{hf}} + \phi_{\textrm{lf}}}^2 = \int_{U(N)} \abs { \phi_{\textrm{hf}}(L_M)  +  \phi_{\textrm{lf}}(L_M) }^2 \mu_{\textrm{Haar}} =  \int_{U(N)} \abs{\prod_{i=1}^{\rp} L_M(s_i ) \prod_{i=\rp+1}^{\rp + \ry} \overline{L _M( s_i ) } }^2  \mu_{\textrm{Haar}} \] \[ = \prod_{i=1}^{\rp+\ry} \left( \int_{U(N)} \abs{L_M(s_i)}^{2(\rp+\ry)} \mu_{\textrm{Haar}}\right)^{1/(\rp+\ry)} =    \int_{U(N)} \abs{L_M(1/2)}^{2(\rp+\ry)} \mu_{\textrm{Haar}} = O( N ^{ (\rp+\ry)^2 }) \]
by H\"{o}lder's inequality, the invariance under translation by diagonal matrices of Haar measure on $U(N)$, and the classical calculation of the moments of the characteristic polynomial of random unitary matrices. By \cref{high-weight-orthogonal} and \cref{rep-to-poly}(4) we have  \[ \int_{U(N)} \phi_{\textrm{hf}}(L_M) \overline{ \phi_{\textrm{lf}}(L_M)} \mu_{\textrm{Haar}}=0,\] i.e. $\phi_{\textrm{lf}}$ and $\phi_{\textrm{hf}}$ are orthogonal, so 
$  \norm{\phi_{\textrm{hf}} }_2^2+ \norm{\phi_{\textrm{lf}} }_2^2 = \norm{\phi_{\textrm{hf}} + \phi_{\textrm{lf}}}^2 $ and thus the indvidual norms are bounded as well.\end{proof}

\begin{lemma}\label{lem-mip} Assume that $q>11$. Let $\rp$ and $\ry$ be nonnegative integers and $s_1,\dots, s_{\rp + \ry}$ be complex numbers with real part $\frac{1}{2}$. We have \begin{equation}\label{moment-to-int-phi}\int_{\Cqs} \Bigl( \prod_{i=1}^{\rp} L(s_i ) \prod_{i=\rp+1}^{\rp + \ry} \overline{L ( s_i ) } \Bigr) \mu_{\textrm{ch}} = \int_{\mathbb C[[q^{-s} ]]}  \phi_{\textrm{lf}} \mu_{\textrm{ep}} + O ( N^{ \frac{(\rp+\ry)^2}{2}  -\beta \frac{q-2}{4}}).\end{equation} \end{lemma}

\begin{proof} From \eqref{L-psi} and the definitions of $\phi_{\textrm{lf}}$ and $\phi_{\textrm{hf}}$ we have \begin{equation}\label{L-phi-split}\begin{split} 
& \int_{\Cqs} \Bigl( \prod_{i=1}^{\rp} L(s_i ) \prod_{i=\rp+1}^{\rp + \ry} \overline{L ( s_i ) } \Bigr) \mu_{\textrm{ch}} = \int_{\Cqs}\sum_o \kappa_o \psi_o \mu_{\textrm{ch}}\\ = &\int_{\mathbb C[[q^{-s} ]]}  (\phi_{\textrm{lf}} +\phi_{\textrm{hf}} ) \mu_{\textrm{ch}} =\int_{\mathbb C[[q^{-s} ]]}  \phi_{\textrm{lf}} \mu_{\textrm{ch}} + \int_{\mathbb C[[q^{-s} ]]} \phi_{\textrm{hf}}  \mu_{\textrm{ch}}.\end{split}\end{equation} 

To $\int_{\mathbb C[[q^{-s} ]]}  \phi_{\textrm{lf}} \mu_{\textrm{ch}}$ we apply \cref{intro-lf}, using \cref{rep-to-poly}(4) to check the hypothesis, and to $\int_{\mathbb C[[q^{-s} ]]} \phi_{\textrm{hf}}  \mu_{\textrm{ch}}$ we apply \cref{intro-hf}, using \cref{high-weight-orthogonal} to check the hypothesis \eqref{hf-hypothesis}. From these results and \eqref{L-phi-split} we obtain
\begin{equation}\label{L-phi-error} \int_{\Cqs} \Bigl( \prod_{i=1}^{\rp} L(s_i ) \prod_{i=\rp+1}^{\rp + \ry} \overline{L ( s_i ) } \Bigr) \mu_{\textrm{ch}} = \int_{\mathbb C[[q^{-s} ]]}  \phi_{\textrm{lf}} \mu_{\textrm{ep}} + O (e^{ (\frac{1}{2} - o_N(1)) N^{1-\beta} \log (N^{1-\beta})} \norm{\phi_{\textrm{lf}} }_2 )+ O ( N^{-  \beta \frac{q-2}{4}} \norm{\phi_{\textrm{hf}} }_2   ) . \end{equation}

  Since $e^{ (\frac{1}{2} - o_N(1)) N^{1-\beta} \log (N^{1-\beta})}$ is bounded by $N^{-  \beta \frac{q-2}{4}} $, \cref{phi-norm-bound} together with \eqref{L-phi-error} gives \eqref{moment-to-int-phi}. \end{proof}

\subsection{Comparing the main terms}\label{ss-main}

To prove \cref{intro-cfkrs}, it remains to compare $\int_{\mathbb C[[q^{-s} ]]}  \phi_{\textrm{lf}} \mu_{\textrm{ep}}$ to $\operatorname{MT}^{\rp,\ry}_N(s_1,\dots, s_{\rp+\ry})$.

To begin, we will calculate $\phi_{\textrm{lf}}$ more precisely, which requires making the representations $V_o$ and coefficients $\kappa_o$ appearing in \eqref{g-moment-rep-sum} explicit. We also make use of the change of variables $\alpha_i = s_i- \frac{1}{2}$, so in particular $q^{\frac{1}{2}-s_i}= q^{-\alpha_i}$ and $\alpha_i$ is imaginary so $\overline{\alpha_i}=-\alpha_i$.

Our calculation will culminate in the formula \eqref{int-phi-formula-adjusted} which expresses $\int_{\mathbb C[[q^{-s} ]]}  \phi_{\textrm{lf}} \mu_{\textrm{ep}}$ using a sum over polynomials $\psi_{\mathbf e}$ against coefficients $\kappa_{\mathbf e}$ indexed by certain tuples of integers $\mathbf e$. To motivate the definitions of $\kappa_{\mathbf e}$ and $\psi_{\mathbf e}$ the proof will proceed in steps.

After proving \eqref{int-phi-formula-adjusted}, we will equate a certain longer sum to $\operatorname{MT}^{\rp,\ry}_N$. The difference between this longer sum and the original introduces a secondary error term which we will also bound.

To prove \eqref{int-phi-formula-adjusted}, we use the method of Bump and Gamburd~\cite{BumpGamburd}, i.e. we apply the Cauchy identity for Schur functions to express the desired moment as a sum of products of pairs of Schur functions. One Schur function in each pair will beocme $\kappa_{\mathbf e}$ and the other will become $\psi_{\mathbf e}$. This method was originally used to calculate expectations of products of the characteristic polynomial of a unitary matrix against Haar measure, but here we apply it (together with other tools) to calculate the expectation againt a non-uniform measure.

If $\eigen_1(M),\dots, \eigen_N(M)$ are the eigenvalues of $M$ then \[ L_M(s_i) = \det( I - q^{ -\alpha_i } M ) = \prod_{\ell=1}^N  ( 1- q^{-\alpha_i } \eigen_\ell(M)) \] while \[ \overline{L_M(s_i)} = \overline{\det( I - q^{ -\alpha_i} M ) } =(-1)^N q^{ N \alpha_i} (\det M)^{-1} \overline{\det( I - q^{ \alpha_i }  M^{-1} )} \] \[= (-1)^N q^{ N \alpha_i } (\det M)^{-1} \det( I - q^{-\alpha_i } M ) = (-1)^N q^{ N \alpha_i} (\det M)^{-1} \prod_{\ell=1}^N ( 1 -q^{-\alpha_i } \eigen_N(M)) \]

so \[  \prod_{i=1}^{\rp} L_M(s_i ) \prod_{i=\rp+1}^{\rp + \ry} \overline{L _M( s_i ) } = (-1)^{N \ry} (\det M)^{-\ry} \prod_{i=\rp+1}^{\rp+\ry} q^{   N \alpha_i  }  \prod_{i=1}^{\rp+\ry} \prod_{\ell=1}^N  ( 1 -q^{-\alpha_i } \eigen_N(M)) .\]

The Cauchy identity for Schur functions gives \[ \prod_{i=1}^{\rp+\ry} \prod_{\ell=1}^N  ( 1 -q^{-\alpha_i } \eigen_N(M)) = \sum_{\rho} s_{\rho} ( q^{-\alpha_1},\dots, q^{-\alpha_{\rp+\ry}} ) s_{\rho'} ( \eigen_1(M),\dots, \eigen_N(M)) \] where $\rho$ denotes a partition, $s_\rho$ the Schur function associated to that partition, $\rho'$ the dual partition, and $s_{\rho'}$ the corresponding Schur function. The $\rp+\ry$-variable Schur function $s_\rho$ vanishes unless $\rho$ has at most $\rp+\ry$ parts and the $N$-variable Schur function $s_{\rho'}$ vanishes unless all parts of $\rho$ have size at most $\leq N$. Partitions satisfying both of these can be equivalently expressed as tuples $e_1,\dots, e_{\rp+\ry}$ of integers satisfying $N \geq e_1 \geq \dots \geq e_{\rp+\ry} \geq 0$, giving
\[ \prod_{i=1}^{\rp+\ry} \prod_{\ell=1}^N  ( 1 -q^{\alpha_i} \eigen_N(M)) =\sum_{\substack{ e_1,\dots e_{\rp+\ry} \in \mathbb Z \\ N \geq e_1 \geq \dots \geq e_{\rp+\ry}  \geq 0 }} s_{(e_1,\dots, e_{\rp+\ry})} (  q^{-\alpha_1},\dots, q^{-\alpha_{\rp+\ry}} )  s_{ (e_1,\dots, e_{\rp+\ry})'} ( \eigen_1(M),\dots, \eigen_N(M) )\]
and thus
\[  \prod_{i=1}^{\rp} L_M(s_i ) \prod_{i=\rp+1}^{\rp + \ry} \overline{L _M( s_i ) } \] \[= \sum_{\substack{ e_1,\dots e_{\rp+\ry} \in \mathbb Z \\ N \geq e_1 \geq \dots \geq e_{\rp+\ry}  \geq 0 }} \Bigl( (-1)^{N \ry}  \prod_{i=\rp+1}^{\rp+\ry} q^{   N\alpha_i  } s_{(e_1,\dots, e_{\rp+\ry})} (  q^{-\alpha_1},\dots, q^{-\alpha_{\rp+\ry}} ) \Bigr) \Bigl( (\det M)^{-\ry}  s_{ (e_1,\dots, e_{\rp+\ry})'} ( \eigen_1(M),\dots, \eigen_N(M) )\Bigr) .\]
Now the significance of this expression is that $s_{ (e_1,\dots, e_{\rp+\ry})'} ( \eigen_1(M),\dots, \eigen_N(M) )$ is the trace of $M$ acting on the irreducible representation of $U(N)$ with highest weight $(e_1,\dots, e_{\rp+\ry})'$ so that $(\det M)^{-\ry}  s_{ (e_1,\dots, e_{\rp+\ry})'} ( \eigen_1(M),\dots, \eigen_N(M) )$ is the trace of $M$ acting on the irreducible representation of $U(N)$ with highest weight obtained from $(e_1,\dots, e_{\rp+\ry})'$ by subtracting $\ry$ from each entry. We refer to this representation as $V_{\mathbf e}$. These representations $V_{\mathbf e}$ have distinct highest weights, and thus are not isomorphic, for distinct $V_{\mathbf e}$. Thus the expression
\begin{equation}\label{precise-moment-rep-sum} \prod_{i=1}^{\rp} L_M(s_i ) \prod_{i=\rp+1}^{\rp + \ry} \overline{L _M( s_i ) } = \sum_{\substack{ e_1,\dots e_{\rp+\ry} \in \mathbb Z \\ N \geq e_1 \geq \dots \geq e_{\rp+\ry}  \geq 0 }} \Bigl( (-1)^{N \ry}  \prod_{i=\rp+1}^{\rp+\ry} q^{   N\alpha_i  } s_{(e_1,\dots, e_{\rp+\ry})} ( q^{-\alpha_1},\dots, q^{-\alpha_{\rp+\ry}}) \Bigr)\tr (M, V_{\mathbf e}) \end{equation}
is a precise form of \eqref{g-moment-rep-sum}. 

We now let $\psi_{\mathbf e}$ be the determinant of the $\rp +\ry \times  \rp +\ry$ matrix whose $ij$th entry is $(-q^{-1/2})^{  e_i + j-i} c_{ e_i+j-i} $ for $i> \ry$ and $(-q^{-1/2})^{ N-e_i + i - j } \overline{c_{ N-e_i +i -j } } $ for $i \leq \ry$. 

\begin{lemma}\label{psi-key} \begin{enumerate}

\item $\psi_{\mathbf e}$ is the polynomial associated to $V_{\mathbf e}$ by \cref{rep-to-poly}(3) with $a=\rp$ and $b=\ry$.

\item  The norm of the highest weight of $V_{\mathbf e}$ is $\sum_{i=1}^{\ry} (N-e_i)  + \sum_{i=\ry+1}^{\rp+\ry} e_i $. \end{enumerate} \end{lemma}

\begin{proof} $(e_1,\dots, e_{\rp+\ry})'$ is the vector consisting of $e_{\rp+\ry}$ copies of $\rp+\ry$,  $e_{i}-e_{i+1}$ copies of $i$ for all $i$ from $\rp+\ry-1$ to $1$, and $N-e_1$ copies of $0$. Subtracting $\ry$ from each entry gives the vector $\omega_1,\dots, \omega_N$ with $e_{\rp+\ry}$ copies of $\rp$,  $e_{i}-e_{i+1}$ copies of $i-\ry$ for all $i$ from $\rp+\ry-1$ to $1$, and $N-e_1$ copies of $-\ry$.

It follows that $\# \{ \ell \mid \omega_\ell < i-\ry \} $ is $N- e_i $ and $ \# \{ \ell \mid \omega_\ell \geq i-\ry \}$ is $e_i$. With notation as in \cref{rep-to-poly}(2), this implies $d_{ij}= e_i +j-i$ for $i>\ry$ and $d_{ij}= N-e_i+i-j$ for $i\leq \ry$, which plugged into \cref{rep-to-poly}(3) gives $\psi_{\mathbf e}$.

The norm of $\omega_1,\dots,\omega_N$ is   \[ e_{\rp+\ry}r  + \sum_{i=1}^{\rp+\ry-1} (e_i -e_{i+1} ) \abs{  i -\ry} + (N-e_1) \ry =  \sum_{i=1}^{\rp+\ry} e_i  (\abs{i-\ry} - \abs{i-\ry-1})  + N \ry = \sum_{i=1}^{\ry} (N-e_i) + \sum_{i=\ry+1}^{\rp+\ry} e_i \] by a telescoping sum.\end{proof} 

We also let  \[ \kappa_{\mathbf e} =  (-1)^{N \ry}  \frac{ \sum_{ \sigma \in S_{\rp+\ry}} \sgn(\sigma) \prod_{i=1}^{\rp+\ry} q^{ - (e_i +\rp+\ry- i)  \alpha_{\sigma(i) }   }}{ \prod_{1 \leq i_1 < i_2 \leq \rp+\ry} (q^{-\alpha_{i_1} } - q^{ - \alpha_{i_2}} ) } \prod_{i=\rp+1}^{\rp+\ry} q^{   N\alpha_i  }  .\]

That  \begin{equation}\label{kappa-key} \kappa_{\mathbf e} = (-1)^{N \ry}  s_{(e_1,\dots, e_{\rp+\ry})} ( q^{-\alpha_1},\dots, q^{-\alpha_{\rp+\ry}})  \prod_{i=\rp+1}^{\rp+\ry}q^{   N\alpha_i  } \end{equation} follows from the definition of Schur polynomials (or, defining them in terms of irreducible representations, from the Weyl character formula).

\eqref{precise-moment-rep-sum}, \cref{psi-key}, and \eqref{kappa-key} imply that \[ \phi_{\textrm{lf}} =  \sum_{\substack{ e_1,\dots e_{\rp+\ry} \in \mathbb Z \\ N \geq e_1 \geq \dots \geq e_{\rp+\ry}  \geq 0 \\  \sum_{i=1}^{\ry} (N - e_i) + \sum_{i=\ry+1}^{\rp+\ry} e_i \leq k  }}\kappa_{\mathbf e} \psi_{\mathbf e} \]
and therefore
\begin{equation}\label{int-phi-formula} \int_{\Cqs }\phi_{\textrm{lf}}\mu_{\textrm{ep}} =  \sum_{\substack{ e_1,\dots e_{\rp+\ry} \in \mathbb Z \\ N \geq e_1 \geq \dots \geq e_{\rp+\ry}  \geq 0 \\  \sum_{i=1}^{\ry} (N - e_i) + \sum_{i=\ry+1}^{\rp+\ry} e_i \leq k  }}\kappa_{\mathbf e} \int_{\Cqs } \psi_{\mathbf e} \mu_{\textrm{ep}}.\end{equation}

However, we have defined both $\kappa_{\mathbf e}$ and $\phi_{\mathbf e}$ to make sense for an arbitrary tuple of integers $\mathbf e$, not the ones where the Schur polynomials are defined. We will use this flexibility to extend the range of summation, which will allow us to compare this sum to $\operatorname{MT}_{N}^{\rp,\ry} $, which is similarly formed by extending a sum beyond the obviously appropriate range. 

Observe first that \eqref{int-phi-formula} implies that
\begin{equation}\label{int-phi-formula-adjusted}  \int_{\Cqs }\phi_{\textrm{lf}}\mu_{\textrm{ep}} =  \sum_{\substack{ e_1,\dots e_{\rp+\ry} \in \mathbb Z \\ N \geq e_1\geq  \dots e_{\ry} \\ e_{\ry+1} \geq \dots  e_{\rp+\ry} \geq 0 \\  \sum_{i=1}^{\ry} (N - e_i) + \sum_{i=\ry+1}^{\rp+\ry} e_i \leq k  }}\kappa_{\mathbf e} \int_{\Cqs } \psi_{\mathbf e} \mu_{\textrm{ep}} \end{equation}
since the sole condition appearing in \eqref{int-phi-formula} but not \eqref{int-phi-formula-adjusted} is $e_{\ry} \geq e_{\ry+1}$, but this is implied by the other conditions of \eqref{int-phi-formula-adjusted} since  $e_{\ry} - e_{\ry+1} = N - (N-e_{\ry}) - e_{\ry+1} \geq N -k \geq 0$.

Our next two goals will be to check that
\begin{equation}\label{mt-identity-to-check}  \sum_{\substack{ e_1,\dots e_{\rp+\ry} \in \mathbb Z \\ N \geq e_1\geq  \dots e_{\ry} \\ e_{\ry+1} \geq \dots  e_{\rp+\ry} \geq 0   }}\kappa_{\mathbf e} \int_{\Cqs } \psi_{\mathbf e} \mu_{\textrm{ep}} =  \operatorname{MT}_{N}^{\rp,\ry} \end{equation}
and
\begin{equation}\label{mt-inequality-to-check}  \sum_{\substack{ e_1,\dots e_{\rp+\ry} \in \mathbb Z \\ N \geq e_1\geq  \dots e_{\ry} \\ e_{\ry+1} \geq \dots  e_{\rp+\ry} \geq 0 \\  \sum_{i=1}^{\ry} (N - e_i) + \sum_{i=\ry+1}^{\rp+\ry} e_i > k  }}\kappa_{\mathbf e} \int_{\Cqs } \psi_{\mathbf e} \mu_{\textrm{ep}} =  O (k^{ O(1)} q^{ - \frac{k}{4}} ).\end{equation}

For both of these, we begin by evaluating $ \int_{\Cqs } \psi_{\mathbf e} \mu_{\textrm{ep}} $ explicitly in terms of polynomials in $\mathbb F_q[u]$.  This evaluation in Lemma \ref{average-is-coefficient} will let us prove a bound for $ \int_{\Cqs } \psi_{\mathbf e} \mu_{\textrm{ep}} $ in Lemma \ref{psi-int-bound}. This bound will be used to prove \eqref{mt-inequality-to-check} and then both the evaluation and the bound will be used to prove \eqref{mt-identity-to-check}. 

\begin{lemma}\label{Schur-to-L} For $L= \sum_{d=0}^{\infty} c_d q^{-d s}$, we have an identity of formal Laurent series in $q^{-\alpha_1},\dots ,q^{-\alpha_{\rp+\ry}}$
\begin{equation}\label{Schur-to-L-left} \sum_{e_1,\dots, e_{\rp+\ry} \in \mathbb Z} \psi_{\mathbf e} \prod_{i=1}^{\ry} (-q^{\alpha_i})^{-e_i} \prod_{i=\ry+1}^{\rp+\ry} (-q^{-\alpha_i})^{e_i} \end{equation}  \begin{equation}\label{Schur-to-L-right} = \prod_{1 \leq i_1< i_2 \leq \rp+\ry} (q^{\alpha_{i_1}} - q^{\alpha_{i_2}})  \prod_{i=1}^{\ry} \Bigl( (-q^{\alpha_i})^{1-i-N} \overline{L (q^{- \frac{1}{2} - \alpha_i})}\Bigr) \prod_{i=\ry+1}^{\rp+\ry}  ((-q^{-\alpha_i})^{i-1}L ( q^{- \frac{1}{2} - \alpha_i} )  \Bigr). \end{equation} \end{lemma}

\begin{proof} Consider the $\rp +\ry \times \rp+\ry $ matrix whose $ij$th entry is 
\[  \sum_{e_i \in \mathbb Z} (-q^{-1/2})^{  e_i + j-i} c_{ e_i+j-i}(L)  (- q^{-\alpha_i})^{e_i} =  (-q^{-\alpha_i})^{i-j}  L ( q^{- \frac{1}{2} - \alpha_i} )   \] for $i> \ry$ and 
\[ \sum_{e_i \in \mathbb Z} (-q^{-1/2})^{ N-e_i + i - j } \overline{c_{ N-e_i +i -j } }   (-q^{\alpha_i})^{- e_i} =(-q^{\alpha_i})^{j-i-N} \overline{L (q^{- \frac{1}{2} - \alpha_i})}   \]
for $i \leq \ry$. 

By additivity of determinants in each row, the determinant of this matrix is \eqref{Schur-to-L-left}.

On the other hand, removing a factor of $(-q^{-\alpha_i})^{i-1}L ( q^{- \frac{1}{2} - \alpha_i} ) $ from the $i$th row for $i>\ry$ and $(-q^{\alpha_i})^{1-i-N} \overline{L (q^{- \frac{1}{2} - \alpha_i})} $ from the $i$'th row for $i\leq \ry$, we obtain the matrix whose $ij$-entry is $( -q^{\alpha_i})^{j-1}$ for all $i,j$, which is a Vandermonde matrix and thus has determinant 
\[  \prod_{1 \leq i_1< i_2 \leq \rp+\ry}  ( - q^{\alpha_{i_2}} -  (- q^{\alpha_{i-1}})) = \prod_{1 \leq i_1< i_2 \leq \rp+\ry} (q^{\alpha_{i_1}} - q^{\alpha_{i_2}}) \]
so by the compatibility of determinants with scalar multiplication of rows, the determinant is also \eqref{Schur-to-L-right}. \end{proof}

\begin{lemma} We have an identity of formal Laurent series in $q^{-\alpha_1},\dots ,q^{-\alpha_{\rp+\ry}}$
\begin{equation}\label{L-to-diag}\int_{\Cqs}  \prod_{i=1}^{\ry} \overline{L (q^{- \frac{1}{2} - \alpha_i}) } \prod_{i=\ry+1}^{\rp+\ry} L ( q^{- \frac{1}{2} - \alpha_i} ) \mu_{\textrm{ep}}  = \sum_{\substack{f_1,\dots, f_{\rp+\ry} \in \mathbb F_q[u]^+ \\ \prod_{i=1}^{\ry} f_i =\prod_{i=\ry+1}^{\rp+\ry} f_i }} \prod_{i=1}^{\ry} \abs{f_i}^{-\frac{1}{2}+\alpha_i}  \prod_{i=\ry+1}^{\rp+\ry} \abs{f_i}^{-\frac{1}{2}-\alpha_i} \end{equation} where the integration is applied separately to each term in the formal Laurent series. \end{lemma} 

\begin{proof} Expanding out we have 
\begin{equation}\label{L-prod-expan} \prod_{i=1}^{\ry} \overline{L_\xi (q^{- \frac{1}{2} + \alpha_i}) } \prod_{i=\ry+1}^{\rp+\ry} L_\xi ( q^{- \frac{1}{2} - \alpha_i} ) = \sum_{\substack{f_1,\dots, f_{\rp+\ry} \in \mathbb F_q[u]^+ }} \prod_{i=1}^{\ry}(\overline{\xi(f_i)} \abs{f_i}^{-\frac{1}{2}+\alpha_i}  )\prod_{i=\ry+1}^{\rp+\ry}( \xi(f_i) \abs{f_i}^{-\frac{1}{2}-\alpha_i}) \end{equation}
and orthogonality of characters on a product of circle groups gives
\begin{equation}\label{lambda-average}\mathbb E[  \prod_{i=1}^{\ry}\overline{\xi(f_i)}  \prod_{i=\ry+1}^{\rp+\ry} \xi(f_i)] = \begin{cases} 1 & \textrm{if } \prod_{i=1}^{\ry} f_i =\prod_{i=\ry+1}^{\rp+\ry} f_i  \\ 0 &\textrm{otherwise}\end{cases}.\end{equation}
Taking the expectation of \eqref{L-prod-expan} over $\xi$ and then plugging in \eqref{lambda-average} gives \eqref{L-to-diag}. \end{proof}

\begin{lemma}\label{average-is-coefficient} For each $e_1,\dots, e_{\rp+\ry} \in \mathbb Z$, $ \int_{\Cqs } \psi_{\mathbf e} \mu_{\textrm{ep}}$ is $(-1)^{\sum_{i=1}^{\rp+\ry} e_i + \binom{\rp+\ry}{2} +N\ry}$ times the coefficient of $q^{ \sum_{i=1}^{\ry}(i-1+N-e_i) \alpha_i - \sum_{i=\ry+1}^{\rp+\ry} (d_i-i+1)\alpha_i}$ in \begin{equation}\label{MS}  \prod_{1 \leq i_1< i_2 \leq \rp+\ry} (q^{\alpha_{i_1}} - q^{\alpha_{i_2}})   \sum_{\substack{f_1,\dots, f_{\rp+\ry} \in \mathbb F_q[u]^+ \\ \prod_{i=1}^{\ry} f_i =\prod_{i=\ry+1}^{\rp+\ry} f_i }} \prod_{i=1}^{\ry} \abs{f_i}^{-\frac{1}{2}+\alpha_i}  \prod_{i=\ry+1}^{\rp+\ry} \abs{f_i}^{-\frac{1}{2}-\alpha_i}. \end{equation}  \end{lemma}

\begin{proof} Integrating Lemma \ref{Schur-to-L} against $\mu_{\textrm{ep}}$ and then plugging in \eqref{L-to-diag} gives
\begin{equation*}\sum_{e_1,\dots, e_{\rp+\ry} \in \mathbb Z} \int_{\Cqs } \psi_{\mathbf e} \mu_{\textrm{ep}} \prod_{i=1}^{\ry} (-q^{\alpha_i})^{-e_i} \prod_{i=\ry+1}^{\rp+\ry} (-q^{-\alpha_i})^{e_i} \end{equation*}  \begin{equation*} = \prod_{1 \leq i_1< i_2 \leq \rp+\ry} (q^{\alpha_{i_1}} - q^{\alpha_{i_2}})  \prod_{i=1}^{\ry} (-q^{\alpha_i})^{1-i-N}  \prod_{i=\ry+1}^{\rp+\ry}  (-q^{-\alpha_i})^{i-1}  \sum_{\substack{f_1,\dots, f_{\rp+\ry} \in \mathbb F_q[u]^+ \\ \prod_{i=1}^{\ry} f_i =\prod_{i=\ry+1}^{\rp+\ry} f_i }} \prod_{i=1}^{\ry} \abs{f_i}^{-\frac{1}{2}+\alpha_i}  \prod_{i=\ry+1}^{\rp+\ry} \abs{f_i}^{-\frac{1}{2}-\alpha_i} . \end{equation*} 

Moving some factors to the left-hand side, we obtain
\begin{equation*} \sum_{e_1,\dots, e_{\rp+\ry} \in \mathbb Z} \int_{\Cqs } \psi_{\mathbf e} \mu_{\textrm{ep}} \prod_{i=1}^{\ry} (-q^{\alpha_i})^{i-1+N-e_i} \prod_{i=\ry+1}^{\rp+\ry} (-q^{-\alpha_i})^{e_i-i+1 }  \end{equation*} \begin{equation*} = \prod_{1 \leq i_1< i_2 \leq \rp+\ry} (q^{\alpha_{i_1}} - q^{\alpha_{i_2}})   \sum_{\substack{f_1,\dots, f_{\rp+\ry} \in \mathbb F_q[u]^+ \\ \prod_{i=1}^{\ry} f_i =\prod_{i=\ry+1}^{\rp+\ry} f_i }} \prod_{i=1}^{\ry} \abs{f_i}^{-\frac{1}{2}+\alpha_i}  \prod_{i=\ry+1}^{\rp+\ry} \abs{f_i}^{-\frac{1}{2}-\alpha_i} . \end{equation*} 
Extracting the coefficient of a single term and grouping together all the powers of $(-1)$, we obtain the statement. \end{proof}

\begin{lemma}\label{psi-int-bound} We have   \[ \int_{\Cqs } \psi_{\mathbf e} \mu_{\textrm{ep}} = O \Bigl(  \Bigl(O(1) + \sum_{i=1}^{\ry} (N-e_i) + \sum_{i=\ry+1}^{\rp+\ry} e_i\Bigr)^{O(1)}  q^{ \frac{- \max ( \sum_{i=1}^{\ry} (N-e_i), \sum_{i=\ry+1}^{\rp+\ry} e_i }{2}})\Bigr).\] Furthermore, $\int_{\Cqs } \psi_{\mathbf e} \mu_{\textrm{ep}} =0$ unless $ \sum_{i=1}^{\ry}(i-1+N-e_i)  - \sum_{i=\ry+1}^{\rp+\ry} (e_i-i+1) = \binom{\rp+\ry}{2} $.   \end{lemma}

\begin{proof} We will apply bounds in \cite[Lemma 3.5]{s-rep} for the coefficients of the series $M_S$, defined in \cite[\S3]{s-rep} as
 \[M_S (\alpha_1,\dots,\alpha_{\rp+\ry}) =  \prod_{1 \leq i_1< i_2 \leq \rp+\ry} (q^{\alpha_{i_1}} - q^{\alpha_{i_2}}) \sum_{ \substack{ f_1,\dots, f_{\rp+\ry} \in \mathbb F_q[u]^+ \\ \prod_{i \in S} f_i  /  \prod_{i\notin S} f_i \in u^{\mathbb Z} } }  \prod_{i \in S} \abs{f_i} ^{ -1/2 + \alpha_i}  \prod_{i\notin S} \abs{f_i} ^{-1/2 - \alpha_i   }. \] 
 Taking $S = \{1,\dots,\ry\}$, this definition specializes to
 \begin{equation}\label{MS-u} M_{\{1,\dots,\ry\} }(\alpha_1,\dots,\alpha_{\rp+\ry}) =  \prod_{1 \leq i_1< i_2 \leq \rp+\ry} (q^{\alpha_{i_1}} - q^{\alpha_{i_2}}) \sum_{ \substack{ f_1,\dots, f_{\rp+\ry} \in \mathbb F_q[u]^+ \\ \prod_{i =1}^{\ry}  f_i  /  \prod_{i=\ry+1}^{\rp+\ry}  f_i \in u^{\mathbb Z} } }  \prod_{i=1}^{\ry} \abs{f_i}^{-\frac{1}{2}+\alpha_i}  \prod_{i=\ry+1}^{\rp+\ry} \abs{f_i}^{-\frac{1}{2}-\alpha_i} . \end{equation}
\eqref{MS} and \eqref{MS-u} agree except that the condition $\prod_{i =1}^{\ry}  f_i  /  \prod_{i=\ry+1}^{\rp+\ry}  f_i \in u^{\mathbb Z}$ in \eqref{MS-u} is laxer than the condition $\prod_{i =1}^{\ry}  f_i  =  \prod_{i=\ry+1}^{\rp+\ry}  f_i $ in \eqref{MS}.

If $f_1,\dots, f_{\rp+\ry}$ satisfy $\prod_{i =1}^{\ry}  f_i  /  \prod_{i=\ry+1}^{\rp+\ry}  f_i = u^{n}$ for $n\in \mathbb Z$, then $\sum_{i=1}^{\ry} \deg f_i = \sum_{i=\ry+1}^{\rp+\ry }\deg f_i +n$, so $ \prod_{i=1}^{\ry} \abs{f_i}^{-\frac{1}{2}+\alpha_i}  \prod_{i=\ry+1}^{\rp+\ry} \abs{f_i}^{-\frac{1}{2}-\alpha_i}$ is a monomial in $q^{\alpha_1},\dots, q^{\alpha_{\rp+\ry}}$ of total degree $n$. Since the Vandermonde $ \prod_{1 \leq i_1< i_2 \leq \rp+\ry} (q^{\alpha_{i_1}} - q^{\alpha_{i_2}})$ has total degree $\binom{\rp+\ry}{2}$ in $q^{\alpha_1},\dots, q^{\alpha_{\rp+\ry}}$, this implies that $(f_1,\dots, f_{\rp+\ry})$ contributes to terms in the formal Laurent series \eqref{MS-u} with total degree $n + \binom{\rp+\ry}{2}$. Thus, restricting to $f_1,\dots, f_{\rp+\ry}$ satisfying $\prod_{i =1}^{\ry}  f_i  =  \prod_{i=\ry+1}^{\rp+\ry}  f_i $, i.e., restricting to the case $n=0$, is equivalent to restricting the series \eqref{MS-u} to terms of total degree $\binom{\rp+\ry}{2}$.

In particular, since $\int_{\Cqs } \psi_{\mathbf e} \mu_{\textrm{ep}} $ is by \cref{average-is-coefficient} $\pm$ the coefficient of $q^{ \sum_{i=1}^{\ry}(i-1+N-e_i)} \alpha_i - \sum_{i=\ry+1}^{\rp+\ry} (d_i- i+ 1)\alpha_i$ in \eqref{MS}, it follows that $\int_{\Cqs } \psi_{\mathbf e} \mu_{\textrm{ep}} $ is either $\pm$ the coefficient of $q^{ \sum_{i=1}^{\ry}(i-1+N-e_i)} \alpha_i - \sum_{i=\ry+1}^{\rp+\ry} (e_i-i+1 )\alpha_i$ in \eqref{MS-u} or equal to zero.

 Hence any upper bound on the coefficients of \eqref{MS-u} also gives an upper bound on  $\int_{\Cqs } \psi_{\mathbf e} \mu_{\textrm{ep}} $, which we will shortly use to establish the first part of the statement.  Furthermore, the case where $\int_{\Cqs } \psi_{\mathbf e} \mu_{\textrm{ep}} =0$ occurs when the total degree $ \sum_{i=1}^{\ry}(i-1+N-e_i)  - \sum_{i=\ry+1}^{\rp+\ry} (e_i-i+1) $ of  $q^{ \sum_{i=1}^{\ry}(i-1+N-e_i)} \alpha_i - \sum_{i=\ry+1}^{\rp+\ry} (e_i-i+1 )\alpha_i$ is not equal to $\binom{\rp+\ry}{2}$, giving the second part of the statement.

We apply the upper bound \cite[Lemma 3.5]{s-rep}, which is stated as an upper bound on the coefficient of $\prod_{i=1}^{\rp+\ry} q^{\alpha_i d_i}$, so we must substitute in $( N-e_1  ,\dots, N- e_{\ry} + \ry-1 , - e_{\ry+1}+\ry, \dots, -e_{\rp+\ry} +\rp + \ry-1)$ for $(d_1,\dots,d_{\rp+\ry})$ in the bounds of \cite[Lemma 3.5]{s-rep}. (Also, $d_1,\dots, d_{\rp+\ry}$ are required to satisfy the inequalities of \cite[Lemma 3.2]{s-rep}, but \cite[Lemma 3.2]{s-rep} guarantees the coefficient vanishes if the inequality is not satisfied, so the bound holds also in that case.)  The bound of \cite[Lemma 3.5]{s-rep}, is a product of two factors, the first of which is \[  O (  ( O(1) + \sum_{i \in S} d_i - \sum_{i\notin S} d_i)^{O(1)} )   = O (  ( O(1) + \sum_{i =1}^{\ry}  d_i - \sum_{i=\ry+1}^{\rp+\ry} d_i)^{O(1)} )\] and substituting gives \[O ((O(1) + \sum_{i=1}^{\ry} (N-e_i) + \sum_{i=\ry+1}^{\rp+\ry} e_i)^{O(1)}) \] since the $i-1$ terms may be absorbed into the $O(1)$. The second factor is the minimum of four different bounds, of which we will only need the middle two, which are
\[ \min \Bigl( q ^{ \frac{ -\sum_{i\in S} d_i +\binom{ \abs{S}}{2}}{2}}, q^{  \frac{ \sum_{i\notin S }d_i   + \binom{\abs{S}}{2} - \binom{ \rp+\ry}{2}}{2}} \Bigr) =  \min \Bigl( q ^{- \frac{ -\sum_{i=1}^{\ry}  d_i + \binom{\ry}{2}}{2}}, q^{  \frac{ \sum_{i=\ry+1}^{\rp+\ry} d_i   + \binom{\ry }{2} - \binom{ \rp+\ry}{2}}{2}} \Bigr)\]
and substituting gives
\[ \min \Bigl( q ^{ \frac{ -\sum_{i=1}^{\ry}  (N-e_i )}{2} }, q^{  \frac{ - \sum_{i=\ry+1}^{\rp+\ry} e_i }{2}} \Bigr) =q^{ - \frac{\max ( \sum_{i=1}^{\ry} (N-e_i), \sum_{i=\ry+1}^{\rp+\ry} e_i )}{2}}  \] since the $- \sum_{i=1}^{\ry} (i-1)$ term cancels $\binom{\ry}{2}$ in the exponent of the first $q$ and $\sum_{i=\ry+1}^{\rp+\ry}  (i-1) $ cancels  $ \binom{\ry }{2} - \binom{ \rp+\ry}{2}$ in the exponent of the second $q$. \end{proof}

\begin{lemma}\label{tail-bound}  We have \[    \sum_{\substack{ e_1,\dots e_{\rp+\ry} \in \mathbb Z \\ N \geq e_1\geq  \dots e_{\ry} \\ e_{\ry+1} \geq \dots  e_{\rp+\ry} \geq 0 \\  \sum_{i=1}^{\ry} (N - e_i) + \sum_{i=\ry+1}^{\rp+\ry} e_i > k  }}\kappa_{\mathbf e} \int_{\Cqs } \psi_{\mathbf e} \mu_{\textrm{ep}} =  O ( N^{ O(1)} q^{ - \frac{k}{4}} ).\] \end{lemma}

\begin{proof}We have
\[\abs{\kappa_{\mathbf e} } =  \abs { \frac{ \sum_{ \sigma \in S_{\rp+\ry}} \sgn(\sigma) \prod_{i=1}^{\rp+\ry} q^{ - (e_i + \rp +\ry - i)  \alpha_{\sigma(i) }   }}{ \prod_{1 \leq i_1 < i_2 \leq \rp+\ry} (q^{-\alpha_{i_1} } - q^{ - \alpha_{i_2}} ) }} \leq  \frac{  \prod_{ 1\leq i_1 < i_2 \leq \rp +\ry } \abs{ e_{i_1} -i_1 - e_{i_2} + i_2} } { \prod_{ 1\leq i_1 < i_2 \leq \rp +\ry } \abs{i_1- i_2}} \] since the value of the Weyl character formula for the trace of the unitary representation at a unitary matrix with eigenvalues $q^{ -\alpha_1},\dots ,q^{-\alpha_{\rp+\ry}}$ is bounded by the dimension of that representation which is given by the Weyl dimension formula. (Even if $e_1 \geq e_2 \geq \dots e_{\rp+\ry} $ is not satisfied, the Weyl character formula still gives a formula for the either plus or minus the trace of some irreducible representation, or the zero representation, and the Weyl dimension formula gives plus or minus the dimension of that representation so the absolute value of the dimension formula still bounds the absolute value of the character formula.) 

If $N \geq e_1\geq  \dots \geq e_{\ry} $ and $ e_{\ry+1} \geq \dots  e_{\rp+\ry} \geq 0 $ then each factor $\abs{ e_{i_1} -i_1 - e_{i_2} + i_2}$ is certainly bounded by $\sum_{i=1}^{\ry} (N-e_i) + \sum_{i=\ry+1}^{\rp+\ry} e_i + O(N)$ so \[ \abs{\kappa_{\mathbf e} }  \leq   \Bigl(O (N) + \sum_{i=1}^{\ry} (N-e_i) + \sum_{i=\ry+1}^{\rp+\ry} e_i\Bigr)^{O(1) } \] which together with \cref{psi-int-bound} implies 
\[ \kappa_{\mathbf e} \int_{\Cqs } \psi_{\mathbf e} \mu_{\textrm{ep}} = O \Bigl(  \Bigl(O (N) + \sum_{i=1}^{\ry} (N-e_i) + \sum_{i=\ry+1}^{\rp+\ry} e_i\Bigr)^{O(1) }  q^{ \frac{- \max ( \sum_{i=1}^{\ry} (N-e_i), \sum_{i=\ry+1}^{\rp+\ry} e_i )}{2}}\Bigr)\] \[=  O \Bigl(  \Bigl(O (N) + \sum_{i=1}^{\ry} (N-e_i) + \sum_{i=\ry+1}^{\rp+\ry} e_i\Bigr)^{O(1) }  q^{ \frac{- \sum_{i=1}^{\ry} (N-e_i)-\sum_{i=\ry+1}^{\rp+\ry} e_i }{4}}\Bigr) .\]

The number of tuples $e_1,\dots, e_{\rp+\ry}$ satisfying $N \geq e_1\geq  \dots e_{\ry} $ and $ e_{\ry+1} \geq \dots  e_{\rp+\ry} \geq 0 $ and $\sum_{i=1}^{\ry} (N-e_i) + \sum_{i=\ry+1}^{\rp+\ry} e_i = d$ is $O ( d^{O(1)} ) $, so in total we have
\[ \sum_{\substack{ e_1,\dots e_{\rp+\ry} \in \mathbb Z \\ N \geq e_1\geq  \dots e_{\ry} \\ e_{\ry+1} \geq \dots  e_{\rp+\ry} \geq 0 \\  \sum_{i=1}^{\ry} (N - e_i) + \sum_{i=\ry+1}^{\rp+\ry} e_i > k  }}\kappa_{\mathbf e} \int_{\Cqs } \psi_{\mathbf e} \mu_{\textrm{ep}}= \sum_{d=k+1}^{\infty}  O ( d^{O(1)} ) O ( (O(N)+ d)^{ O(1)}  q^{ - \frac{d}{4}} )\] \[ \ll N^{O(1)} \sum_{d=k+1}^{\infty} d^{O(1)} q^{- \frac{d}{4}} = O ( N^{O(1)} k^{O(1)} q^{ - \frac{k}{4}} ) = O ( N^{ O(1)} q^{ - \frac{k}{4}} ).\qedhere \] \end{proof}

Recall that our desired main term has the form (after substituting $\frac{1}{2}+\alpha_i$ for $s_i$)
\[ \operatorname{MT}_{N}^{\rp,\ry} (\frac{1}{2}+\alpha_1,\dots, \frac{1}{2}+\alpha_{\rp+\ry} )= \prod_{i=\rp+1}^{\rp+\ry} q^{  \alpha_i N}  \sum_{\substack{  S \subseteq \{1,\dots,\rp+\ry\} \\ \abs{S}=\ry}}  \prod_{i \in S} q^{ -\alpha_i N} \sum_{ \substack{ f_1,\dots, f_{\rp+\ry} \in \mathbb F_q[u]^+ \\ \prod_{i \in S} f_i = \prod_{i\notin S} f_i}}  \prod_{i\in S}\abs{f_i} ^{ -\frac{1}{2}- \alpha_i  } \prod_{i \notin S} \abs{f_i} ^{ -\frac{1}{2}-\alpha_i  } .\]
This is interpreted by continuing each term $\sum_{ \substack{ f_1,\dots, f_{\rp+\ry} \in \mathbb F_q[u]^+ \\ \prod_{i \in S} f_i = \prod_{i\notin S} f_i}}  \prod_{i\in S}\abs{f_i} ^{ -\frac{1}{2}-\alpha_i  } \prod_{i \notin S} \abs{f_i} ^{ -\frac{1}{2}+\alpha_i  }$ meromorphically from its domain of absolute convergence and then summing the meromorphic functions. Thus, in this segment of the proof only, we will allow the $\alpha_i$ to be arbitrary complex numbers instead of imaginary numbers. We first evaluate the summand associated to $S =\{1,\dots, \ry\}$, before using this to evaluate the summand associated to arbitrary $S$, and finally evaluate the full main term.

\begin{lemma}\label{first-mt} We have an equality of holomorphic functions on  \[ \{ \alpha_1,\dots, \alpha_{\rp+\ry}\in \mathbb C \mid \operatorname{Re}(\alpha_i)< \frac{1}{4}  \textrm{ for } i \leq \ry, \operatorname{Re}(\alpha_i)> - \frac{1}{4}\textrm{ for } i > \ry \} \] given by \begin{equation}\label{first-mt-left} \prod_{1 \leq i_1 < i_2 \leq \rp+\ry} (q^{-\alpha_{i_1} } - q^{ - \alpha_{i_2}} )  \prod_{i=1}^{\ry} q^{ - \alpha_iN }  \sum_{ \substack{ f_1,\dots, f_{\rp+\ry} \in \mathbb F_q[u]^+ \\ \prod_{i=1}^{\ry} f_i = \prod_{i=\ry+1}^{\rp+\ry} f_i}}  \prod_{i=1}^{\ry}\abs{f_i} ^{ -\frac{1}{2}+\alpha_i  } \prod_{i=\ry+1}^{\rp+\ry}  \abs{f_i} ^{ -\frac{1}{2}-\alpha_i  } \end{equation}\begin{equation}\label{first-mt-right}=  \sum_{\substack{ e_1,\dots e_{\rp+\ry} \in \mathbb Z \\ N \geq e_1\geq  \dots e_{\ry} \\ e_{\ry+1} \geq \dots  e_{\rp+\ry} \geq 0   }}\sum_{ \sigma \in S_{\ry} \times S_{\rp} } \sgn(\sigma) \prod_{i=1}^{\rp+\ry} q^{ - (e_i +\rp+\ry- i)  \alpha_{\sigma(i) } }   \int_{\Cqs } \psi_{\mathbf e} \mu_{\textrm{ep}}\end{equation} where \eqref{first-mt-left} is interpreted by holomorphic continuation from its domain of absolute convergence and \eqref{first-mt-right} is absolutely convergent, and where $S_{\ry} \times S_{\rp}$ is embedded in $S_{\rp +\ry}$ as the subgroup preserving the partition into $\{1,\dots, \ry\}$ and $\{\ry+1,\dots, \rp+\ry\}$. \end{lemma}

\begin{proof} We first check absolute convergence of \eqref{first-mt-right}.   By \cref{psi-int-bound} we have
\[ \int_{\Cqs } \psi_{\mathbf e} \mu_{\textrm{ep}} = O \Bigl(  \Bigl(O(1) + \sum_{i=1}^{\ry} (N-d_i) + \sum_{i=\ry+1}^{\rp+\ry} d_i\Bigr)^{O(1)}  q^{ \frac{- \sum_{i=1}^{\ry} (N-d_i) - \sum_{i=\ry+1}^{\rp+\ry} d_i }{4 }}\Bigr).\]
For $i \leq \ry $ we have  \[\abs{ q^{ - (e_i +\rp+\ry- i)  \alpha_{\sigma(i) } }} = q^{ - (e_i +\rp+\ry- i)  \operatorname{Re} (\alpha_{\sigma(i)} ) } \] \[ = q^{  (N-e_i) \operatorname{Re} (\alpha_{\sigma(i)} ) } q^{ - (N + \rp+\ry-i) \operatorname{Re} (\alpha_{\sigma(i)} )}  \ll_{N, \alpha} q^{  (N-e_i) \operatorname{Re} (\alpha_{\sigma(i)} ) } \leq  q^{ (N-e_i) \max_{1\leq j \leq \ry} \operatorname{Re} (\alpha_{j} ) }  \] where $\ll_{N,\alpha}$ denotes an implicit constant that may depend on $N$ and $\alpha_1,\dots, \alpha_{\rp+\ry}$ but does not depend on $e_1,\dots, e_{\rp+\ry}$ (which is all that is needed for absolute convergence). Similarly, for $i> \ry$ we have \[\abs{ q^{ - (e_i +\rp+\ry- i)  \alpha_{\sigma(i) } }} = q^{ - (e_i +\rp+\ry- i)  \operatorname{Re} (\alpha_{\sigma(i)} ) } = q^{ -e_i  \operatorname{Re} (\alpha_{\sigma(i)} ) }  q^{ - (\rp+\ry-i) \operatorname{Re} (\alpha_{\sigma(i)} )} \] \[\ll_{\alpha} q^{ -e_i  \operatorname{Re} (\alpha_{\sigma(i)} ) }    \leq q^{ -e_i \min_{\ry+1 \leq j \leq \rp+\ry} \operatorname{Re} (\alpha_{j} )} \]
so
\[ \prod_{i=1}^{\rp+\ry} q^{ - (e_i +\rp+\ry- i)  \alpha_{\sigma(i) } }  \ll_{N,\alpha}   q^{ \sum_{i=1}^{\ry} (N-e_i) \max_{1\leq j \leq \ry} \operatorname{Re} (\alpha_{j} )   - \sum_{i=\ry+1}^{\rp+\ry} e_i \min_{\ry+1 \leq j \leq \rp+\ry} \operatorname{Re} (\alpha_{j}  )   }  \]
and thus
\[  \prod_{i=1}^{\rp+\ry} q^{ - (e_i +\rp+\ry- i)  \alpha_{\sigma(i) } }   \int_{\Cqs } \psi_{\mathbf e} \mu_{\textrm{ep}} \] \[ \ll_{N,\alpha}  q^{ -\sum_{i=1}^{\ry} (N-e_i) \left( \frac{1}{4} - \max_{1\leq j \leq \ry} \operatorname{Re} (\alpha_{j} ) \right)   - \sum_{i=\ry+1}^{\rp+\ry} e_i \left( \frac{1}{4}+ \min_{\ry+1 \leq j \leq \rp+\ry}  \operatorname{Re} (\alpha_{j}  )  \right)   }    O \Bigl(  \Bigl(O(1) + \sum_{i=1}^{\ry} (N-e_i) + \sum_{i=\ry+1}^{\rp+\ry} e_i\Bigr)^{O(1)} \Bigr) \]
When we sum over $e_1,\dots, e_{\rp+\ry}$, the assumptions on $\operatorname{Re}(\alpha_i)$ imply that the exponential term in $q$ is exponentially decreasing and thus dominates the polynomial term and leads to absolute convergence of \eqref{first-mt-right}.

Both \eqref{first-mt-left} and \eqref{first-mt-right} may be expressed as formal Laurent series in $q^{-\alpha_1}, \dots, q^{-\alpha_{\rp+\ry}}$. It now suffices to check that, for each $d_1,\dots, d_{\rp+\ry}\in \mathbb Z$, the coefficient of $q^{ \sum_{i=1}^{\rp+\ry} d_i \alpha_i} $ in  \eqref{first-mt-left} equals the coefficient of $q^{ \sum_{i=1}^{\rp+\ry} d_i \alpha_i} $ in  \eqref{first-mt-right}. Then both sides will be equal on the locus where both are absolutely convergent, hence equal everywhere by analytic continuation. This in particular implies \eqref{first-mt-left} is holomorphic on the same region as \eqref{first-mt-right}.

To check the equality of coefficients, we first observe that $\sum_{ \sigma \in S_{\ry} \times S_{\rp} } \sgn(\sigma) \prod_{i=1}^{\rp+\ry} q^{ - (e_i +\rp+\ry- i)  \alpha_{\sigma(i) } } $ is antisymmetric in the variables $\alpha_1,\dots,\alpha_{\ry}$ in the sense that swapping two variables is equivalent to multiplying the sum by $-1$, and similarly antisymmetric in the variables $\alpha_{\ry+1},\dots, \alpha_{\rp+\ry}$. Antisymmetry is stable under linear combinations, so \eqref{first-mt-right} is antisymmetric in $\alpha_1,\dots,\alpha_{\ry}$ and in $\alpha_{\ry+1},\dots, \alpha_{\rp+\ry}$. Similarly, $\sum_{ \substack{ f_1,\dots, f_{\rp+\ry} \in \mathbb F_q[u]^+ \\ \prod_{i=1}^{\ry} f_i = \prod_{i=\ry+1}^{\rp+\ry} f_i}}  \prod_{i=1}^{\ry}\abs{f_i} ^{ -\frac{1}{2}-\alpha_i  } \prod_{i=\ry+1}^{\rp+\ry}  \abs{f_i} ^{ -\frac{1}{2}+\alpha_i  }$ is symmetric in  $\alpha_1,\dots,\alpha_{\ry}$ and in $\alpha_{\ry+1},\dots, \alpha_{\rp+\ry}$ and the Vandermonde is antisymmetric, so their product \eqref{first-mt-left} is antisymmetric.

It follows that the coefficient of $q^{ \sum_{i=1}^{\rp+\ry} d_i \alpha_i} $ in either \eqref{first-mt-left} or \eqref{first-mt-right} is antisymmetric in $d_1,\dots, d_{\ry}$ and in $d_{\ry+1},\dots, d_{\rp+\ry}$. In particular, the coefficient vanishes if we have $d_i=d_j$ for $i< j \leq \ry$ or $ \ry+1< i<j$, as in that case swapping $d_i$ and $d_j$ both preserves the value of the coefficient and negates it.  Furthermore, to check that the coefficients of $q^{ \sum_{i=1}^{\rp+\ry} d_i \alpha_i} $ are equal for all $d_1,\dots, d_{\rp+\ry}$, it suffices to check for only tuples such that $d_1< \dots < d_{\ry}$ and $d_{\ry+1}< \dots < d_{\rp+\ry}$: If the $d_1,\dots, d_{\ry}$ are distinct, we can swap them until they are in increasing order, multiplying both coefficients by the same power of $-1$, and similarly with the $d_{\ry+1},\dots, d_{\rp+\ry}$, but if they are not distinct, both coefficients vanish and are trivially equal.

So it remains to check for $d_1<\dots < d_{\ry}$ and $d_{\ry+1}< \dots <d_{\rp+\ry}$ that the coefficients of $q^{ \sum_{i=1}^{\rp+\ry} d_i \alpha_i} $ in  \eqref{first-mt-left} and  \eqref{first-mt-right} are equal. First, we observe that the only $\sigma$ that contributes to this coefficient in \eqref{first-mt-right} is the identity, since we have $e_1 \geq \dots \geq e_{\ry}$ and $e_{\ry+1} \geq \dots e_{\rp+\ry}$ so that $-(e_1+\rp+\ry-1 )<-( e_2 + \rp+\ry-2) < \dots <-( e_{\ry }+\rp)$ and  $-(e_{\ry+1} + \rp-1)< \dots <- e_{\rp+\ry} $ and any permutation other than the identity would change the order and thus not take these exponents to $d_1,\dots , d_{\rp+\ry}$.  So the only relevant term is the one with $\sigma = \textrm{id}$ and $e_i =  i - \rp-\ry -d_i$ for all $i$. 

If $N \geq 1-\rp-\ry -d_1$ and $d_{\rp+\ry} \leq 0$  so that $N \geq e_1$ and $e_{\rp+\ry} \geq 0$ then this  term has coefficient $  \int_{\Cqs } \psi_{\mathbf e} \mu_{\textrm{ep}}$ which by \cref{average-is-coefficient} is \[(-1)^{\sum_{i=1}^{\rp+\ry} e_i + \binom{\rp+\ry}{2} +N\ry } =(-1)^{ \sum_{i=1}^{\rp+\ry} d_i+N\ry} \] times the coefficient of $q^{ \sum_{i=1}^{\ry}(\rp+\ry -1+N +d_i ) \alpha_i - \sum_{i=\ry+1}^{\rp+\ry} (- d_i - \rp- \ry  +1)\alpha_i}$ in \eqref{MS}.

Even if $N < 1-\rp- \ry  -d_1$ or $d_{\rp+\ry} > 0$ then the same conclusion holds as then no term in \eqref{first-mt-right} contributes so the coefficient is zero but it is compared to the coefficient in \eqref{MS} of a monomial whose exponent in $q^{\alpha_1}$ is $<- N-\rp- \ry +1 $ or whose exponent of $q^{\alpha_{\rp+\ry}}$ is $> 0 $ and no monomial of this form appears so this is also trivially zero.

On the other hand, \eqref{first-mt-right} is equal to the product of \eqref{MS} with \[  \frac{ \prod_{1 \leq i_1 < i_2 \leq \rp+\ry} (q^{-\alpha_{i_1} } - q^{ - \alpha_{i_2}} ) }{ \prod_{1 \leq i_1 < i_2 \leq \rp+\ry} (q^{\alpha_{i_1} } - q^{  \alpha_{i_2}} )}  \prod_{i=1}^{\ry}  q^{ - \alpha_iN }   =   \prod_{1 \leq i_1 < i_2 \leq \rp+\ry}( - q^{-\alpha_{i_1}} q^{- \alpha_{i_2}}) \prod_{i=1}^{\ry} q^{ - \alpha_iN }  \] so the coefficient of $q^{ \sum_{i=1}^{\rp+\ry} d_i \alpha_i} $ in \eqref{first-mt-right} is $(-1)^{ \binom{\rp+\ry}{2}}$ times the coefficient of  $q^{ \sum_{i=1}^{\rp+\ry} \alpha_i (d_i +\rp+\ry-1)  + \sum_{i=1}^{\ry} \alpha_i N } $ in \eqref{MS}. Since \[ \sum_ {i=1}^{\rp+\ry} \alpha_i (d_i +\rp+\ry-1)  + \sum_{i=1}^{\ry}  \alpha_i N =  \sum_{i=1}^{\ry}(\rp+\ry -1+N +d_i ) \alpha_i - \sum_{i=\ry+1}^{\rp+\ry} (- d_i - \rp- \ry  +1)\alpha_i\] the two coefficients agree up to a factor of $(-1)^{ \sum_{i=1}^{\rp+\ry} d_i + \binom{\rp+\ry}{2} +N \ry } $.

However, by \cref{psi-int-bound}, that coefficient vanishes unless  
\[ \binom{\rp+\ry}{2} = \sum_{i=1}^{\ry}(\rp+\ry -1+N+ d_i )   - \sum_{i=\ry+1}^{\rp+\ry} (-d_i - \rp- \ry  +1)= (\rp+\ry) (\rp+\ry-1) + N \ry  + \sum_{i=1}^{\ry} d_i - \sum_{i=\ry+1}^{\rp+\ry} d_i \] in which case $\sum_{i=1}^{\rp+\ry} d_i + \binom{\rp+\ry}{2} +N \ry$ is even, so either the two coefficients are equal or they are negatives of each other but zero and hence equal anyways. \end{proof} 

\begin{lemma}\label{second-mt} For $S \subseteq\{1,\dots, \rp+\ry\}$ of size $\ry$, we have an equality of holomorphic functions on  \[ \{ \alpha_1,\dots, \alpha_{\rp+\ry}\in \mathbb C \mid \operatorname{Re}(\alpha_i)< \frac{1}{4}  \textrm{ for } i \in S, \operatorname{Re}(\alpha_i)> - \frac{1}{4}\textrm{ for } i \notin S \} \] given by \begin{equation}\label{second-mt-left} \prod_{1 \leq i_1 < i_2 \leq \rp+\ry} (q^{-\alpha_{i_1} } - q^{ - \alpha_{i_2}} )  \prod_{i\in S}  q^{ - \alpha_iN }  \sum_{ \substack{ f_1,\dots, f_{\rp+\ry} \in \mathbb F_q[u]^+ \\ \prod_{i\in S } f_i = \prod_{i\notin S} f_i}}  \prod_{i\in S} \abs{f_i} ^{ -\frac{1}{2}+\alpha_i  } \prod_{i\notin S}   \abs{f_i} ^{ -\frac{1}{2}-\alpha_i  } \end{equation}\begin{equation}\label{second-mt-right}=  \sum_{\substack{ e_1,\dots e_{\rp+\ry} \in \mathbb Z \\ N \geq e_1\geq  \dots e_{\ry} \\ e_{\ry+1} \geq \dots  e_{\rp+\ry} \geq 0   }}\sum_{ \substack{  \sigma \in S_{\rp+\ry} \\ \sigma^{-1} (S) = \{1,\dots, \ry\} } } \sgn(\sigma) \prod_{i=1}^{\rp+\ry} q^{ - (e_i +\rp+\ry- i)  \alpha_{\sigma(i) } }   \int_{\Cqs } \psi_{\mathbf e} \mu_{\textrm{ep}}.\end{equation} \end{lemma}

\begin{proof}  We let $\tau$ be any fixed permutation in $S_{\rp+\ry}$ with $\tau^{-1} (S) =\{1,\dots, \ry\}$ and perform the change of variables replacing $\alpha_i$ with $\alpha_{\tau(i)}$ in the statement of \cref{first-mt}, obtaining

\begin{equation}\label{third-mt-left} \prod_{1 \leq i_1 < i_2 \leq \rp+\ry} (q^{-\alpha_{\tau(i_1)} } - q^{ - \alpha_{\tau(i_2)}} )  \prod_{i=1}^{\ry}  q^{ -\alpha_{\tau(i)} N }  \sum_{ \substack{ f_1,\dots, f_{\rp+\ry} \in \mathbb F_q[u]^+ \\ \prod_{i=1}^{\ry} f_i = \prod_{i=\ry+1}^{\rp+\ry} f_i}}  \prod_{i=1}^{\ry}\abs{f_i} ^{ -\frac{1}{2}+ \alpha_{\tau(i)} } \prod_{i=\ry+1}^{\rp+\ry}  \abs{f_i} ^{ -\frac{1}{2}-\alpha_{\tau(i)}  } \end{equation}\begin{equation}\label{third-mt-right} =  \sum_{\substack{ e_1,\dots e_{\rp+\ry} \in \mathbb Z \\ N \geq e_1\geq  \dots e_{\ry} \\ e_{\ry+1} \geq \dots  e_{\rp+\ry} \geq 0   }}\sum_{ \sigma \in S_{\ry} \times S_{\rp} } \sgn(\sigma) \prod_{i=1}^{\rp+\ry} q^{ - (e_i +\rp+\ry- i)  \alpha_{\tau( \sigma(i) )} }   \int_{\Cqs } \psi_{\mathbf e} \mu_{\textrm{ep}}.\end{equation}
We have
\[ \prod_{i=1}^{\ry}  q^{ -\alpha_{\tau(i)} N }  =   \prod_{i\in S}  q^{ - \alpha_iN }\] and 
\[  \sum_{ \substack{ f_1,\dots, f_{\rp+\ry} \in \mathbb F_q[u]^+ \\ \prod_{i=1}^{\ry} f_i = \prod_{i=\ry+1}^{\rp+\ry} f_i}}  \prod_{i=1}^{\ry}\abs{f_i} ^{ -\frac{1}{2}-+ \alpha_{\tau(i)} } \prod_{i=\ry+1}^{\rp+\ry}  \abs{f_i} ^{ -\frac{1}{2}-\alpha_{\tau(i)}  } = \sum_{ \substack{ f_1,\dots, f_{\rp+\ry} \in \mathbb F_q[u]^+ \\ \prod_{i\in S } f_i = \prod_{i\notin S} f_i}}  \prod_{i\in S} \abs{f_i} ^{ -\frac{1}{2}+\alpha_i  } \prod_{i\notin S}   \abs{f_i} ^{ -\frac{1}{2}-\alpha_i  } \] 
using the change of variables $f_i \mapsto f_{\tau(i)}$, and we have
\[ \prod_{1 \leq i_1 < i_2 \leq \rp+\ry} (q^{-\alpha_{\tau(i_1)} } - q^{ - \alpha_{\tau(i_2)}} )  = \sgn(\tau)  \prod_{1 \leq i_1 < i_2 \leq \rp+\ry} (q^{-\alpha_{i_1} } - q^{ - \alpha_{i_2}} ) \] 
so \eqref{third-mt-left} is equal to $\sgn(\tau)$ times \eqref{second-mt-left}. Similarly, the change of variables $\sigma \to \tau^{-1} \sigma$ gives
\[ \sum_{ \sigma \in S_{\ry} \times S_{\rp} } \sgn(\sigma) \prod_{i=1}^{\rp+\ry} q^{ - (e_i +\rp+\ry- i)  \alpha_{\tau( \sigma(i) ) }}= \sum_{ \tau^{-1} \sigma \in S_{\ry} \times S_{\rp} } \sgn(\tau^{-1}\sigma) \prod_{i=1}^{\rp+\ry} q^{ - (d_i +\rp+\ry- i)  \alpha_{ \sigma(i) } }\] \[= \sgn(\tau) \sum_{ \substack{  \sigma \in S_{\rp+\ry} \\ \sigma^{-1} (S) = \{1,\dots, \ry\} } } \sgn(\sigma) \prod_{i=1}^{\rp+\ry} q^{ - (e_i +\rp+\ry- i)  \alpha_{\sigma(i) }}\]
so \eqref{third-mt-right} is equal to $\sgn(\tau)$ times \eqref{second-mt-right}. Thus \eqref{third-mt-left} and \eqref{third-mt-right} are equal to each other. \end{proof}

\begin{lemma}\label{mt-is-extended} \[    \sum_{\substack{ e_1,\dots e_{\rp+\ry} \in \mathbb Z \\ N \geq e_1\geq  \dots e_{\ry} \\ e_{\ry+1} \geq \dots  e_{\rp+\ry} \geq 0   }}\kappa_{\mathbf e} \int_{\Cqs } \psi_{\mathbf e} \mu_{\textrm{ep}} =  \operatorname{MT}_{N}^{\rp,\ry} .\] \end{lemma}

\begin{proof} Since the sum $  \sum_{\substack{ e_1,\dots e_{\rp+\ry} \in \mathbb Z \\ N \geq e_1\geq  \dots e_{\ry} \\ e_{\ry+1} \geq \dots  e_{\rp+\ry} \geq 0   }}\kappa_{\mathbf e} \int_{\Cqs } \psi_{\mathbf e} \mu_{\textrm{ep}}$ is uniformly convergent as a function of $\alpha_1,\dots, \alpha_{\rp+\ry}$ on the imaginary axis, both sides are continuous functions of $\alpha_1,\dots, \alpha_{\rp+\ry}$. So it suffices to prove this identity after restricting to  a dense subset, and therefore suffices to prove it after multiplying by the Vandermonde $ \prod_{1 \leq i_1 < i_2 \leq \rp+\ry} (q^{-\alpha_{i_1} } - q^{ - \alpha_{i_2}} )$.  When we do this, the right-hand side becomes, by definition of $\operatorname{MT} $,
\[ \prod_{1 \leq i_1 < i_2 \leq \rp+\ry} (q^{-\alpha_{i_1} } - q^{ - \alpha_{i_2}})\prod_{i=\rp+ 1}^{\rp+\ry} q^{  \alpha_i N}  \sum_{\substack{  S \subseteq \{1,\dots,\rp+\ry\} \\ \abs{S}=\ry}}  \prod_{i \in S} q^{ -\alpha_i N} \sum_{ \substack{ f_1,\dots, f_{\rp+\ry} \in \mathbb F_q[u]^+ \\ \prod_{i \in S} f_i = \prod_{i\notin S} f_i}}  \prod_{i\in S}\abs{f_i} ^{ -\frac{1}{2}+\alpha_i  } \prod_{i \notin S} \abs{f_i} ^{ -\frac{1}{2}-\alpha_i  } \]
which is the sum over $S$ of \eqref{second-mt-left} multiplied by $\prod_{i=\rp+1}^{\rp+\ry} q^{  \alpha_i N}  $. The left-hand side becomes, by definition of $\kappa$,
\[  \sum_{\substack{ e_1,\dots e_{\rp+\ry} \in \mathbb Z \\ N \geq e_1\geq  \dots e_{\ry} \\ e_{\ry+1} \geq \dots  e_{\rp+\ry} \geq 0   }}  (-1)^{N \ry} \sum_{ \sigma \in S_{\rp+\ry}} \sgn(\sigma) \prod_{i=1}^{\rp+\ry} q^{ - (e_i +\rp+\ry- i)  \alpha_{\sigma(i) }   } \prod_{i=\rp+1}^{\rp+\ry} q^{   N\alpha_i  }  \int_{\Cqs } \psi_{\mathbf e} \mu_{\textrm{ep}} \] 
\[ = \prod_{i=\rp+1}^{\rp+\ry} q^{   N\alpha_i  }  \sum_{\substack{  S \subseteq \{1,\dots,\rp+\ry\} \\ \abs{S}=\ry}} \sum_{\substack{ e_1,\dots e_{\rp+\ry} \in \mathbb Z \\ N \geq e_1\geq  \dots e_{\ry} \\ e_{\ry+1} \geq \dots  e_{\rp+\ry} \geq 0   }}  (-1)^{N \ry} \sum_{\substack{ \sigma \in S_{\rp+\ry}\\ \sigma^{-1} (S) = \{1,\dots,\ry\}}} \sgn(\sigma) \prod_{i=1}^{\rp+\ry} q^{ - (e_i +\rp+\ry- i)  \alpha_{\sigma(i) }   }   \int_{\Cqs } \psi_{\mathbf e} \mu_{\textrm{ep}} \]
which is the sum over $S$ of \eqref{second-mt-right} multiplied by $\prod_{i=\rp+1}^{\rp+\ry} q^{  \alpha_i N}  $, since each permutation $\sigma$ sends $\{1,\dots,\ry\}$ to exactly one set $S$ of cardinality $\ry$. (The rearrangement of the sum is justified by absolute convergence.) The claim then follows from \cref{second-mt}.\end{proof}

\begin{proof}[Proof of \cref{intro-cfkrs}] This follows from combining \cref{lem-mip} and \eqref{int-phi-formula-adjusted} with \cref{tail-bound} and \cref{mt-is-extended}, noting the error term $O ( N^{ O(1)} q^{ - \frac{k}{4}} )$ of \cref{tail-bound} is easily absorbed into the error term $O (N^ { \frac{ (\rp+\ry)^2}{2} - \beta \frac{q-2}{4}}) $ of \cref{lem-mip} since $k = \lfloor N^\beta \rfloor $ so the exponential term  $q^{ - \frac{k}{4}} $ dominates any polynomial in $N$. \end{proof}

\bibliographystyle{alpha}
\bibliography{references.bib}

\end{document}